\documentclass[a4paper,12pt]{article}
\usepackage{amsmath}
\usepackage{amsfonts}
\usepackage{appendix}
\usepackage{amsthm}
\usepackage{amssymb}
\usepackage{enumitem}
\usepackage{yhmath}
\usepackage{fullpage}
\usepackage[all]{xy}
\usepackage{marginnote}
\usepackage{pgf,tikz,pgfplots}
\pgfplotsset{compat=1.15}
\usepackage{mathrsfs}
\usetikzlibrary{arrows}
\numberwithin{equation}{section}
\usepackage{graphicx}
\usepackage{caption}
\usepackage{subcaption}
\usepackage[french,english]{babel}

\usepackage[colorlinks, linkcolor=blue, citecolor=blue, urlcolor=blue]{hyperref}

\newcommand{\Conformal}{G(T_\Omega)}
\newcommand{\Bergmann}[1]{H^2_{#1}(T_\Omega)}
\newcommand{\Tr}{\operatorname{Tr}}
\newcommand{\deriv}[2]{\frac{d^{#1}}{d#2^{#1}}}
\newcommand{\derpar}[2]{\frac{\partial^{#1}}{\partial #2^{#1}}}
\newcommand{\g}{\mathfrak{g}}
\newcommand{\n}{\mathfrak{n}}

\newcommand{\C}{\mathbb{C}}
\newcommand{\R}{\mathbb{R}}

\newcommand{\N}{\mathbb{N}}
\newcommand{\tr}{\operatorname{tr}}
\newtheorem{theorem}{Theorem}[section]
\newtheorem{lemme}[theorem]{Lemma}
\newtheorem{coro}[theorem]{Corollary}
\newtheorem{prop}[theorem]{Proposition}

\newtheorem{defin}[theorem]{Definition}
\newtheorem{example}[theorem]{Example}

\newcommand{\unitball}[1]{B^{#1}}
\newcommand{\symm}[1]{Sym_{#1}(\R)}
\newcommand{\symmdef}[1]{Sym_{#1}^{++}(\R)}

\newcommand{\Hom}{\operatorname{Hom}}
\newcommand{\Hi}[2]{H^2(T_{\Omega_{#1}})_{#2}}
\newcommand{\Juhl}{D_{\lambda \to \nu}}
\newcommand{\Juhladj}{\left(\Juhl\right)^*}
\newcommand{\confL}[2]{L^2_{#1}(\Omega_{#2})}
\newcommand{\derpone}[2]{\frac{\partial #1}{\partial {#2}}}
\newcommand{\derptwo}[2]{\frac{\partial^2 #1}{\partial {#2}^2}}
\newcommand{\derptwobis}[3]{\frac{\partial^2 #1}{\partial #2 \partial #3}}

\title{A geometrical point of view for branching problems for holomorphic discrete series of conformal Lie groups}
\author{Quentin Labriet}
\date{\today}
\begin{document}
\maketitle

\begin{abstract}
This article is devoted to branching problems for holomorphic discrete series representations of a conformal group $G$ of a tube domain $T_\Omega$ over a symmetric cone $\Omega$. More precisely, we analyse restrictions of such representations to the conformal group $G'$ of a tube domain $T_{\Omega'}$ holomorphically embedded in $T_\Omega$. The goal of this work is the explicit construction of the symmetry breaking and holographic operators in this geometrical setting. To do so, a stratification space for a symmetric cone is introduced. This structure put the light on a new functional model, called the \emph{stratified model}, for such infinite dimensional representations. 

The main idea of this work is to give a geometrical interpretation for the branching laws of infinite dimensional representations. The stratified model answers this program by relating branching laws of holomorphic discrete series representations to the theory of orthogonal polynomials on the stratification space. This program is developed in three cases. First, we consider the $n$-fold tensor product of holomorphic discrete series of the universal covering of $SL_2(\R)$. Then, it is tested on the restrictions of a member of the scalar-valued holomorphic discrete series of the conformal group $SO(2,n)$ to the subgroup $SO(2,n-p)$, and finally to the subgroup $SO(2,n-p)\times SO(p)$.
\end{abstract}

\section{Introduction}
\newtheorem{theoremIntro}{Theorem}
\newtheorem{propIntro}[theoremIntro]{Proposition}
\newtheorem{exampleIntro}{Example}
\setcounter{equation}{0}
\renewcommand\thetheoremIntro{\Alph{theoremIntro}} 

\paragraph{Branching problems for unitary representations}$\ $

In representation theory, \emph{branching problems} ask how a given irreducible representation $\pi$ of a group $G$ behaves when restricted to a subgroup $G'\subset G$, and by \emph{branching law} we mean the decomposition into irreducible components of the representation $\pi$ of $G$ restricted to $G'$. As a special case of branching laws we mention the decomposition of the tensor product representation $\pi\otimes\rho$ of two irreducible representations of $G$ restricted to the diagonal subgroup $\Delta(G)\simeq G\subset G\times G$. 

In this work, we are interested in the branching problems of some infinite dimensional irreducible unitary representations $\pi$ of real reductive Lie groups $G$ in the sense of \cite{Wall} once restricted to some closed subgroup $G'$. An application of a fundamental Theorem of Mautner \cite{Maut} gives an answer in this setting in terms of a direct integral decomposition of underlying Hilbert spaces:
\begin{equation}\label{eq:BranchingMautner}
\pi|_{G'}\simeq \int^\oplus_{\widehat{G'}} m_\pi(\rho)\rho~d\mu(\rho),
\end{equation}
where $\widehat{G'}$ denotes the unitary dual of $G'$ and $d\mu$ is an appropriate Borel measure on $\widehat{G'}$. The function $m_\pi(\rho):\widehat{G'}\to \N\cup \{ \infty\}$ is called the \emph{multiplicity} of the representation $\rho$ in the irreducible decomposition of $\pi$. It is defined almost everywhere on $\widehat{G'}$ with respect to the measure $d\mu$. 

T. Kobayashi has recently proposed \cite{Kob15} a program called (ABC)-program that divides the branching problems into the three following steps:
\begin{itemize}
\item \textbf{A} - Abstract features of the restriction $\pi|_{G'}$.
\item \textbf{B} - Branching laws.
\item \textbf{C} - Construction of \emph{symmetry breaking operators},
\end{itemize}
where a symmetry breaking operator is an element of the space of continuous intertwining operators $\Hom_{G'}(\pi,\rho)$, for given $\pi \in \widehat{G}$ and $\rho \in \widehat{G'}$. 

At the stage A, one expect to have some estimates about the multiplicities (infinite multiplicities, multiplicity-free decomposition...), or about the spectrum (discrete decomposition, non-trivial continuous part...). As explained in \cite{Kob15}, theorems in stage A might suggest different approaches for stages B and C. 

We should emphasize that the (ABC)-program mentioned above does not give an ordering in the study of branching problems. For example, an answer to the stage B gives a direct answer to the stage A. Similarly, a complete answer at the stage C also answers the question raised at the stage B.  

For more about branching problems for unitary representations of real reductive Lie groups we refer to \cite{Kobayashi02},\cite{Kobayashi05} and references therein.

\paragraph{Symmetry breaking and holographic operators}$\ $

This work focuses on the stage C of the (ABC) program. The goal at the stage C is to construct an explicit basis for the space of symmetry breaking operators $\Hom_{G'}(\pi,\rho)$ where $\pi$ is an irreducible unitary representation of $G$ and $\rho$ an irreducible representation of $G'$ appearing in the decomposition of $\pi|_{G'}$. We denote by $V_\pi$ the underlying vector space of the representation $\pi$.

In their two articles \cite{KP_SBO_1} and \cite{KP_SBO}, T. Kobayashi and M. Pevzner introduced a new method for an explicit construction of symmetry breaking operators for infinite-dimensional representations. More precisely, they consider a symmetric pair $(G,G')$ of real reductive Lie groups of the Hermitian type such that the symmetric space $G'/K'$ is a complex submanifold holomorphically embedded in the symmetric space $G/K$, where $K$ (resp. $K'$) denotes a maximal compact subgroup in $G$ (resp. $G'$), and they consider the restriction to $G'$ of a member of the holomorphic discrete series of $G$. This method, called the F-method, is based on an algebraic Fourier transform of generalized Verma modules, and it relates a symmetry breaking operator to some polynomial solution to a system of partial differential equations. As an example, the authors explicitly describe all symmetry breaking operators for six different geometries and find that, for three of them, symmetry breaking operators are related to some classical orthogonal polynomials (Jacobi or Gegenbauer polynomials). For the last three geometries however, symmetry breaking operators are given by normal derivatives. The study of symmetry breaking operators is an active field of research and a considerable progress was made in this area in the recent years \cite{Juhl09}, \cite{Ibukiyama12},\cite{Ibukiyama20} \cite{KobayashiSpeh15},\cite{KobayashiKuboPevzner14}, \cite{KobayashiKuboPevzner16},  \cite{KobayashiSpeh18}, \cite{BensaidClercKoufany20},\cite{MckeePasqualePrzebinda21}, etc.

Considering the notion dual to the one of symmetry breaking operators leads to the definition of \emph{holographic operators} as elements of the space $\Hom_{G'}(\rho,\pi)$ of continuous $G'$-intertwining operators between $\rho$ and $\pi$. In the unitary situation, there is a bijection between $\Hom_{G'}(\rho,\pi)$ and $\Hom_{G'}(\pi,\rho)$ given by
\begin{equation}
T\in\Hom_{G'}(\pi,\rho) \mapsto T^*\in \Hom_{G'}(\rho,\pi),
\end{equation}
where $T^*$ denotes the adjoint operator of $T$. Hence, one needs to understand how to compute the adjoint of a given symmetry breaking operator.

Unlike at the stages A and B, one should notice that the construction of symmetry breaking and holographic operators depends on the spaces $V_\pi$ and $V_\rho$ on which the representations are realized.

\paragraph{Branching for holomorphic discrete series representations}$\ $

In this article, we are interested in some aspects of branching problems for holomorphic discrete series representations of some reductive Hermitian groups $G$. The branching laws of such representations to various subgroups has been intensively studied, see for example \cite{JakoVergne79},\cite{Kobayashi94},\cite{Kobayashi98b},\cite{Kobayashi98},\cite{Zhang01},\cite{Zhang02},\cite{Dijk_Pev},\cite{Kobay_08},\cite{DufloGalinaVargas17},\cite{OrstedVargas20}.

More precisely, we are going to consider a scalar-valued holomorphic discrete series  representation $\pi_\lambda$ of the conformal group $G=G(T_{\Omega_2})$ of a tube domain over a symmetric cone $\Omega_2$, restricted to the conformal group $G'=G(T_{\Omega_1})$ of a tube domain over a symmetric cone $\Omega_1$ such that $T_{\Omega_1}$ is "naturally" embedded in $T_{\Omega_2}$ (see section \ref{sec:defintionDiffeo} for a precise setting). Constructions of symmetry breaking operators in similar settings were developed in \cite{Ibukiyama99},\cite{Ibukiyama12},\cite{Nakahama19},\cite{JL.Clerc},\cite{Clerc2},\cite{Nakahama21}.

As we already mentioned, the choice of the model for the scalar-valued holomorphic discrete series representation $\pi_\lambda$ is important for actual computations of symmetry breaking and holographic operators. A first model of interest, called the \emph{holomorphic model}, is given by the action of the group $G(T_{\Omega_2})$ on the  weighted Bergman space $\mathcal{H}^2_\lambda (T_{\Omega_2})$ over the tube domain $T_{\Omega_2}$. A second one, called the \emph{$L^2$-model}, is given by the action of $G(T_{\Omega_2})$ on the $L^2$ space on the symmetric cone $\Omega_2$ with respect to the measure $\Delta_2(x)^{\lambda-\frac{n}{r}}~dx$, denoted $L^2_\lambda(\Omega_2)$, and they are related through the Laplace transform \eqref{def:LaplaceTransformGeneral} on the underlying Jordan algebra. In this work, we define yet another model, called the \emph{stratified model}, induced by a change of coordinates on the cone $\Omega_2$, in which the branching problem seems to be related to the construction of a basis of orthogonal polynomials as suggested in \cite{KP}.

An example of this situation is given by the branching of the tensor product of two holomoprhic discrete series representations of the universal cover of $SL_2(\R)$. Here, $\Omega_1=\R^+$, $T_{\Omega_1}$ is the Poincaré upper half-plane $\Pi$, and $\Omega_2=\R^+\times\R^+$, $T_{\Omega_2}=\Pi\times\Pi$. In the holomorphic model, the symmetry breaking operators correspond to the so-called Rankin--Cohen brackets \cite{Dijk_Pev}
\begin{equation}
RC_{\lambda',\lambda''}^{\lambda'''}(f)(z)=\sum_{j=0}^l \frac{(-1)^j(\lambda'+l-j)_j(\lambda''+j)_{l-j}}{j!(l-j)!}\frac{\partial^lf}{\partial z_1^{l-j} \partial z_2^j}(z,z).
\end{equation}
The symbol of this differential operator can be expressed using the classical Jacobi polynomials. One way two explain this fact was through the F-method which relates the symbol of the Rankin--Cohen brackets to polynomial solutions of the Gauss hypergeometric ordinary differential equation. However these polynomials, namely the Jacobi polynomials, are known to be orthogonal on $L^2((-1;1),(1-v)^\alpha(1+v)^\beta~dv)$, hence the question of understanding where this property plays a role in this branching problem is raised naturally.

Following \cite{KP}, the results in \cite{Labriet20} put the light on yet another useful model for the tensor product representation of $SL_2(\R)$, which we called the \emph{stratified model}, in which the orthogonality of the Jacobi polynomials plays a major role. Indeed, it is shown (see \cite{Labriet20}, Proposition 4.2) that the classical Jacobi transform can be interpreted as a symmetry breaking operator in this model as already noticed in \cite{KP}. The key ingredient of the definition of the stratified model is the diffeomorphism $\iota :\R^+\times (-1;1)\to \R^+\times \R^+$ defined by 
\begin{equation}
\iota(t,v)=\left(\frac{t(1-v)}{2},\frac{t(1+v)}{2}\right).
\end{equation}

In the present work, we extend this diffeomorphism $\iota$ to the case of symmetric cones $\Omega_1\subset \Omega_2$. More precisely, consider a Euclidean unital Jordan subalgebra $V_1$ of a given Euclidean algebra $V_2$, and let $\Omega_1$ and $\Omega_2$ be their respective symmetric cones. We define the \emph{stratification space}
\begin{equation}
X:=\{v\in V_1^\bot~|~e+v\in\Omega_2\},
\end{equation}
where $e$ denotes the identity in $V_2$ and $V_1$, and $V_1^\bot$ denotes the orthogonal complement of $V_1$ in $V_2$. Then, we introduce the following diffeomorphism $\iota :\Omega_1\times X \to \Omega_2$ defined by:
\begin{equation}
\iota(t,v):=\frac{r_1}{r_2} P_2(t^\frac{1}{2})(e+v),
\end{equation}
where $r_i$ is the rank of the Jordan algebra $V_i$, $i=1,2$, and $P_2$ denotes the quadratic representation \eqref{eq:QuadraticRepresentation} of the Jordan algebra $V_2$. In the case of a symetric pair $(G(T_{\Omega_2}),G(T_{\Omega_1}))$, the stratification space $X$ is a symmetric space as described in \cite{Labriet20}.

We consider two different cases in which this diffeomorphism gives promising results for studying branching problems:
\begin{itemize}
\item In the first case $V_2$ is supposed to be simple and $V_1$ and $V_2$ have the same rank, $r_1=r_2$. 
\item The product case $V_2=\overbrace{V\times \cdots \times V}^p$ and $V_1=V$ is a simple Euclidean Jordan algebra, hence $r_2=p r_1$. The $p=2$ case was already considered in \cite{JL.Clerc}. In this setting, the goal is to study the decomposition of the $p$-fold tensor product of holomorphic discrete series  representations of the conformal group $G(T_\Omega)$ where $T_\Omega$ is the tube domain associated to $V$.
\end{itemize}

In both cases, the pullback of the the $L^2$ model by the diffeomorphism $\iota$ induces the stratified model (see Proposition \ref{prop:Hilbert_space_isom_EqualRank} and \ref{prop:Hilbert_space_isom_tensorProduct}). In this stratified model the scalar-valued holomorphic discrete series representation $\pi_\lambda$ acts on the Hilbert space:
\begin{equation}
L^2_\lambda(\Omega_1)\hat{\otimes} L^2(X,d\mu),
\end{equation}
where $d\mu$ is a finite measure on $X$. 

The philosophy of this work is that branching problems should be related to the theory of orthogonal polynomials on the space $L^2(X,d\mu)$. Indeed, if one introduces the space $Pol_k(X)$ of polynomials of degree $k$ which are orthogonal with respect to the measure $d\mu$ with all polynomials of smaller degree, then subspaces of the form
\begin{equation}
L^2_\lambda(\Omega_1)\hat{\otimes} Pol_k(X),
\end{equation}
are stable under the action of the subgroup $G_{\Omega_1}$ and the subgroup $N_1$ of translations by elements in $V_1$. Furthermore, if $K_1$ denotes a maximal compact subgroup of $G_{\Omega_1}$, then there is a natural action of $K_1$ on the space $X$ which induces a representation of $K_1$ on $Pol_k(X)$. If it decomposes into irreducibles as $Pol_k(X)\simeq \bigoplus W_{j,k}$ then the spaces $L^2_\lambda(\Omega_1)\hat{\otimes} W_{j,k}$ remain irreducible for $G_{\Omega_1}$ and $N_1$.

It is known \cite{Far_Kor} that the group $G(T_{\Omega_1})$ is generated by $G_{\Omega_1}$, $N_1$ and an element $j$ called the inversion. Thus we reduce our study to the analysis of the action of $j$ on $L^2_\lambda(\Omega_1)\hat{\otimes} Pol_k(X)$. According to this strategy we investigated three important cases. Namely, the $n$-fold tensor product of holomorphic discrete series representations of $SL_2(\R)$, the restriction of a scalar valued holomorphic discrete series representation $\pi_\lambda^{(n)}$ of $SO_0(2,n)$ to $SO_0(2,n-p)$ and finally the restriction of $\pi_\lambda^{(n)}$ to $SO_0(2,n-p)\times SO(p)$, where $n\geq 2$ and $p<n$.

In the case of the $n$-fold tensor product of holomorphic discrete series representations of $SL_2(\R)$. The stratification space $X$ can be identified with the $(n-1)$-dimensional simplex $D_{n-1}$. Using our results on the usual tensor product case, we were able to prove a one-to-one correspondence between the space of symmetry breaking operators and the space $Pol_k(X)$. The correspondence is given, up to some shift by a power of the Jordan determinant, by the orthogonal projection on the space generated by a polynomial $P\in Pol_k(X)$. This is proved by recursion on $n$ from the $n=2$ case. More precisely, we proved the following (see section \ref{sec:ntensorPorduct} for precise notations)

\begin{theoremIntro}
Let $k\in \N$ and $Pol_k(D_{n-1})$ be the space of polynomials in $n-1$ variables of degree $k$ which are orthogonal to all polynomials of smaller degree with respect to the inner product on $L^2(D_{n-1},(1-|v|)^{\lambda_n-1}\prod_{i=1}^{n-1}v_i^{\lambda_i-1}~dv)$. 

For a polynomial $P\in Pol(D_{n-1})$ define, on $L^2(\R^+,t^{|\Lambda|-1}dt)\hat{\otimes}L^2(D_{n-1},(1-|v|)^{\lambda_n-1}\linebreak \prod_{i=1}^{n-1}v_i^{\lambda_i-1}~dv)$, the operator $\Psi_k^\Lambda(P)$ by
\begin{equation}
\Psi_k^\Lambda(P)f(t)=t^{-k}\int_{D_{n-1}}f(t,v)P(v) (1-|v|)^{\lambda_n-1}\prod_{i=1}^{n-1}v_i^{\lambda_i-1}~dv.
\end{equation}

Then $\Psi_k^\Lambda(P)\in Hom_{SL_2(\R)}(\bigotimes_{i=1}^n \pi_{\lambda_i} ,\pi_{|\Lambda|+2k})$ if and only if $P\in Pol_k(D_{n-1})$.
\end{theoremIntro}

In the second case, we considered the restriction to $SO_0(2,n-p)$ of a scalar-valued holomorphic discrete series representation $\pi_\lambda^{(n)}$ of $SO_0(2,n)$, for $n\geq 3$ and $n-p\geq 2$. Notice that the case $p=1$ was already explored in \cite{KP}. Here, the stratification space $X$ is identified with the $p$-dimensional unit ball $\unitball{p}$. Once again there is a one-to-one correspondence between the space of symmetry breaking operators and $Pol_k(X)$, and the correspondence is also given, up to some shift by a power of the Jordan determinant, by the orthogonal projection on a space generated by a $p$ variables polynomial $P\in Pol_k(X)$. In this example we give two different proofs of the result: one using recursion on $p$ from the $p=1$ case, and another one using the derived representation of the corresponding Lie algebra and the Bessel operators. This correspondence is given explicitly by the following Theorem (see section \ref{sec:SO(2,n)SO(2,n-p)} for precise notations)

\begin{theoremIntro}
Let $\lambda>n-1$, and $Pol_k(\unitball{p})$ be the space of polynomials in $p$ variables of degree $k$ which are orthogonal to all polynomials of smaller degree with respect to the inner product on $L^2(\unitball{p},d\mu_\alpha(v))$.

For a polynomial $P\in Pol(\unitball{p})$, define, on $L^2_\lambda(\Omega_{n-p})\hat{\otimes} L^2(\unitball{p}, d\mu_\lambda(v))$, the operator $\Phi_k^\alpha(P)$ by 
\begin{equation}
\Phi_k^\alpha(P)f(x)=Q_{1,n-p-1}(x)^{-\frac{k}{2}}\int_{\unitball{p}}f(x,v)P(v)~d\mu_\alpha(v).
\end{equation}
Then $\Phi_k^\alpha \in \Hom_{SO_0(2,n-p)}(\pi^{(n)}_\lambda,\pi^{(n-p)}_{\lambda+k})$ if and only if $P\in Pol_k(\unitball{p})$.
\end{theoremIntro}

In both cases, one important feature is that the $K_1$ action is trivial on $X$, and hence on $Pol(X)$. This is related to the fact that both branching laws involve only scalar-valued discrete series representations whereas in general one may have vector-valued discrete series as well. However, these two examples allowed us to give a nice description of the space of the symmetry breaking operators even if the multiplicities are not uniformly bounded.

The last case explores the restriction to $SO_0(2,n-p)\times SO(p)$ of a scalar-valued holomorphic discrete series representation $\pi_\lambda^{(n)}$ of $SO_0(2,n)$. In this setting, we use the same diffeomorphism as in the previous example but the $SO(p)$ action on the unit ball $\unitball{p}$ implies that the $K_1$ action on $Pol_k(X)$ is not trivial anymore. In this case, $Pol_k(X)$ decomposes as a $K_1$-module:
\begin{equation}
Pol_k(X)=\bigoplus_{j=1}^{\lfloor \frac{k}{2}\rfloor} W_{j,k}\simeq \bigoplus_{j=1}^{\lfloor \frac{k}{2}\rfloor} \mathcal{H}^p_{k-2j},
\end{equation}
where $\mathcal{H}^p_{k-2j}$ denotes the space of harmonic polynomials in $p$ variables of degree $k-2j$. It turns out that the spaces $L^2_\lambda(\Omega_1)\hat{\otimes} \mathcal{H}_{k-2j}$ are actually stable under the action of the whole group $SO_0(2,n-p)\times SO(p)$. The space of symmetry breaking operators is one dimensional in this situation and its generator is given, up to some shift by a power of the Jordan determinant, by the orthogonal projection on $W_{j,k}$. Finally, we proved the following Theorem (see section \ref{sec:SO(2,n)SO(2,n-p)SO(p)} for precise notations)
\begin{theoremIntro}
Let $\lambda>n-1$, then the restriction to $SO_0(2,n-p)\times SO(p)$ of the scalar valued holomorphic discrete series representation $\pi_\lambda^{(n)}$ of $SO_0(2,n)$ is given by
\begin{equation}
\left.\pi_\lambda^{(n)}\right|_{G_1}\simeq {\sum_{k\in \N}}^\oplus\bigoplus_{j=0}^{\lfloor \frac{k}{2}\rfloor}\pi^{(n-1)}_{\lambda+k}\otimes \eta_{k-2j}.
\end{equation}
\end{theoremIntro}

\paragraph{Organization of the article}$\ $

The first section fixes notations, and presents the necessary background about Euclidean Jordan algebras, symmetric cones and scalar-valued holomorphic discrete series. The second one presents the geometrical setting and the diffeomorphism used to introduce the stratified model.

The last two sections are devoted to applications of these ideas to different branching problems. First, we study the $n$-fold tensor product of holomorphic discrete series representations of $SL_2(\R)$. In a second part, we study the restriction of a scalar valued holomorphic discrete series representation of $SO_0(2,n)$ to $SO_0(2,n-p)$ and then to $SO_0(2,n-p)\times SO(p)$.

\section{Preliminaries}\label{chap:JordanAlgebras}
The aim of this section is to set up the notation and to recall the construction of two different models for the scalar-valued holomorphic discrete series representations for Hermitian Lie groups of tube type. 
One may describe these models using the theory of Euclidean Jordan algebras and of symmetric cones. We essentially follow \cite{Far_Kor} and refer to this book for more details.

\subsection{Euclidean Jordan algebras}\label{sec:JordanAlgebra}

In our context, we fix the ground field $\mathbb{K}$ to be $\R$ or $\C$. 
\begin{defin}
A \emph{Jordan algebra} $V$ is a vector space over a field $\mathbb{K}$ together with a bilinear map $(x,y)\mapsto xy$, and a unit element $e$ such that, for all $x,y\in V$:
\begin{align}
xy&=yx,\tag{J1}\\
x(x^2y)&=x^2(xy).\tag{J2}
\end{align}
A Jordan algebra $V$ is called \emph{simple} if it has no non-trivial ideal.
\end{defin}
We denote by $n$ the dimension of the vector space $V$, and by $L(x)$ the endomorphism of $V$ defined by:
\begin{equation}
L(x)y:=x y.
\end{equation}

To an element $x\in V$, one can associate the \emph{minimal polynomial} $f_x$ defined as the monic polynomial generating the ideal $J_x=\{P\in\mathbb{K}[X]~|~P(x)=0\}$. This allows to define the rank $r$ of $V$ by 
\begin{equation}
r:=\max \{\operatorname{deg}(f_x)~|~x\in V\},
\end{equation}
and an element $x\in V$ is called \emph{regular} if $\operatorname{deg}(f_x)=r$. 
The set of regular elements is open and dense in $V$ and for a regular element $x\in V$ we have
\begin{equation}
f_x(\lambda)=\lambda^r-a_1(x)\lambda^{r-1}+\cdots +(-1)^r a_r(x),
\end{equation}
where $a_j(x)$ is a homogeneous polynomials of degree $j$ (see \cite{Far_Kor}, Prop.II.2.1). Using this minimal polynomial, we define the \emph{trace} $\tr$ and the \emph{determinant} $\Delta$ on a Jordan algebra. Let $x\in V$ be a regular element, we set:
\begin{equation}\label{def:JordanTraceDeterminant}
\tr(x):=a_1(x),~~~~~~\Delta(x):=a_r(x),
\end{equation}
and we extend this definition to any element $x\in V$ by density. We use the notation $\det$ and $\Tr$ for the usual determinant and trace of an endomorphism of a vector space. Finally, notice that 
\begin{equation}
\tr(e)=r,~~~~~~\Delta(e)=1.
\end{equation}

For convenience, we set:
\begin{equation}
m:=\frac{n}{r}.
\end{equation}

The quadratic representation $P$ of $V$ is defined, for $x,y\in V$, by
\begin{equation}\label{eq:QuadraticRepresentation}
P(x)y:=2x(xy)-x^2y.
\end{equation}
Notice that $P(x)=2L(x)^2-L(x^2)$.
The quadratic representation $P$ admits a polarization, again denoted $P$, defined by
\begin{equation}\label{eq:QuadraticRepresentationPolarization}
P(x,y):=\frac{1}{2}(P(x+y)-P(x)-P(y)),
\end{equation}
or equivalently:
\begin{equation}
P(x,y)=L(x)L(y)+L(y)L(x)-L(xy).
\end{equation}
We also define the linear map $x\square y$, for $x,y\in V$, by 
\begin{equation}
(x\square y)z:=(xy)z+x(yz)-y(xz),
\end{equation}
or equivalently by $x\square y=L(xy)+[L(x),L(y)]$. Let $V\square V$ be the vector space of endomorphisms generated by the elements $x\square y$, $x,y\in V$.

\begin{defin}
A Jordan algebra $V$ over the field $\R$ is called \emph{Euclidean} if there is an inner product $(\cdot|\cdot)$ on $V$ which is associative, i.e. such that for all $x,y,z\in V$
\begin{equation}
(xy|z)=(y|xz).
\end{equation}
\end{defin}

In a Euclidean Jordan algebra the bilinear form given by 
\begin{equation}\label{eq:JordanInnerProduct}
(x|y)=\tr(xy).
\end{equation}
is an associative inner product, called the \emph{trace form}, and it is known that all associative inner products are proportional to the trace form if $V$ is a simple Jordan algebra (see \cite{Far_Kor} Prop.III.4.1). In the following, we always suppose that Euclidean Jordan algebras are endowed with the inner product \eqref{eq:JordanInnerProduct}  unless otherwise is mentioned. For a linear map $A$ on $V$ we denote by $A'$ its adjoint with respect to this inner product.

The structure group $Str(V)$ of the Euclidean Jordan algebra $V$ is defined by
\begin{equation}\label{def:StructureGroup}
Str(V):=\{g\in GL(V)~|~\forall x\in V,~P(gx)=gP(x)g'\}.
\end{equation}
Its Lie algebra coincides with the Lie algebra $V\square V$ where the Lie bracket is defined by the usual commutator of endomorphisms (see \cite{Satake}, Chap.1,§7).

The next proposition gives some useful formulas for the quadratic representation in a simple Euclidean Jordan algebra $V$.
\begin{prop}[\cite{Far_Kor}, Prop.III.4.2]\label{prop:FormulesDeterminantQuadratic}
Let $V$ be a simple Euclidean Jordan algebra and $x,y\in V$.
\begin{enumerate}
\item $\det P(x)=\Delta(x)^{2m}$.
\item $\Delta(P(y)x)=\Delta(y)^2\Delta(x)$.
\end{enumerate}
\end{prop}

\begin{example}\label{ex:JordanAlgebra}
\begin{enumerate}
\item On the vector space $\symm{n}$ of symmetric $n\times n$ matrices over $\R$ let us define the bilinear map
\begin{equation}
(x,y)\mapsto x\circ y:=\frac{1}{2}(xy+yx).
\end{equation}
With this product $\symm{n}$ is a simple Euclidean Jordan algebra of rank $n$ with the identity matrix $I_n$ as unit element. The Jordan trace and determinant are the usual trace and determinant of matrices. For $x,y\in \symm{n}$, the quadratic representation is given by:
\begin{equation}
P(x)y=xyx.
\end{equation}
\item \label{ex:JordanAlgebraRank2}Let $W$ be a vector space over $\R$, and $B$ a symmetric bilinear form defined on $W$. The space $V=\R\times W$ equipped with the product
\begin{equation}
((x,u),(y,v))\mapsto (x,u)\cdot (y,v):=(xy+B(u,v),xv+yu),
\end{equation}
is a Jordan algebra. Moreover, if the bilinear form $B$ is positive definite then the bilinear form on $V$ given by
\begin{equation}\label{eq:InnerProductRank2}
((x,v)|(y,u))=xy+B(u,v),
\end{equation}
is an associative inner product on $V$ and hence $V$ is Euclidean.

Finally, if one set $T=L((0,u))$, for $(x,u)\in V$, then we have
\begin{equation}
P((x,u))=\left(x^2-B(u,u)\right)I+2xT+2T^2.
\end{equation}
\end{enumerate}
\end{example}

An \emph{idempotent} in $V$ is an element $c\in V$ such that $c^2=c$. 
\begin{defin}
A family of idempotents $(c_1,\cdots,c_r)$ in $V$ is called a Jordan frame if none of the $c_i$ can be written as a sum of two idempotents, and if the following properties are satisfied:
\begin{equation}
c_ic_j=0 \text{ if }i\neq j,~c_i^2=c_i,\\
\sum_{i=1}^rc_i=e.
\end{equation}
\end{defin}

The notion of Jordan frame allows us to state the spectral theorem on a Euclidean Jordan algebra $V$.
\begin{theorem}[Spectral theorem,\cite{Far_Kor},Thm.III.1.2]\label{thm:spectralTheorem}
For $x\in V$ there exist a Jordan frame $(c_1,\cdots,c_r)$ and real numbers $(\lambda_1,\cdots,\lambda_r)$ such that 
\begin{equation}
x=\sum_{i=1}^r \lambda_i c_i.
\end{equation}
The real numbers $\lambda_i$ are uniquely determined by $x$, and furthermore
\begin{equation}
\Delta(x)=\prod_{i=1}^r \lambda_i,~~ \tr(x)=\sum_{i=1}^r \lambda_i.
\end{equation}
\end{theorem}

For an idempotent $c$ in a Jordan algebra it is known that its only possible eigenvalues are $1,~0$ and $\frac{1}{2}$ (\cite{Far_Kor},Prop.III.1.3). We denote $V(c,1)$, $V(c,0)$ and $V(c,\frac{1}{2})$ the corresponding eigenspaces. We then consider the subspaces of $V$:
\begin{equation*}
V_{ii}:=V(c_i,1),~~~V_{ij}:=V(c_i,\frac{1}{2})\cap V(c_j,\frac{1}{2}).
\end{equation*}
If $V$ is a simple Jordan algebra, then all the $V_{ij}$ have the same dimension denoted $d$. Finally, this leads to the following decomposition of the space $V$, called the Pierce decomposition.

\begin{prop}[Pierce decomposition,\cite{Far_Kor},Thm.IV.2.1]
Let $V$ be a Euclidean Jordan algebra, and $(c_1,\cdots,c_r)$ a Jordan frame in $V$, then
\begin{equation}\label{eq:PierceDecomposition}
V=\bigoplus_{1\leq i\leq j\leq r} V_{ij}.
\end{equation}
\end{prop}

Finally, if $V$ is not simple then it is isomorphic to a direct sum $V_1\oplus \cdots \oplus V_k$ where all the $V_i$ are simple ideals in $V$ (\cite{Far_Kor}, Prop.III.4.4). The following Lemma gives some useful formulas when $V$ is not simple.
\begin{lemme}\label{lem:NonSimpleJordanAlgebras}
Let $V$ be a Euclidean Jordan algebra such that $V\simeq V_1\oplus \cdots \oplus V_k$ where all the $V_i$ are all simple ideals in $V$ with rank $r_i$, determinant $\Delta_i$, trace $\tr_i$ and quadratic representation $P_i$. Then, the following formulas hold, for any $x=\sum_i x_i\in V$ such that $x_i\in V_i$:
\begin{itemize}
\item $r=\sum_i r_i$.
\item $\Delta(x)=\prod_i \Delta_i(x_i)$.
\item $\tr(x)=\sum_i \tr_i(x_i)$.
\item $P(x)=P_1(x_1)\oplus \cdots \oplus P_k(x_k)$.
\end{itemize}
\end{lemme}
\begin{proof}
If we choose a Jordan frame for each $V_i$, then the union of all these Jordan frame is a Jordan frame for $V$. The first three statements are then clear from the spectral Theorem \ref{thm:spectralTheorem}. The last one is clear from the fact that if $x=\sum_i x_i\in V$ then $x_jx_k\in V_j\cap V_k=\{0\}$.
\end{proof}

\subsection{Symmetric cones}
Starting from a Euclidean Jordan $V$ algebra one can define a symmetric cone $\Omega$ which carries the structure of a Riemannian symmetric space. 

Let $V$ be a Euclidean vector space whose inner product is denoted $(\cdot|\cdot)$.
\begin{defin}\label{def:SymmetricCone}
Let $\Omega$ be an open convex cone in $V$. Its \emph{dual cone} is defined as
\begin{equation}
\Omega^*:=\{y\in V~|~(x|y)>0,~\forall x\in\overline{\Omega}\backslash \{0\}\}.
\end{equation}
An open convex cone $\Omega$ is called \emph{self-dual} if $\Omega^*=\Omega$.\\
Let  $G_\Omega$ be the group of bijective linear maps of $V$ which preserves $\Omega$, i.e.
\begin{equation}
G_\Omega:=\{g\in GL(V)~|~g\cdot\Omega\subset \Omega\}.
\end{equation} 
A self-dual cone $\Omega$ is called \emph{symmetric} if $G_\Omega$ acts transitively on $\Omega$.
\end{defin}

Suppose now that $V$ is a Euclidean Jordan algebra. To $V$ one can associate a symmetric cone $\Omega$ by considering the interior of the set of quadratic elements $\{x^2~|~x\in V\}$ in $V$. It turns out that this is a one-to-one correspondence between Euclidean Jordan algebras and symmetric cones \cite{Satake}, Chap.1, Prop. 9.2. For the inverse map, see \cite{Satake}, Chap. 1, Theorem 8.5. From now on, we suppose that $V$ is a Euclidean Jordan algebra equipped with the trace form, and $\Omega$ is its associated symmetric cone. Furthermore, notice that the measure $\Delta(t)^{-m}~dt$ is $G_\Omega$-invariant on $\Omega$.

We denote by $G_0$ the connected component of the identity of $G_\Omega$, and by $K_0$ a maximal compact subgroup in $G_0$. It turns out that $K_0$ can be chosen as the stabilizer of the identity element $e$, and with our choice of inner product \eqref{eq:JordanInnerProduct} on $V$ we have (see \cite{Far_Kor} Thm.III.5.1):
\begin{equation}\label{eq:CompactGroupeStructure}
K_0=G_\Omega\cap O(V).
\end{equation}

There is a polar decomposition on $G_\Omega$:
\begin{prop}[\cite{Far_Kor},Thm.III.5.1]
For any $g\in G_\Omega$ there exists a unique $x\in \Omega$ and $k\in K_0$ such that:
\begin{equation}\label{eq:PolarDecomposition}
g=P(x)k.
\end{equation} 
Notice that $x^2=g\cdot e$ and so $k=P(g\cdot e)^{-1}g$.
\end{prop}

\begin{example}\label{ex:SymmetricCone}
\begin{enumerate}
\item For the Jordan algebra $\symm{n}$, the associated symmetric cone is the cone of positive definite symmetric matrices $\symmdef{n}$, and $G_\Omega=GL_n(\R)$. The action of $GL_n(\R)$ on $\symmdef{n}$ is given by
\begin{equation}
g\cdot M=gM g^t.
\end{equation}
\item \label{ex:SymmetricConeRank2}For the Euclidean Jordan algebra $V=\R\times W$ associated to a positive definite bilinear form $B$ on $W$, the associated symmetric cone is the Lorentz cone:
\begin{equation}
\Omega=\{(x,v)\in V~|~x>0, x^2-B(v,v)>0\}.
\end{equation}
Consider the indefinite orthogonal group $SO(1,n)$ and its identity component $SO_0(1,n)$, then $G_\Omega=\R^+\times SO_0(1,n)$ the direct product of positive dilatations by $SO_0(1,n)$.
\end{enumerate}

\end{example}

The group $G_\Omega$ is related to the structure group of $V$ by the formula:
\begin{equation}
Str(V)=G_\Omega\times\{\pm I_n\}.
\end{equation}
In particular, they have the same Lie algebra $V\square V$.

A symmetric cone $\Omega\subset V$ is said to be \emph{irreducible}, if there do not exist non-trivial subspace $V_1$ and $V_2$ and symmetric cones $\Omega_1\subset V_1$ and $\Omega_2\subset V_2$ such that $V=V_1\oplus V_2$ and $\Omega=\Omega_1\times \Omega_2$. Any symmetric cone is the direct product of irreducible symmetric cone and there is a bijection between irreducible symmetric cones and simple Euclidean Jordan algebras (\cite{Far_Kor}, Prop. III.4.5).

The \emph{Gamma function} of an irreducible cone $\Omega$, introduced by Gindikin \cite{Gindikin64}, is defined for $Re(\lambda) >(r-1)\frac{d}{2}$ by the integral
\begin{equation}\label{def:GammaFunction}
\Gamma_\Omega(\lambda):=\int_\Omega e^{-\tr(x)}\Delta(x)^{\lambda-m}~dx,
\end{equation}
where $dx$ denotes the Euclidean measure on $V$. It satisfies the following identity:
\begin{equation}
\Gamma_\Omega(\lambda)=(2\pi)^{\frac{n-r}{2}}\prod_{i=1}^r \Gamma \left(\lambda-(i-1)\frac{d}{2}\right),
\end{equation}
where $\Gamma$ denotes the usual Gamma function. This identity gives a meromorphic extension of the Gamma function to the complex plane. 

One can also define the generalized \emph{Beta function}, denoted $B_\Omega$, by the following integral formula, for $Re(\lambda_1),Re(\lambda_2)>(r-1)\frac{d}{2}$:
\begin{equation}\label{def:BetaFunctionJordan}
B_\Omega(\lambda_1,\lambda_2):=\int_{\Omega\cap(e-\Omega)}\Delta(x)^{\lambda_1-m}\Delta(x)^{\lambda_2-m}~dx.
\end{equation}
It satisfies the following property (\cite{Far_Kor},Thm.VII.1.7):
\begin{equation}\label{eq:LinkBetaGammaJordan}
B_\Omega(\lambda_1,\lambda_2)=\frac{\Gamma_\Omega(\lambda_1)\Gamma_\Omega(\lambda_2)}{\Gamma_\Omega(\lambda_1+\lambda_2)}.
\end{equation}

\subsection{Tube domains}$\ $\\
We now consider the \emph{tube domain} $T_\Omega$ associated to $V$ and $\Omega$ which is defined as $T_\Omega:=V+i\Omega$. This is a subset of the complexification $V^\C$ of the Jordan algebra $V$. The \emph{conformal group} $G(T_\Omega)$ of $T_\Omega$ is defined as the group of bi-holomorphic automorphisms of $T_\Omega$. Let $H$ be a maximal compact subgroup in $\Conformal$ then the space $T_\Omega\simeq \Conformal/H$ is a Hermitian symmetric space of tube type. Every such tube domain is biholomorphic to a symmetric bounded domain (see \cite{Far_Kor},Thm.X.4.3).

We are going to describe the generators of the conformal group $\Conformal$. First, notice that we can see $G_\Omega$ as a subgroup of $\Conformal$. Indeed, $g\in G_\Omega$ acts on $T_\Omega$ via 
\begin{equation}
z=x+iy\mapsto g\cdot x+ig\cdot y.
\end{equation}

We define the group $N$ as the subgroup of $\Conformal$ consisting of translations $\tau_v$ by an element $v\in V$: 
\begin{equation} 
\tau_v(z):=z+v.
\end{equation}
The subgroup $P=G_\Omega \ltimes N$ is a parabolic subgroup of $\Conformal$.

Finally, we define the map $j$, called the inversion, by
\begin{equation}\label{eq:DefInversion}
j(z):=-z^{-1}.
\end{equation}
One shows that $j\in \Conformal$ (see \cite{Far_Kor},Thm.X.1.1).
Finally, this leads to the following result.
\begin{theorem}[\cite{Far_Kor},Thm.X.5.6]\label{prop:GeneratorsConformalGroup}
The group $G(T_\Omega)$ is generated by $N$, $G_\Omega$ and the inversion $j$.
\end{theorem}

Let us describe the Lie algebra $\g(T_\Omega)$ of $\Conformal$ (see \cite{Far_Kor}, Thm X.5.10). An element $X$ of $\g(T_\Omega)$ is identified with a vector field on $T_\Omega$ of the form 
\begin{equation}
X(z)=u+Tz+P(z)v,
\end{equation}
with $u,v\in V$ and $T\in V\square V$ the Lie algebra of $G_\Omega$. Thus we identify $X\in \g(T_\Omega)$ with the triple $(u,T,v)\in V\times V\square V\times V$ and the bracket is given by 
\begin{equation}
[(u,T,v),(u',T',v')]=(Tu'-T'u,[T,T']+2(u\square v'-u'\square v),T'^*u-T^*v).
\end{equation}
Together with the involution $\theta$ defined by:
\begin{equation}\label{eq:InvolutionKKT}
\theta(u,T,v)=(v,-T',u),
\end{equation}
it gives $\g(T_\Omega)$ the structure of a \emph{symmetric Lie algebra} (see \cite{Satake}, Thm.7.1 for more details). This construction is known as the \emph{Kantor-Koecher-Tits} Lie algebra associated to the Jordan algebra $V$. This description of the Lie algebra $\g(T_\Omega)$ leads to the so-called Gelfand-Naimark decomposition 
\begin{equation}\label{eq:decompositionLieAlgConformal}
\g(T_\Omega)=\n\oplus \mathfrak{l} \oplus \bar{\n},
\end{equation}
where 
\begin{align*}
&\n=\{(u,0,0)~|~u\in V\}\simeq V,\\
&\mathfrak{l}=\{(0,T,0)~|~T\in V\square V\}\simeq V\square V,\\
&\bar{\n}=\{(0,0,-v)~|~v\in V\}\simeq V.
\end{align*}
Furthermore, we have $\exp(\n)=N$, $\exp(\mathfrak{l})=G_0$, and $\exp(\bar{\n})=jNj$.

\subsection{Scalar valued holomorphic discrete series representations}

A unitary representation $\pi$ of $\Conformal$ on a Hilbert space $\mathcal{H}$ is said to be in the \emph{discrete series} if its matrix coefficients $m_{u,v}(g):=\langle\pi(g)u|v\rangle_{\mathcal{H}}$, $g\in \Conformal$, $u,v\in \mathcal{H}$, are square integrable functions with respect to the Haar measure on $\Conformal$. Furthermore, some members of the discrete series can be realized on Hilbert spaces $\mathcal{H}$ of holomorphic functions on the tube domain $T_\Omega$, and they are called \emph{holomorphic discrete series}. Finally, they are said to be \emph{scalar valued}, or of the \emph{scalar type}, if $\mathcal{H}$ is composed of scalar valued functions. The next paragraphs describes two realizations of such scalar valued holomorphic discrete series representations of $\Conformal$.

\paragraph{Holomorphic model}$\ $\\
Consider $\lambda \in \R$, and define the weighted Bergman space $H^2_\lambda(T_\Omega)$ as the space of holomorphic functions $f$ on $T_\Omega$ such that:
\begin{equation}
\|f\|_\lambda^2=\int_{T_\Omega} |f(z)|^2 \Delta(y)^{\lambda-2m}~dxdy < \infty.
\end{equation}

It turns out that this space is reduced to $\{0\}$ if $\lambda<1+d(r-1)=2m-1$, so from now on, we suppose that $\lambda>2m-1$. In this case, $\Bergmann{\lambda}$ is a Hilbert space of holomorphic functions which admits a reproducing kernel $K_\lambda$ defined by (\cite{Far_Kor},Prop.XIII.1.2):
\begin{equation}\label{eq:ReproducingKernelGeneral}
K_\lambda(z,w):=\frac{1}{(4\pi)^n}\frac{\Gamma_\Omega(\lambda)}{\Gamma_\Omega(\lambda-m)}\Delta\left(\frac{z-\bar{w}}{2i}\right)^{-\lambda}.
\end{equation}

We denote by $G$ the connected component of the identity in $\Conformal$ and by $\tilde{G}$ the universal covering group of $G$. Notice that $\tilde{G}$ acts on $T_\Omega$ by composition of the action of $G$ with the covering map $\tilde{G}\to G$, thus $\tilde{G}$ acts on $\Bergmann{\lambda}$. The scalar valued holomorphic discrete series representation $\pi_\lambda$ of $\tilde{G}$ is then defined by:
\begin{equation}\label{def:HolomorphicDiscreteScalarValued}
\pi_\lambda(g)f(z):=\det(D_g(z))^{\frac{\lambda }{2m}}f(g^{-1}\cdot z),
\end{equation}
where $D_g(z)$ denotes the differential of the map $z\mapsto g^{-1}\cdot z$. Notice that the power function is well defined on the universal cover $\tilde{G}$.

\paragraph{$L^2$-model}$\ $\\
In order to study the restriction of a member of the scalar valued holomorphic discrete series to some specific subgroup, another model for this representation  will be useful. Il will be done through the Laplace transform on the cone $\Omega$.

More precisely, let $\lambda>2m-1$ and consider the space $L^2_\lambda(\Omega):=L^2_\lambda(\Omega,\Delta(t)^{\lambda-m}~dt)$. Then the Laplace transform $\mathcal{F}$ defined by
\begin{equation}\label{def:LaplaceTransformGeneral}
\mathcal{F}f(z)=\int_\Omega f(t)e^{i(t|z)}\Delta(t)^{\lambda-m}~dt,
\end{equation}
is a one-to-one isometry (up to a scalar) from $L^2_\lambda(\Omega)$ onto $\Bergmann{\lambda}$ (see \cite{Far_Kor}, Thm.XIII.1.1). More precisely, we have:
\begin{equation}
\|\mathcal{F}f\|^2_\lambda=2^{2n-r\lambda}\pi^n \Gamma_\Omega(\lambda-m)\|f\|^2.
\end{equation}

Using the Laplace transform, one may make $\tilde{G}$ act on $L^2_\lambda(\Omega)$ via the operators:
\begin{equation}\label{def:L2ModelGeneral}
\rho_\lambda(g)=\mathcal{F}^{-1}\circ \pi_\lambda(g)\circ \mathcal{F},
\end{equation}
for $g\in \tilde{G}$.

Using Proposition \ref{prop:GeneratorsConformalGroup}, we get:
\begin{prop} Let $f\in L^2_\lambda(\Omega)$ and $t\in \Omega$.
\begin{itemize}
\item Let $g\in G_\Omega$, then
\begin{equation}
\rho_\lambda(g)f(t)=\det(g)^{\frac{\lambda}{2m}}f(g'\cdot t).
\end{equation} 

\item Let $v\in V$, then
\begin{equation}
\rho_\lambda(\tau_v)f(t)=e^{-i(t|v)}f(t).
\end{equation}

\item The action of the inversion $j$ is given by
\begin{equation}\label{eq:HankelTransformL2modelJordan}
\rho_\lambda(j)f(t)=\frac{1}{\Gamma_\Omega(\lambda )i^{\lambda r}}\int_\Omega J_{\lambda }\left(P(x^{\frac{1}{2}})t\right)f(x)\Delta(x)^{\lambda-m}~dx,
\end{equation}
\end{itemize}

where $J_\lambda$ is the \emph{Bessel function} on the simple Jordan algebra $V$ defined for $Re(\lambda)>(r-1)\frac{d}{2}$ by
\begin{equation}\label{def:BesselFunctionJordan}
\int_\Omega e^{-(x|y)}J_\lambda(x)\Delta(x)^{\lambda-m}~dx = \Gamma_\Omega(\lambda)\Delta(y)^{-\lambda}e^{i \tr(y^{-1})}.
\end{equation}
\end{prop}

\textit{Remark:} Notice that in the case $\Omega=\R^+$, the Bessel function \eqref{def:BesselFunctionJordan} coincides with the hypergeometric function $_0F_1$ and not the classical Bessel function of the first kind. Similarly as in this case, the operator $\rho_\lambda(j)$ is called the \emph{generalised Hankel transform }.

\begin{proof}
Let $g\in G_\Omega$, we have $D_g(z)=g^{-1}$ for any $z\in T_\Omega$ so that, using \eqref{def:HolomorphicDiscreteScalarValued}:
\begin{align*}
\pi_\lambda(g)\mathcal{F}(z)&=\det(g)^{-\frac{\lambda}{2m}}\int_\Omega f(t)e^{i(t|g^{-1}\cdot z)}\Delta(t)^{\lambda-m}~dt \\
&=\det(g)^{-\frac{\lambda}{2m}}\int_\Omega f(t)e^{i(g'^{-1}\cdot t| z)}\Delta(t)^{\lambda-m}~dt.\\
\end{align*}
and we get the result by a direct computation.

Let $v\in V$ then $D_{\tau_v}(z)=I_V$ for any $z\in T_\Omega$, and the result for $\rho_\lambda(\tau_v)$ is then immediate.

For the inversion $j$, we have $D_j(z)=P(z)^{-1}$  for any $z\in T_\Omega$ (see \cite{Far_Kor}, Prop.II.3.3). This gives:
\[\pi_\lambda(j)\mathcal{F}f(z)=\int_\Omega \Delta(z)^{-\lambda} e^{-i(t|z^{-1})}f(t)\Delta(t)^{\lambda-m}~dt.\]
Using the definition of Bessel functions \eqref{def:BesselFunctionJordan}, we find:
\begin{align*}
\Delta(z)^{-\lambda}e^{-i(t|z^{-1})}&=i^{-\lambda r}\Delta(t)^{-\lambda} \Delta\left(P(t^{-\frac{1}{2}})\frac{z}{i}\right)^{-\lambda}e^{-tr\left(\left(P(t^{-\frac{1}{2}})\frac{z}{i}\right)^{-1}\right)}\\
&=\frac{\Delta(t)^{-\lambda}}{i^{\lambda r}\Gamma_\Omega(\lambda)}\int_\Omega e^{i(x|P(t^{-\frac{1}{2}})z)}J_\lambda(x)\Delta(x)^{\lambda-m}~dx\\
&=\frac{1}{i^{\lambda r}\Gamma_\Omega(\lambda)}\int_\Omega e^{i(x|z)} J_\lambda(P(t^\frac{1}{2} )x)\Delta(x)^{\lambda-m}~dx.
\end{align*}
Finally, this gives:
\begin{equation*}
\pi_\lambda(j)\mathcal{F}f(z)=\int_\Omega\left( \frac{1}{i^{\lambda r}\Gamma_\Omega(\lambda )}\int_\Omega J_{\lambda}(P(t^\frac{1}{2} )x)f(t)\Delta(t)^{\lambda-m}~dt \right) e^{i(x|z)} \Delta(x)^{\lambda -m}~dx.
\end{equation*}
The Laplace transform being one-to-one this ends the proof.
\end{proof}

Finally, we introduce the \emph{Bessel operator} to describe the derived representation of the scalar valued holomorphic discrete series in the $L^2$-model . First, we need to fix some notations about differential operators.

For a scalar valued function $f~:~V\to \C$, we denote $D_u g$ the directional derivative in the direction $u\in V$:
\begin{equation}
D_u g(x)=\left.\deriv{}{t} f(x+tu)\right|_{t=0}.
\end{equation} 
The gradient of a scalar valued function $f$ is denoted $\frac{\partial f}{\partial x}$, and is expressed in an orthonormal basis $\{e_k\}$ of $V$ by
\begin{equation}
\frac{\partial f}{\partial x}=\sum_k \frac{\partial f}{\partial x_k}
e_k.
\end{equation}
The Bessel operator $\mathcal{B}_\lambda~:~C^\infty(V)\to C^\infty(V)\otimes V^\C$ (see \cite{Far_Kor}, section XV.2) is defined by the following expression 
\begin{equation}\label{def:BesselOperator}
\mathcal{B}_\lambda =P(\frac{\partial }{\partial x})x+\lambda\frac{\partial }{\partial x}.
\end{equation}
In an orthonormal basis $\{e_k\}$ of $V$ this expression has the following meaning
\begin{equation}\label{eq:BesselOperatorBasis}
\mathcal{B}_\lambda f(x)=\sum_{k,l}\frac{\partial^2f}{\partial x_k\partial x_l}P(e_k,e_l)x+\lambda\sum_k \frac{\partial f}{\partial x_k}
e_k.
\end{equation}
Finally, the derived representation of the Lie algebra $\g(T_\Omega)$ of the scalar valued holomorphic discrete series is given, in the $L^2$-model, on the space of smooth vestors $L^2_\lambda(\Omega)^\infty$ by the following Proposition.
\begin{prop}[\cite{Mollers14}, section 2.1]\label{prop:LieAlgebraAction}
Using the decomposition \eqref{eq:decompositionLieAlgConformal}, $\g(T_\Omega)=\n\oplus \mathfrak{l}\oplus \bar{\n}$, the derived representation of $\g(T_\Omega)$ is given, for $f\in L^2_\lambda(\Omega)^\infty$ by the operators:
\begin{align}
&d\rho_\lambda(u,0,0)f(x)=i(u|x)f(x),\\
&d\rho_\lambda(0,T,0)f(x)=\frac{\lambda}{2m}\Tr(T^*)f(x)+D_{T^*x}f(x),\\
&d\rho_\lambda(0,0,v)f(x)=i(v|\mathcal{B}_\lambda f(x)).
\end{align}

\end{prop}

\section{Stratification of a symmetric cone}

This section introduces a diffeomorphism between a symmetric cone $\Omega_2$ and a product $\Omega_1\times X$ where $\Omega_1$ is a symmetric cone inside $\Omega_2$ and $X$ is a space to be specified. Using this decomposition, we obtain a Hilbert space isomorphism between some $L^2$ spaces which carry the $L^2$-model presented in the previous section, and we use it in order to define yet another model for the scalar valued holomorphic discrete series. This model will be useful to study the restriction of scalar valued holomorphic discrete series of $G(T_{\Omega_2})$ to $G(T_{\Omega_1})$.

\subsection{Geometric setting}\label{sec:defintionDiffeo}

Let $V_1,V_2$ be two Euclidean Jordan algebras with inner product $(x|y)_i=\tr_i(xy)$ ($i=1,2$). We denote with an index $i=1,2$ all the associated notions related to $V_i$ described in Section \ref{chap:JordanAlgebras}. Let $\eta~:~V_1\to V_2$ be an injective unital Jordan algebra homomorphism such that for all $x,\ y \in V_1$ there exists $\mu\in \R^*$ such that 
\begin{equation}\label{eq:conditionScalarProduct}
(\eta(x)|\eta(y))_2=\mu(x|y)_1.
\end{equation}
If this is true, we have $\mu=\frac{r_2}{r_1}$ if we apply this equality to $x=y=e_1$.

The embedding $\eta$ can be extended into an holomorphic embedding of $T_{\Omega_1}$ into $T_{\Omega_2}$ via
\begin{equation}
\eta(x+iy)=\eta(x)+i\eta(y).
\end{equation}

\textit{Remark:} If $V_1$ is supposed to be a simple Jordan algebra the condition $(\eta(x)|\eta(y))_2=\mu(x|y)_1$ for $x,~y\in V_1$ is always verified. However, we want to consider some examples in which $V_1$ is not simple and it turns out to be false in general. As a counter example, choose $V_1=V\oplus W$ and $V_2=V\oplus V\oplus W$ with $V,~W$ two simple Euclidean Jordan algebras, and $\eta(v,w)=(v,v,w)$. Then $\eta$ is  a unital Jordan algebra homomorphism but $(\eta(v,w)|\eta(v',w'))_2=2(v|v')_V+(w|w')_W\neq ((v,w)|(v,w))_1$. This condition will be necessary to study restrictions of the holomorphic discrete series. 

\begin{example}\label{ex:ProductJordanAlgebra}
\begin{enumerate}
\item Take $V_1=\R$ and $\Omega_1=\R^+$, $V_2$ a Euclidean Jordan algebra with identity $e$ and $\Omega_2$ its associated symmetric cone. Then, the map $\eta~:~\R^+ \mapsto V_2$ defined for $t\in \R^+$ by $\eta(t)=t e$ is an example of this construction.
\item Another example is given by the diagonal embedding of a Euclidean Jordan algebra $V$ into the direct product $V^p$ with $p\in\N$. Here, we have $\eta(v)=(v,\cdots,v)$ for all $v\in V$. In this setting, $r_2=p r_1$ so that $(\eta(x)|\eta(y))_2=p(x|y)_1$, and $m_2=m_1$.
\end{enumerate}
\end{example}

The map $\eta$ induces a map $\rho : \g(T_{\Omega_1}) \to \g(T_{\Omega_2})$ defined, using the Kantor-Koecher-Tits construction $\g(T_{\Omega_i})\simeq V_i\times V_i\square_i V_i\times V_i$ ($i=1,2$) with $V_i\times V_i$ the Lie algebra of $G_{\Omega_i}$, by:
\begin{equation}
\rho((u,T,v))=(\eta(u),\rho_0(T),\eta(v)), \text{ for } (u,T,v)\in V_1\times V_1\square_1 V_1\times V_1,
\end{equation}
where $\rho_0:V_1\square_1 V_1 \to V_2\square_2 V_2$ is defined by:
\begin{equation}\label{def:embeddingStructureGroup}
\rho_0(x\square_1 y)=\eta(x)\square_2\eta(y).
\end{equation}

One can check that $\rho$ and $\rho_0$ are Lie algebras morphisms, and that $\rho$ satisfies:
\begin{align}
 \rho\circ \theta_1&=\theta_2\circ \rho,\\
 \rho(\n_1)\subset \n_2,
 \rho(\mathfrak{l}_1)&\subset \mathfrak{l}_2,
 \rho(\bar{\n}_1)\subset \bar{\n}_2,
 \end{align}
 where $\g(T_{\Omega_i})=\n_i\oplus\mathfrak{l}_i\oplus\bar{\n}_i$ is the Gelfand-Naimark decomposition and $\theta_i$ is defined in \ref{eq:InvolutionKKT}.
It is known that, under these conditions, $\rho$ and $\eta$ determines each other uniquely (see \cite{Satake}, Chap.I, Prop.9.1). 

Moreover, for $T\in V_1\square V_1$ and $x\in V_1$, $\rho_0$ satisfies:
\begin{align}
\eta(T\cdot x)&=\rho_0(T)\cdot\eta(x), \\
\rho_0(T')&=\rho_0(T)^*,\label{eq:PropertyRho}
\end{align}
where $'$ and $^*$ denote the adjoint with respect to the inner product in $V_1$ and $V_2$ respectively. Once again, under these conditions, $\rho_0$ and $\eta$ determines each other uniquely (see \cite{Satake}, Chap.I, Prop.9.2). 

Finally, a direct computation shows that, for $X\in \g(T_{\Omega_1})$, and $z\in T_{\Omega_1}$, we have:
\begin{equation}
\eta(X(z))=\rho(X)\eta(z).
\end{equation}

In the following, we always assume that the morphism $\rho$ (resp. $\rho_0$) can be lifted to a morphism from $G(T_{\Omega_1})$ to $G(T_{\Omega_2})$ (resp. from $G_{\Omega_1}$ to $G_{\Omega_2}$), and we denote it by the same letter. This assumption will always be satisfied in the examples treated in Section \ref{sec:ntensorPorduct} and \ref{sec:SO(2,n)}.

From now on, suppose $V_1$ is a unital subalgebra of $V_2$ so that there is a natural embedding of $V_1$ into $V_2$ and the notation $\eta$ can be omitted. In this setting, the identity element is the same for $V_1$ and $V_2$, and it will be denoted $e$ for both algebras. If the context is clear, we also drop the notation $\rho$ for the embedding of $G_{\Omega_1}$ into $G_{\Omega_2}$, and in view of the property \eqref{eq:PropertyRho} we denote $g'$ the adjoint of an element in $G_{\Omega_1}$ with respect to any of the inner product on $V_1$ and on $V_2$.

We have the following facts:
\begin{prop}$\ $\label{prop:formule_determinant}
\begin{enumerate}
\item If $x\in V_1$ then $P_1(x)=P_2(x)|_{V_1}$.
\item If $x\in V_1$, then $\tr_1(x)=\frac{r_2}{r_1}\tr_2(x)$.
\item If $r_1=r_2$, then for $x\in V_1$, $\Delta_1(x)=\Delta_2(x)$ and $\tr_1(x)=\tr_2(x)$.
\end{enumerate}
\end{prop}

\begin{proof}
The first statement is obvious since $V_1$ is a Jordan subalgebra of $V_2$. The second one, is a consequence of the equality $(x|y)_2=\frac{r_2}{r_1}(x|y)_1$ for $y=e$.

For the last one, notice that the set of regular elements in $V_1$ is a subset of the regular elements in $V_2$, which is not the case if $r_1\neq r_2$. Let $x\in V_1$, then the minimal polynomial $P_x$ associated to $x$ does not depend on the fact that we consider $x$ in $V_1$ or in $V_2$. We deduce the result from the definition \eqref{def:JordanTraceDeterminant} of the Jordan determinant $\Delta$ for regular elements.
\end{proof}

\textit{Remark:} Generally, we have $\Delta_1(x)\neq \Delta_2(x)$ for $x\in V_1$. An example is given by the embedding of $\R$ in any Euclidean Jordan algebra $V$ of rank $r\neq 1$ given by $\eta(t)=t\cdot e$ for $t\in\R$, we have $\Delta_1(t)=t$ and $\Delta_2(t)=t^r$. This example suggests that we could have the relation $\Delta_1^{r_2}=\Delta_2^{r_1}$ on $V_1$, at least these two polynomials have the same degree of homogeneity. However, the example given by $V_1=V\oplus W$ and $V_2=V\oplus V\oplus W$ with $V,~W$ two simple Euclidean algebras provides a counter example.

\begin{prop}
Let $i=1,2$ and $\Omega_i$ be the symmetric cone in the Euclidean Jordan algebra $V_i$, where $V_1$ is a unital subalgebra of $V_2$, then
 \begin{equation}
 \Omega_1=\Omega_2\cap V_1.
 \end{equation}
\end{prop}
\begin{proof}
We have to show that $\Omega:=\Omega_2\cap V_1$ is a subset of $\Omega_1$. We know that $\Omega_1\subset \Omega$ so $\Omega^*\subset \Omega_1^*=\Omega_1$ where $\Omega^*$ is the dual cone of $\Omega$ in $V_1$.

Let $x \in\Omega$, we have $x\in\Omega_2$ so there is $x_0\in V_2$ such that $x=x_0^2$. There is a Jordan frame $(c_1,\cdots,c_{r_2})$ and $(\lambda_1,\cdots,\lambda_{r_2})$ with $\lambda_i>0$ such that $x_0=\sum_{i=1}^{r_2}\lambda_i c_i$. Finally, for $y\in \Omega$:
\[(x|y)_1=\frac{r_1}{r_2}(x|y)_2=\frac{r_1}{r_2}\sum_{i=1}^{r_2} \lambda_i^2(c_i| y)_2.\]

We have $(c_i| y)_2 \geq 0$ because $c_i\in \overline{\Omega_2}$, but all the $(c_i| y)_2$ cannot be zero at the same time, otherwise $y=0$, so $(x|y)_1>0$.  This shows that $\Omega\subset \Omega^*\subset \Omega_1$.
\end{proof}
Let $V_1^{\bot}$ be the orthogonal complement of $V_1$ in $V_2$ with respect to the trace form \eqref{eq:JordanInnerProduct} on $V_2$.
\begin{prop}
\begin{enumerate}
\item The space $V_1^\bot$ is stable under the action of $G_{\Omega_1}$.
\item Every $x\in \Omega_2$ can be written $x=x_1+y$ with $x_1\in \Omega_1$ and $y\in V_1^{\bot}$. 
\item Let $x\in \Omega_1$, then:
\begin{equation}\label{eq:identite_determinant}
\det\left(P_2(x)\right)=\det\left(P_1(x)\right)\times \det\left(P_{V_1^\bot}(x)\right).
\end{equation}
\end{enumerate}
\end{prop}

\begin{proof}
\begin{enumerate}
\item Let $g\in G_{\Omega_1}$, $x\in V_1$ and $y\in V_1^\bot$. Then
\begin{equation*}
(g\cdot y|x)_2=(y|g'\cdot x)_2=0,
\end{equation*}
because $G_{\Omega_1}$ is stable under the map $g\mapsto g'$.
\item Let $x\in \Omega_2\subset V_2$, then $x=x_1+y$ with $x_1\in V_1$ and $y\in V_1^\bot$. and, for $z\in \overline{\Omega_1}\backslash\{0\}$ we have:
$$(x_1|z)_1=\frac{r_1}{r_2}\left((x|z)_2-(y|z)_2\right)=(x|z)_2>0,$$
which proves that $x_1\in \Omega_1$ according to Definition \ref{def:SymmetricCone}.
\item For $x\in\Omega_1$, the map $P_2(x)\in G_{\Omega_1}$ stabilizes $V_1$ and $V_1^\bot$. This gives a decomposition $P_2(x)=P_{V_1}(x)\oplus P_{V_1^\bot}(x)$ where the subscript is for the restriction of $P_2(x)$ to the corresponding subspace. From Proposition \ref{prop:formule_determinant}, we have $P_{V_1}(x)=P_1(x)$ so the third point is clear.
\end{enumerate}
\end{proof}

Define the \emph{stratification space} $X$ as
\begin{equation}\label{def:SubsetXGeneral}
X:=\left\{u\in V_1^{\bot}~|~e+u\in \Omega_2\right\},
\end{equation}
and define the map $\iota~:~\Omega_1\times X \to \Omega_2$ by:
\begin{equation}\label{def:DiffeoGeneral}
\iota(x,u):=\frac{r_1}{r_2} P_2\left(x^\frac{1}{2}\right)(e+u).
\end{equation}

The following Lemma describes the space $X$.
\begin{lemme}
The stratification space $X=\{u\in V_1^{\bot}~|~e+u\in \Omega_2\}$ is an open bounded convex subset of $V_1^\bot$.
\end{lemme}

\begin{proof}
The space $X$ is obviously open and convex in $V_1^\bot$ because $\Omega_2$ is open and convex. To see that it is bounded notice that, because $e\in V_1$, we have:
\[
X\subset \{x \in\left(\R\cdot e\right)^\bot~|~e+x\in\Omega_2\}.
\]
The right-hand side is equal to 
\[
\{y\in\Omega_2~| ~(y-e |e)=0\},
\]
which is compact. 
\end{proof}

The next proposition shows the stratification of the cone $\Omega_2$ by the cone $\Omega_1$.
\begin{prop}\label{prop:DiffeoGeneral}
The map $\iota~:~\Omega_1\times X\to \Omega_2$ defined in \eqref{def:DiffeoGeneral} is a diffeomorphism whose inverse is given, for $z=x+y$ with $x\in \Omega_1,~y\in V_1^\bot$, by 
\begin{equation}
\iota^{-1}(z)=\left(\frac{r_2}{r_1}x,P_2(x^{-\frac{1}{2}})y\right).
\end{equation}
Its Jacobian is $\left(\frac{r_1}{r_2}\right)^{n_2}\det(P_2(x^\frac{1}{2}))\cdot \det(P_1(x^\frac{1}{2}))^{-1}$.

Moreover, it satisfies the following identities:
\begin{align} 
\Delta_2(\iota(x,u))&=\Delta_2(x)\Delta_2(e+u)\label{eq:detDiffeo},\\
 \tr_2(\iota(x,u))&=\tr_1(x).\label{eq:traceDiffeo}
\end{align}
\end{prop}
\begin{proof}
The image of $\iota$ is in $\Omega_2$ in view of the definition \eqref{def:SubsetXGeneral} of the stratification space $X$. For $\iota^{-1}$, we have $z=x+y=P_2(x^{\frac{1}{2}})(e+P_2(x^{-\frac{1}{2}})y)$, so $P_2(x^{-\frac{1}{2}})y$ is in $X$.

A direct computation proves that $\iota^{-1}$ is actually the inverse for $\iota$, and the fact that $x\mapsto x^{1/2}$ is a diffeomorphism proves that $\iota$ is a diffeomorphism.

Finally, the Jacobian matrix is given, in natural coordinates on $\Omega_1\times X$ and $\Omega_2$, by
$$\left(\frac{r_1}{r_2}\right)\cdot\begin{pmatrix}
I & 0\\ \star & P_{V_1^\bot}(x^{\frac{1}{2}})
\end{pmatrix},$$
so the Jacobian, using \eqref{eq:identite_determinant}, is given by:
$$\left(\frac{r_1}{r_2}\right)^{n_2}\det(P_{V_1^\bot}(x^{\frac{1}{2}}))=\left(\frac{r_1}{r_2}\right)^{n_2}\det(P_2(x^\frac{1}{2}))\cdot \det(P_1(x^\frac{1}{2}))^{-1}.
$$

The statement \eqref{eq:detDiffeo} is a consequence of Proposition \ref{prop:FormulesDeterminantQuadratic}. For \eqref{eq:traceDiffeo}, we have:
\[
\tr_2(\iota(x,u))=\frac{r_1}{r_2}\left(\tr_2(x)+\tr_2(P(x^\frac{1}{2})u)\right)=\tr_1(x)+\frac{r_1}{r_2}(u|x)_2.
\]
But $(u|x)=0$, as $u\in V_1^\bot$ and $x\in V_1$, what ends the proof.
\end{proof}

\begin{example}\label{ex:ProduitTensorielDiffeo}
Let $V_1=V$ be a simple Euclidean Jordan algebra and $V_2=V^p=\overbrace{V\times\cdots\times V}^{p}$ with $p\in\N$ (see example \ref{ex:ProductJordanAlgebra}). In this setting, $\Omega_1=\Omega$ is the irreducible symmetric cone associated with $V$, and $\Omega_2=\Omega^p$. 

Consider the diagonal embedding of $\Omega$ into $\Omega^p$, then we get
\begin{equation}
V_1^\bot=\left\{(v_1,\cdots,v_{p-1},-\sum_{k=1}^{p-1} v_k)\right\},
\end{equation}
And hence
\begin{align}
X&=\left\{(v_1,\cdots,v_{p-1},-\sum_{k=1}^{p-1} v_k)~|~e+v_k\in \Omega\text{ and } e-\sum_{k=1}^{p-1} v_k\in \Omega\right\}\\
&\simeq\left\{(v_1,\cdots,v_{p-1})~|~e+v_k\in \Omega\text{ and } e-\sum_{k=1}^{p-1} v_k\in \Omega\right\}.\nonumber
\end{align}

For a vector $x=(x_1,\cdots,x_n )\in \C^n$, we define:
\begin{equation}
|x|=\sum_{k=1}^n x_k.
\end{equation}

The map $\iota$ is then given by:
\begin{equation}
\iota(t,v)=\left(\frac{P(t^\frac{1}{2})(e+v_1)}{p},\cdots,\frac{P(t^\frac{1}{2})(e+v_{p-1})}{p},\frac{P(t^\frac{1}{2})(e-|v|)}{p}\right).
\end{equation}
The inverse map is given, for $(x_1,\cdots,x_p)=(\frac{|x|}{p},\cdots,\frac{|x|}{p})+(\frac{px_1-|x|}{p},\cdots,\frac{px_p-|x|}{p})\in V_1\oplus V_1^\bot$, by:
\begin{equation}
\iota^{-1}(x_1,\cdots,x_p)=\left(|x|,P\left((|x|)^{-\frac{1}{2}}\right)\cdot(p x_1-|x|),\cdots,P\left((|x|)^{-\frac{1}{2}}\right)\cdot(p x_{p-1}-|x|)\right).
\end{equation}
Finally, in these coordinates on the target space, the Jacobian is given by $\frac{1}{p^{(p-1)n}}\Delta(x)^{(p-1)m}$.
The case $p=2$ corresponds to the diffeomorphism given in \cite{JL.Clerc}. 
\end{example}

As a first application of the diffeomorphism $\iota$ \eqref{def:DiffeoGeneral}, we get the following result which relates the Gamma function on the cone $\Omega_2$ to the one on the cone $\Omega_1$.

\begin{coro}
Let $V_2$ be a simple Euclidean Jordan algebras a $V_1$ a subalgebra of $V_2$ such that \eqref{eq:conditionScalarProduct} is verified and $r_1=r_2$, then for $Re(\lambda)> \frac{(r_2-1)d_2}{2}$:
\begin{equation}
\Gamma_{\Omega_2}(\lambda)=V_X(\lambda) \Gamma_{\Omega_1}(\lambda),
\end{equation}
where $V_X(\lambda)=\int_X\Delta_2(e+u)^{\lambda-m_2}~du$ is the volume of $X$ with respect to the measure $\Delta_2(e+u)^{\lambda-m_2}~du$.
\end{coro} 

\begin{proof}
This result comes from a direct computation of the Gamma function, using Proposition \ref{prop:formule_determinant}:
\begin{align*}
\Gamma_{\Omega_2}(\lambda)&=\int_{\Omega_2} e^{-\tr_2(x)}\Delta_2(x)^{\lambda-m_2}~dx\\
&=\int_{\Omega_1\times X} e^{-\tr_2(\iota(y,u))}\Delta_1(y)^{\lambda-m_1}\Delta_2(e+u)^{\lambda-m_2}~dydu.
\end{align*}
Finally, one gets the result by using formula \eqref{eq:traceDiffeo}.
\end{proof}

Some formulas about Gamma functions can be obtained if one consider an irreducible symmetric cone $\Omega$ and $\Omega_2=\Omega^p$ with $p\in \N$. In this setting, we get
\begin{coro}\label{coro:ProductGammaFunctionJordan}
Let $p\in\N$, $\Omega$ an irreducible symmetric cone, and $\Lambda=(\lambda_1,\cdots,\lambda_p)$ such that $Re(\lambda_k)>(r-1)\frac{d}{2}$, for $1\leq k\leq p$. Then the following equality holds:
\begin{equation}
\prod_{k=1}^p \Gamma_\Omega(\lambda_k)=\frac{\Gamma_\Omega(|\Lambda|)}{p^{r|\Lambda|-n}}I_\Omega(\lambda_1,\cdots,\lambda_p),
\end{equation}
where
\begin{equation}
I_\Omega(\lambda_1,\cdots,\lambda_p)=\int_X\Delta(e-|v|)^{\lambda_p-m}\prod_{k=1}^{p-1} \Delta(e+v_k)^{\lambda_k-m}~dv_k.
\end{equation}
Moreover, we have:
\begin{equation}
I_\Omega(\lambda_1,\cdots,\lambda_p)=p^{r|\Lambda|-n}\prod_{k=1}^{p-1} B_\Omega(\sum_{s=1}^k\lambda_s,\lambda_{k+1}).
\end{equation}
\end{coro}

\begin{proof}
Using the diffeomorphism described in Example \ref{ex:ProduitTensorielDiffeo}, we get:
\begin{align*}
&\prod_{k=1}^p \Gamma_\Omega(\lambda_k)=\prod_{k=1}^p \int_{\Omega}e^{-\tr(x_k)}\Delta(x_k)^{\lambda_k-m}~dx_k\\
&=\int_{\Omega^p}e^{-\tr(|x|)}\prod_{k=1}^p\Delta(x_k)^{\lambda_k-m}~dx_k\\
&=\frac{1}{p^{(p-1)n}}\int_{\Omega\times X}e^{-\tr(t)}\Delta(t)^{|\Lambda|-m}\prod_{k=1}^{p-1}\Delta\left(\frac{e+v_k}{p}\right)^{\lambda_k-m}\Delta\left(\frac{e-|v|}{p}\right)^{\lambda_p-m}~dv_kdt
\end{align*}
which ends the proof of the first statement.

The second statement is proved using recursion on $p$ and the formula \eqref{eq:LinkBetaGammaJordan}.
\end{proof}

\textit{Remark:} In the case $p=2$, one recovers the Beta function $B_\Omega$  \eqref{def:BetaFunctionJordan}. More precisely, when $p=2$ we have $X\simeq (\Omega-e)\cap(e-\Omega)$, and hence:
\begin{equation}
B_\Omega(\lambda_1,\lambda_2)=\frac{1}{2^{r(\lambda_1+\lambda_2)}}I_\Omega(\lambda_1,\lambda_2).
\end{equation}

\subsection{Application to scalar valued holomorphic discrete series}

Let $V_2$ be a Euclidean Jordan algebra and $V_1$ be a unital Euclidean Jordan subalgebra of $V_2$ such that $(x|y)_1=\frac{r_1}{r_2}(x|y)_2$. Let $\Omega_1\subset \Omega_2$ their associated symmetric cone, and we denotes with an index $i=1,2$ all the associated notions (inner product, rank of the Jordan algebras, etc...). This gives naturally a pair of tube domains $(T_{\Omega_1},T_{\Omega_2})$ such that $T_{\Omega_1}\subset T_{\Omega_2}$. The diffeomorphism $\iota$ will be used to introduce yet another model for the representations of the scalar valued holomorphic discrete series which will be useful to study the branching law of such representations of the group $G(T_{\Omega_2})$ restricted to the subgroup $G(T_{\Omega_1})$. We shall treat two different kinds of pairs $(T_{\Omega_1},T_{\Omega_2})$:
\begin{enumerate}
\item\label{case:rankEqual} In the first case we suppose that $V_2$ is a simple Euclidean Jordan algebras and $V_1$ is a subalgebra in $V_2$ such that they have the same rank $r_1=r_2$, and we want to study the restriction of a member of the scalar valued holomorphic discrete series of $G(T_{\Omega_2})$ to $G(T_{\Omega_1})$.
\item \label{case:tensorProduct} The second case is for the direct product $V_2=\overbrace{V\times \cdots \times V}^p$ where $r_2=p r_1$ and $V$ is a simple Euclidean Jordan algebra. This setting is useful to the study of the irreducible decomposition of the $p$-fold tensor product of  representations of the scalar valued holomorphic discrete series of $G(T_{\Omega})$ when restricted to the diagonal.
\end{enumerate}

\paragraph{Case \ref{case:rankEqual}}$\ $\\
Using the diffeomorphism $\iota$, we deduce the following Theorem:
\begin{prop}\label{prop:Hilbert_space_isom_EqualRank}
Let $V_1$ and $V_2$ be two Euclidean Jordan algebras which have the same rank $r_1=r_2:=r$ with $V_2$ simple, and $\lambda\in \C$ such that $Re(\lambda)>(r-1)\frac{d_2}{2}$ where $d_2$ is the common dimension of the spaces $V_{ij}$ which appears in the Pierce decomposition \eqref{eq:PierceDecomposition}. Then, the pullback $\iota^*f=f\circ \iota$ induces the following isomorphism of Hilbert spaces:
\begin{equation}\label{eq:Hilbert_space_isom_EqualRank}
 L^2_\lambda(\Omega_2)\simeq L^2_\lambda(\Omega_1)\hat{\otimes}L^2(X,\Delta_2(e+v)^{\lambda-m_2}~dv).
\end{equation}
\end{prop}

\begin{proof}
First, suppose that $V_2$ is simple, then for $f\in L^2_\lambda(\Omega_2)$:
\begin{align*}
&\int_{\Omega_2} |f(x)|^2 \Delta_2(x)^{\lambda-m_2}~dx \\
&=\int_{\Omega_1\times X} |f\circ \iota(t,v)|^2 \Delta_2\left(P_2({t}^\frac{1}{2})(e+v)\right)^{\lambda-m_2}\det(P_2(t^\frac{1}{2}))\det(P_1(t^\frac{1}{2}))^{-1}~dt dv\\
&=\int_{\Omega_1\times X} |f\circ \iota(t,v)|^2 \Delta_1(t)^{\lambda-m_1}\Delta_2(e+v)^{\lambda-m_2}~dvdt,
\end{align*}
where we used formula \eqref{eq:detDiffeo}.
\end{proof}

Next, we consider the restriction of a member of the holomorphic discrete series for the group $G(T_{\Omega_2})$ to the subgroup $G(T_{\Omega_1})$. Let us introduce yet another model for this representation on the space $L^2_\lambda(\Omega_1)\hat{\otimes}L^2(X,\Delta_2(e+v)^{\lambda-m_2}~dv)$ called the \emph{stratified model}. The group $G(T_{\Omega_1})$ then acts on this space via the operators 
\begin{equation}\label{def:ModeleStratifJordan}
S_\lambda(g)=\iota^*\circ \rho_\lambda(g)\circ {\iota^*}^{-1},
\end{equation}
where $\rho_\lambda$ denotes the action of $G(T_{\Omega_2})$ on the space $L^2_\lambda(\Omega_2)$ and $G(T_{\Omega_1})$ is viewed as a subgroup of $G(T_{\Omega_2})$.
The following proposition describes this action on the generators of $G(T_{\Omega_1})$ (see Theorem \ref{prop:GeneratorsConformalGroup}).

\begin{prop}\label{prop:ModeleStratifJordanEqualRank}
Let $f\in L^2_\lambda(\Omega_1)\hat{\otimes}L^2(X,\Delta_2(e+v)^{\lambda-m_2}~dv)$, and $(t,v)\in \Omega_1\times X$.
\begin{enumerate}
\item Let $g\in G_{\Omega_1}$, then
\begin{equation}\label{eq:actionRestrictionStructure}
S_\lambda(g)f(t,v)={\det}_2 (g)^{\frac{\lambda}{2m_2}}f(g'\cdot t,k_{g',t}\cdot v),
\end{equation}
where $k_{g',t}=P\left((g'\cdot t)^\frac{-1}{2}\right)g'P(t^\frac{1}{2})\in K_1$. 
\item Let $u\in V_1$, then
\begin{equation}\label{eq:actionRestrictionTranslation}
S_\lambda(\tau_u)f(t,v)=e^{-i(t|u)_1}f(t,v).
\end{equation}

\item The action of the inversion $j$ is given by
\begin{multline}
S_\lambda(j)f(t,v)=\\ 
\frac{1}{\Gamma_{\Omega_2}(\lambda)i^{r_2\lambda}}\int_{\Omega_1\times X}f(t',v')J_\lambda\left(P(\iota(t',v')^\frac{1}{2})\iota(t,v)\right)\Delta_1(t')^{\lambda-m_1}\Delta_2(e+v')^{\lambda-m_2}~dt'dv'.
\end{multline}
\end{enumerate}
\end{prop}

\begin{proof}
$\ $
\begin{enumerate}
\item Let $g\in G_{\Omega_1}$, then using the definition \eqref{def:ModeleStratifJordan} one finds
\begin{equation*}
S_\lambda(g)f(t,v)={\det}_2(g)^{\frac{\lambda}{2m_2}}f\left(\iota^{-1}(g'\cdot \iota(t,v))\right).
\end{equation*}
This gives the result if one notice that
\begin{equation*}
\iota^{-1}(g'\cdot \iota(t,v))=\iota^{-1}\left(g'\cdot t +g'\cdot P(t^\frac{1}{2})v\right)=(g'\cdot t,k_{g,t}\cdot v),
\end{equation*}
because $g'$ leaves the decomposition $V_1\oplus V_1^\bot$ invariant. Finally, we have
$k_{g',t}\cdot e=P\left((g'\cdot t)^\frac{-1}{2}\right)g'P(t^\frac{1}{2})\cdot e =e$, so that $k_{g',t}\in K_1$.
\item For the elements $\tau_u$ a direct computation leads to the result using the formula \eqref{eq:traceDiffeo}.
\item This statement is a consequence of the formula \eqref{eq:HankelTransformL2modelJordan} and the properties of the diffeomorphism $\iota$.
\end{enumerate}
\end{proof}

\paragraph{Case \ref{case:tensorProduct}}$\ $\\ 
The situation is parallel in the case \ref{case:tensorProduct}. Indeed, one can get a similar isomorphism in the setting described in Example \ref{ex:ProduitTensorielDiffeo}.

\begin{prop}\label{prop:Hilbert_space_isom_tensorProduct}
Let $p\in \N$, $\Omega$ be an irreducible symmetric cone, and $\Lambda=(\lambda_1,\cdots,\lambda_p)$ such that $Re(\lambda_k)>(r-1)\frac{d}{2}$. Then, the pull-back $\iota^*f=f\circ \iota$ induces the following isomorphism (up to a scalar) of Hilbert spaces:
\begin{multline}\label{eq:Hilbert_space_isom_tensorProduct}
 L^2\left(\Omega^p, \prod_{k=1}^p\Delta(x_k)^{\lambda_k-m}~dx_k\right)\\
 \simeq L^2_{|\Lambda|}(\Omega)\hat{\otimes}L^2\left(X,\Delta\left(e-|v|\right)^{\lambda_p-m}\prod_{k=1}^{p-1}\Delta(v_k)^{\lambda_k-m}~dv_k\right).
\end{multline}
More precisely, we have:
\begin{equation}
\|\iota^*f\|^2=p^{r|\Lambda|-n}\|f\|^2.
\end{equation}
\end{prop}

\begin{proof}
Following the ideas in the proof of Corollary \ref{coro:ProductGammaFunctionJordan}, we get:
\begin{multline*}
\int_{\Omega^p}|f(x_1,\cdots,x_p)|^2\prod_{k=1}^p\Delta(x_k)^{\lambda_k-m}~dx_k=\\
\frac{1}{p^{r|\Lambda|-n}} \int_{\Omega\times X}|\iota^*f(t,v_1,\cdots v_{p-1})|^2\Delta(t)^{|\Lambda|-m}\Delta\left(e-|v|\right)^{\lambda_p-m}\prod_{k=1}^{p-1}\Delta(v_k)^{\lambda_k-m}~dv_kdt.
\end{multline*}
\end{proof}

For convenience, we set:
\begin{equation}
V_\Lambda=L^2_{|\Lambda|}(\Omega)\hat{\otimes}L^2\left(X,\Delta\left(e-|v|\right)^{\lambda_p-m}\prod_{k=1}^{p-1}\Delta(v_k)^{\lambda_k-m}~dv_k\right).
\end{equation}

Now, consider $(\lambda_1,\cdots,\lambda_p)$ such that $Re(\lambda_k)>(r-1)\frac{d}{2}$ and the outer tensor product $\pi_{\lambda_1}\boxtimes \cdots \boxtimes \pi_{\lambda_p}$ of scalar valued holomorphic discrete series representations of $G(T_\Omega)$. In the $L^2$-model this representation acts on the space $L^2\left(\Omega^p, \prod_{k=1}^p\Delta(\xi_k)^{\lambda_k-m}~d\xi_k\right)$, and using the diffeomorphism $\iota$ we define the stratified model on the space $V_\Lambda$, via the operators:
\begin{equation}\label{def:ModeleStratifJordanTensoriel}
S_\Lambda(g)=\iota^*\circ \rho_{\lambda_1}(g)\boxtimes \cdots \boxtimes\rho_{\lambda_p}(g) \circ {\iota^*}^{-1}.
\end{equation}
When restricted to the diagonal subgroup $\Delta(G(T_{\Omega})^p)\simeq G(T_{\Omega})$, the representation becomes the inner tensor product, and its action on the space $V_\Lambda$ is given by the following Proposition.

\begin{prop}\label{prop:ModeleStratifDirectProduct}
Let $f\in V_\Lambda$, and $(t,v_1,\cdots,v_{p-1})\in \Omega\times X$.
\begin{enumerate}
\item Let $g\in G_\Omega$, then
\begin{equation}\label{eq:FormulStratifEqualRankConforme}
S_\Lambda(g)f(t,v_1,\cdots,v_{p-1})=\det(g)^{\frac{|\Lambda|}{2m}}f\left(g'\cdot t,k_{g',t}\cdot v_1,\cdots,k_{g',t}\cdot v_{p-1}\right),
\end{equation}
where $k_{g',t}=P\left((g'\cdot t)^\frac{-1}{2}\right)g'P(t^\frac{1}{2})\in K$. 
\item Let $u\in V$, then
\begin{equation}\label{eq:FormulStratifEqualRankTranslation}
S_\Lambda(\tau_u)f(t,v)=e^{-i(t|u)}f(t,v).
\end{equation}

\item The action of the inversion $j$ is given by
\begin{multline}
S_\Lambda(j)f(t,v)=\\c_\Lambda\int_{\Omega\times X}f(t',v')J_\Lambda\left(P(\iota(t',v')^\frac{1}{2})\iota(t,v)\right)\Delta(t')^{|\Lambda|-m_1}\Delta(e-|v'|)^{\lambda_p-m}\prod_{k=1}^{p-1}\Delta(e+v'_k)^{\lambda_k-m}~dv'dt',
\end{multline}
where $J_\Lambda$ is defined by:
\begin{equation}
J_\Lambda(x_1,\cdots,x_p):=\prod_{k=1}^p J_{\lambda_i}(x_i),
\end{equation}
and
\begin{equation}
c_\Lambda=\frac{1}{i^{r|\Lambda|}\Gamma_\Omega(|\Lambda|)I_\Omega(\Lambda)}.
\end{equation}
\end{enumerate}
\end{prop}

\begin{proof}
The proof is similar to Proposition \ref{prop:ModeleStratifJordanEqualRank}. The scalar $c_\Lambda$ is then given by:
\[
c_\Lambda=\frac{1}{p^{r|\Lambda|-n}}\prod_{k=1}^p \frac{1}{\Gamma_\Omega(\lambda_k)i^{r\lambda_k}},
\]
and Corollary \ref{coro:ProductGammaFunctionJordan} ends the computation.
\end{proof}

\subsection{Applications to branching rules}$\ $\\
This paragraph is devoted to a general consideration about branching rules in the setting of case \ref{case:rankEqual} but similar results can be obtained for case \ref{case:tensorProduct}. The goal is to give some hints on how one might find informations on the branching law by studying orthogonal polynomials on $X$ with respect to the measure $\Delta_2(e+v)^{\lambda-m_2}~dv$. The following Lemma assures that the space of polynomials on $X$ is large enough.

\begin{lemme}
The space $Pol[V_1^\bot]$ of polynomials on $V_1^\bot$ is dense in $L^2(X,\Delta_2(e+v)^{\lambda-m_2}~dv)$.
\end{lemme}

\begin{proof}
The stratification space $X$ is bounded and $\Delta_2(e+v)$ is a polynomial hence the measure $\Delta_2(e+v)^{\lambda-m_2}~dv$ is bounded on $X$. Hence the integral 
\[ \int_X e^{(v|v)_2} \Delta_2(e+v)^{\lambda-m_2}~dv  \]
is finite. Theorem 3.1.18 in \cite{DunklXu} then proves the lemma.
\end{proof}
Introduce the subspace $Pol_k(V_1^\bot)$ of polynomials of degree $k$ which are orthogonal to all the polynomials of smaller degree. Thus, we have the decomposition:
\begin{equation}
L^2(X,\Delta_2(e+v)^{\lambda-m_2}~dv)\simeq{\sum_{k\geq 0}}^\oplus Pol_k(V_1^\bot).
\end{equation}
Notice that the space $Pol_k(V_1^\bot)$ is isomorphic to the space of homogeneous polynomials on $V_1^\bot$ of degree $k$ and hence $\dim(Pol_k(V_1^\bot))= \binom{k+n_2-n_1-1}{n_2-n_1-1}$.

The stratification space $X$ is invariant under the action of the maximal compact subgroup $K_1$ of $G_{\Omega_1}$, and so is the measure $\Delta_2(e+v)^{\lambda-m_2}~dv$. Thus, the space $Pol_k(V_1^\bot)$ is stable under the action 
\begin{equation}\label{eq:actionCompactSubgroupPolynomial}
k\cdot P(v)=P(k^{-1}\cdot v).
\end{equation}
This defines a finite dimensional representation of the group $K_1$ which does not need to be irreducible. Suppose that $W_1,\cdots, W_\ell$ are the irreducible components of this representation.

The following proposition gives a hint on why the stratified model is interesting in order to find informations about the branching law. 
\begin{prop}\label{prop:restrictionParabolicVectorValued}
The subspace $L^2_\lambda(\Omega_1)\hat{\otimes}W_k$, with $1\leq k\leq \ell$, is irreducible under the action of $S_\lambda(g)$ for any $g$ in the parabolic subgroup $P_1=G_{\Omega_1}\ltimes N_1$.
\end{prop}

\begin{proof}
From the formulas \eqref{eq:FormulStratifEqualRankConforme} and \eqref{eq:FormulStratifEqualRankTranslation} it is obvious that $L^2_\lambda(\Omega_1)\hat{\otimes}W_1$ is stable under the action of $P_1$.
The irreducibility is a consequence of Mackey theory \cite{Mackey} applied in the case of the semidirect product (a good reference for this situation is \cite{Folland}, Thm.6.43). More precisely, $P_1=G_{\Omega_1} \ltimes N_1$ where $N_1$ is the normal abelian subgroup in $P_1$ composed of translation by an element of $V_1$. The action (by conjugation) of $G_{\Omega_1}$ on $N_1$ is given, for $g\in G_{\Omega_1}$ and $u\in V_1$, by
\[g.\tau_u.g^{-1}=\tau_{g\cdot u}.\]
This action induces an action on the dual group $\widehat{N_1}= \{t\mapsto e^{i(t|b)_1}~|~b\in V_1\}\simeq V_1$. For each $b\in V_1\simeq \widehat{N_1}$, we denote $G_b$ its stabilizer
\[G_b=\{g\in G_{\Omega_1}~|~g\cdot b=b\},\]
and we denote $\mathcal{O}_b$ its orbit
\[\mathcal{O}_b=\{g\cdot b~|~b\in V_1\}.\]
Notice that for $b=e$ we have $G_e=K_1$ and $\mathcal{O}_e=\Omega_1$. For any irreducible representation $\rho$ and any $b\in V_1$, define the representation $b\rho$ of $G_b\ltimes N_1$ by
\[(b\rho)(g,\tau_u)= e^{i(u|b)_1}\rho(g).\]
Finally, for any $b\in V_1$, the $G_{\Omega_1}$-action on $V_1$ gives rise to an homeomorphism from $G_{\Omega_1}/G_b$ to $\mathcal{O}_b$, and then \cite{Folland}, Thm.6.43 says that, for any irreducible representation $\rho$ and any $b\in V_1$, the induced representation $Ind_{G_b\ltimes N_1}^{P_1}(b\rho)$ is an irreducible unitary representation of $P_1$. Applying this result to $b=e$ and $\rho=W_k$, we realize the induced representation on the space of sections over $G_{\Omega_1}/K_1\simeq \Omega_1$ with values in $W_k$ which are square integrable with respect to the measure $\Delta_1(t)^{-m_1}~dt$. As $\Omega_1$ is open this space is identified with the space $L^2(\Omega_1, \Delta_1(t)^{-m_1}~dt)$ which is isomorphic with $L^2_\lambda(\Omega_1)$ using 
\[f \mapsto \Delta(t)^{-\frac{\lambda}{2}}f(t).\]
Using this isomorphism one checks that the action of $P_1$ coincides with formulas \eqref{eq:actionRestrictionStructure} and \eqref{eq:actionRestrictionTranslation}.
\end{proof}

Finally, in order to study the branching law of restrictions of scalar valued holomorphic discrete series, it remains to understand how the inversion $j$ acts on the space $L^2(\Omega_1,\Delta_1(x)^{\lambda-m_1}~dx)\hat{\otimes}W_k$. In the following chapter, we answer this question on some examples for which $K_1$ acts trivially on $Pol_k(V_1^\bot)$ so that the $W_k$ are one dimensional. In this case, Propostion \ref{prop:restrictionParabolicVectorValued} is a reformulation of Proposition  2.13 in \cite{Mollers14}.

\section{Application: The $n$-fold tensor product of holomorphic discrete series of $SL_2(\R)$}\label{sec:ntensorPorduct}

In this section, we are going to prove a bijection between symmetry breaking operators for the restriction of the $n$-fold tensor product of holomorphic discrete series of the unversal covering $\widetilde{SL_2(\R)}$ to the diagonal subgroup and some orthogonal polynomials on the $(n-1)$-dimensional simplex. More precisely, we are going to work in the stratified model described in Proposition \ref{prop:ModeleStratifDirectProduct} in the case $V=\R$ and $G(T_\Omega)=SL_2(\R)$, and use the associativity of the tensor product in order to construct symmetry breaking operators by induction from the $n=2$ case described in \cite{Labriet20}, and we are also going to give an alernative proof using Bessel operators on the Jordan algebra defined in \eqref{def:BesselOperator}. In the holomorphic model, this situation was studied in \cite{Rosen} where the author uses an infinitesimal approach to construct holographic operators and finds symmetry breaking operators in this model by computing the adjoint of the holographic operators.\\

\textit{Notations:} For a vector $x=(x_1,\cdots,x_p)\in \R^n$ and $0\leq p\leq n$, we set the following notations:
\begin{align}
x^{(p)}&=(x_1,\cdots,x_p),
&x_{(p)}&=(x_p,\cdots,x_n),\\
|x|&=\sum_{k=1}^n x_k,
&\|x\|^2&=\sum_{k=1}^n x_k^2.
\end{align}
Moreover, we fix $x_{(n+1)}=0$.

We use the Pochammer symbol, defined for $\lambda \in \C$ and $n\in\N$ by:
\[(\lambda)_n=\lambda(\lambda+1)\cdots(\lambda+n-1)=\frac{\Gamma(\lambda+n)}{\Gamma(\lambda)}.\]

\subsection{Setting}

\paragraph{Holomorphic discrete series for $\widetilde{SL_2(\R)}$}$\ $

We denote by $\Pi$ the Poincaré upper half-plane endowed with the hyperbolic metric, and by $H^2_\lambda(\Pi)$ the weighted Bergman space defined, for $\lambda\in \R$, by:
\[H^2_\lambda(\Pi)=\mathcal{O}\left(\Pi\right)\cap L^2\left(\Pi,y^{\lambda-2}~dxdy\right).\]
It is known that $H^2_\lambda(\Pi)=\{0\}$ for $\lambda \leq 1$, so we suppose that $\lambda>1$. In this case $H^2_\lambda(\Pi)$ admits a reproducing kernel, also called the Bergman kernel, given  by (see, for instance, \cite{Far_Kor}, Prop.XIII.1.2):

 \begin{equation} \label{eq:Noyau_bergmann}
 K_\lambda(z,w)=\frac{\lambda-1}{4\pi}\left(\frac{z-\bar{w}}{2i}\right)^{-\lambda}.
 \end{equation}

For $\lambda\in \N\backslash \{0;1\}$, the holomorphic discrete series representations $\pi_\lambda$ of $SL_2(\R)$ can be realized on $H_\lambda^2(\Pi)$ by the following formula:
\begin{equation}\label{eq:holomorphic_discrete_series}
 \left(\pi_{\lambda}(g)f\right)(z)=(cz+d)^{-\lambda}f\left(\frac{az+b}{cz+d}\right),
\end{equation}
where $g^{-1}=\begin{pmatrix}
a&b\\c&d
\end{pmatrix} \in SL_2(\R),$ and $f \in H_\lambda^2(\Pi)$.
This representation lifts to a unitary and irreducible representation of the universal covering group $\widetilde{SL_2(\R)}$ for $\lambda>1$, through an appropriate choice of the determination of the power function.

We now describe the $L^2$-model for this representation in which the group acts on the space of square integrable functions on $\R^+$ with respect to the measure $t^{\lambda-1}~dt$.
Following the general case, we introduce, for $\lambda>1$, the Laplace transform defined by:
\begin{equation}\label{eq:fourier_transform_SL2}
\mathcal{F} f(z) =\int_0^\infty f(t)e^{izt}t^{\lambda-1}~dt.
\end{equation}
This is a surjective isometry (up to a constant) from $L^2_\lambda(\mathbb{R}^+):=L^2(\mathbb{R}^+,t^{\lambda-1}dt)$ to $H^2_\lambda (\Pi)$ (see \cite{Far_Kor}, Thm. XIII.1.1). More precisely:
\[\|\mathcal{F} f\|^2_{H^2_\lambda(\Pi)}=b(\lambda)\|f\|^2_{L^2_{\lambda}(\R^+)},\] for every $f\in L^2_{\lambda}(\R^+)$, where $b(\lambda)=2^{2-\lambda}\pi\Gamma(\lambda-1)$. Using this transform, the $L^2$-model for the holomorphic discrete series representations of $SL_2(\R)$ is realized on the space $L^2_\lambda(\R^+)$.

\textit{Remark:} In \cite{KP}, the authors used a different formula for the Laplace transform, namely:
\begin{equation}
\mathcal{F} g(z) =\int_0^\infty g(t)e^{izt}~dt.
\end{equation}
This leads to a surjective isometry from $L^2(\mathbb{R}^+,t^{1-\lambda}~dt)$ to $H^2_\lambda (\Pi)$. Both $L^2$-model are related by the isometry $f\in L^2(\R^+,t^{1-\lambda}~dt)\mapsto f(t)t^{1-\lambda}\in L^2_\lambda(\R^+)$. Thus, we may use this map to compare our formulas.

\paragraph{$n$-fold tensor product of holomorphic discrete series representations} $\ $\\

Let $n\geq 2$, $\Lambda=(\lambda_1,\cdots,\lambda_n)$ such that $\lambda_j >1$, and $|\Lambda| =\sum_{i=1}^n \lambda_i$. We consider the tensor product of $n$ holomorphic discrete series representations \eqref{eq:holomorphic_discrete_series} of $\widetilde{SL_2(\R)}$, and we build a basis of the space of symmetry breaking operators in the stratified model. 

For the $n=2$ case the branching law is given as follows
\begin{equation}
\pi_{\lambda'}\otimes \pi_{\lambda''}\simeq {\sum_{l\geq 0}}^\oplus \pi_{\lambda'+\lambda''+2l}.
\end{equation}

Notice that this branching law is multiplicity free and it involves only discrete series representations (see \cite{Repka} and  \cite{Molch}). A recursion on $n$ gives the following branching in the general case:
\begin{equation}\label{eq:branching_rule_n_SL2R}
\bigotimes_{i=1}^n \pi_{\lambda_i} \simeq \left.\sum_{k\in\N}\right.^\oplus\binom{n+k-2}{n-2} \pi_{|\Lambda|+2k}.
\end{equation}

As mentioned before, we are going to work in the stratified model introduced in Proposition \ref{prop:ModeleStratifDirectProduct} in the case $V=\R$ and $\Omega=\R^+$. We are now going to describe our setting from the point of view of section \ref{sec:defintionDiffeo}. 

Let $V_n=\R^n$, $\Omega_n=\left(\R^+\right)^n$, and $V_1=\R\cdot(1,\cdots,1)$, $\Omega_1=\R^+\cdot (1,\cdots,1)$. Hence, in the $L^2$-model the $n$ fold tensor product representation acts on the space $L^2\left((\R^+)^n,\prod_{i=1}^n x_i^{\lambda_i-1}~dx_i\right)$. The diagonal embedding of $V_1$ in $V_n$  satisfies the setting of section \ref{sec:defintionDiffeo}.

We have:
\[V_1^\bot=\left\{v\in V_n~| ~|v|=0\right\}.\]
Hence, the stratifiaction space is 
\begin{align} 
 X_{1n}&=\left\{v\in V_n~|~|v|=0,~ 1+v_i>0\right\}\\
 &=\left\{(v_1,\cdots,v_{n-1},-\sum_{i=1}^{n-1} v_i)~|~ 1+v_i>0,~1-\sum_{i=1}^{n-1} v_i>0\right\}
 \end{align}
Using this parametrisation of $X_{1n}$, we define the diffeomorphism $\iota_{1n}$, as in \eqref{def:DiffeoGeneral}:
\begin{equation}\label{eq:diffeo_iota_1n}
\iota_{1n}~:~(t,v)\in \Omega_1\times X_{1n}\mapsto \left(\frac{t(1+v_1)}{n},\cdots,\frac{t(1+v_{n-1})}{n},\frac{t(1-\sum_{i=1}^{n-1} v_i)}{n}\right)\in\Omega_n.
\end{equation}
We find $|Jac(\iota_{1n})|=\frac{t^{n-1}}{n^{n-1}}$, and hence Proposition \ref{prop:Hilbert_space_isom_tensorProduct} corresponds to the following Hilbert space isomorphism induced by the pullback $\iota_{1n}^*f= f\circ \iota_{1n}$:
\begin{equation}
L^2\left(\Omega_n,\prod_{i=1}^n x_i^{\lambda_i-1}~dx_i\right)\simeq L^2(\Omega_1,t^{|\Lambda|-1}~dt)\hat{\otimes}L^2\left(X_{1n},(1-\sum_{i=1}^{n-1}v_i)^{\lambda_n-1}\prod_{i=1}^{n-1}(1+v_i)^{\lambda_i-1}~dv_i\right),
\end{equation}
and we have $\|f\circ \iota_{1n}\|^2=n^{n-1}\|f\|^2$.

Introduce the $n$-dimensional simplex $D_n$ defined by
\begin{equation}\label{def:simplex}
D_n=\{(x_1,\cdots,x_n)~|~x_i>0, 1 -|x|>0\}.
\end{equation}
Then the map $\Phi_n$ from $X_{1n}$ to $D_{n-1}$ defined by
\begin{equation}
\Phi_n(v_1,\cdots,v_n)=\left(\frac{1+v_1}{n},\cdots,\frac{1+v_{n-1}}{n}\right),
\end{equation}
is a diffeomorphism, and it induces the following isomorphism of Hilbert spaces
\begin{equation}
L^2\left(X_{1n},(1-\sum_{i=1}^{n-1}v_i)^{\lambda_n-1}\prod_{i=1}^{n-1}(1+v_i)^{\lambda_i-1}~dv_i\right)\simeq L^2\left(D_{n-1},(1-|x|)^{\lambda_n-1}\prod_{i=1}^{n-1}x_i^{\lambda_i-1}~dx_i\right).
\end{equation}

We are going to prove a one-to-one correspondence between symmetry breaking operators for the restriction of the $n$-fold tensor product, and orthogonal polynomials on the $(n-1)$-dimensional simplex with respect to the measure $(1-|x|)^{\lambda_n-1}\prod_{i=1}^{n-1}x_i^{\lambda_i-1}~dx_i$. More precisely, we prove the following:
\begin{theorem}\label{thm:nfoldTensorProduct}
Let $k\in \N$, $\Lambda=(\lambda_1,\cdots,\lambda_n)$ such that $\lambda_j>1$ and let $Pol_k(D_{n-1})$ be the space of polynomials in $n-1$ variables of degree $k$ which are orthogonal to all polynomials of smaller degree with respect to the inner product on $L^2(D_{n-1},(1-|v|)^{\lambda_n-1}\prod_{i=1}^{n-1}v_i^{\lambda_i-1}~dv)$. 

For a polynomial $P\in Pol_k(D_{n-1})$ define, on $L^2(\R^+,t^{|\Lambda|-1}dt)\hat{\otimes}L^2(D_{n-1},(1-|v|)^{\lambda_n-1}\linebreak \prod_{i=1}^{n-1}v_i^{\lambda_i-1}~dv)$, the operator $\Psi_k^\Lambda(P)$ by
\begin{equation}\label{def:symBreakingSimplex}
\Psi_k^\Lambda(P)f(t)=t^{-k}\int_{D_{n-1}}f(t,v)P(v) (1-|v|)^{\lambda_n-1}\prod_{i=1}^{n-1}v_i^{\lambda_i-1}~dv.
\end{equation}

Then $\Psi_k^\Lambda(P)\in Hom_{\widetilde{SL_2(\R)}}(\bigotimes_{i=1}^n \pi_{\lambda_i} ,\pi_{|\Lambda|+2k})$ if and only if $P\in Pol_k(D_{n-1})$.
\end{theorem}

This theorem shows that the decomposition 
\begin{equation}
L^2(\R^+,t^{|\Lambda|-1}~dt)\hat{\otimes}L^2(D_{n-1},(1-|v|)^{\lambda_n-1}\prod_{i=1}^{n-1}v_i^{\lambda_i-1}~dv)\simeq {\sum_{k\geq 0}}^\oplus L^2(\R^+,t^{|\Lambda|-1}~dt)\hat{\otimes}Pol_k(D_{n-1}),
\end{equation}
is the isotypic decomposition, in the stratified model, of the $n$-fold tensor product of holomorphic discrete series representations of $SL_2(\R)$ where $L^2(\Omega_1,t^{|\Lambda|-1}~dt)\hat{\otimes}Pol_k(D_{n-1})$ is the isotypic component of the representation $\pi_{|\Lambda|+k}$. Notice that the $n=2$ case was already treated in \cite{Labriet20} and is related to the classical Jacobi transform (see Proposition 4.2).
A direct computation then proves the following corollary which describes the holographic operators in this setting.
\begin{coro}
Let $k\in \N$, $\Lambda=(\lambda_1,\cdots,\lambda_n)$ such that $\lambda_j>1$ and $P\in Pol_k(D_{n-1})$, the operator $\Phi_k^\Lambda(P)$ defined, on $L^2_{|\Lambda|+2k}(\R^+)$ by
\begin{equation}
\Phi_k^\Lambda(P)f(t,v)=P(v)t^kf(t),
\end{equation}
belongs to $\Hom_{\widetilde{SL_2(\R)}}(\bigotimes_{i=1}^n \pi_{\lambda_i} ,\pi_{|\Lambda|+2k})$. More precisely, we have 
\begin{equation}
\Phi_k^\Lambda(P)=\left(\Psi_k^\Lambda(P)\right)^*,
\end{equation}
where ${}^*$ denotes the adjoint operator.
\end{coro}

\subsection{A family of orthogonal polynomials on the simplex}
In this section, we describe a basis of orthogonal polynomials on the simplex $D_n$  necessary for the proof of Theorem \ref{thm:nfoldTensorProduct}. Let $\Lambda=(\lambda_1,\cdots, \lambda_{n+1})$ with $\lambda_i>0$, some explicit polynomial Hilbert basis on the space $L^2(D_n,(1-|v|)^{\lambda_{n+1}-1}\prod_{i=1}^n v_i^{\lambda_i-1}~dv_i)$ is given in \cite{DunklXu} but we are going to describe yet another one which is adapted to our geometry.

We consider the family of polynomials, indexed by $\textbf{k}=(k_1,\cdots,k_n)\in \N^n$, given by:
\begin{equation}\label{def:polOrthoSimplex}
^{(n)}R_{\textbf{k}}^\Lambda(x_1,\cdots,x_n):=P_{k_n}^{\lambda_{n+1}-1,\alpha_n-1}(2|x|-1)\prod_{i=1}^{n-1}(|x^{(i)}|+x_{i+1})^{k_i}P_{k_i}^{\lambda_{i+1}-1,\alpha_i-1}\left(\frac{|x^{(i)}|-x_{i+1}}{|x^{(i)}|+x_{i+1}}\right)
\end{equation}
where $\alpha_i=|\Lambda^{(i)}|+2|\textbf{k}^{(i-1)}|$ and $\{P_k^{\alpha,\beta}\}_n$ denotes the family of Jacobi polynomials (see \ref{eq:Jacobi_def}).

If the context is clear one can drop the superscript $(n)$ and write $R_{\textbf{k}}^\Lambda$ instead of $^{(n)}R_{\textbf{k}}^\Lambda$. 

\begin{theorem}\label{thm:polOrthoSimplex}
The family $\{R_\textbf{k}^\Lambda\}$, with $k\in \N^n$ and $\Lambda\in \N_+^n$, forms a Hilbert basis of the space $L^2(D_n,(1-|v|)^{\lambda_{n+1}-1}\prod_{i=1}^n v_i^{\lambda_i-1}~dv_i)$.
\end{theorem}

\begin{proof}
We are going to prove this result by induction on $n$. For $n=2$, the result amounts to the orthogonality of the Jacobi polynomials. In order to use the induction, we introduce the diffeomorphism $\phi$ from $D_{n-1}\times (-1,1)$ to $D_n$ defined by
\begin{equation}
\phi((y_1,\cdots,y_{n-1}),u)=\left(\frac{y_1(1+u)}{2},\frac{y_1(1-u)}{2},y_2,\cdots,y_{n-1}\right).
\end{equation}
Using this diffeomorphism one sees that
\begin{equation}
(1-|x|)^{\lambda_{n+1}-1}\prod_{i=1}^n x_i^{\lambda_i-1}~dx_i=\frac{1}{2^{\lambda_1+\lambda_2-1}}(1-|y|)^{\lambda_{n+1}-1} y_1^{\lambda_1+\lambda_2-1}\prod_{i=2}^{n-1}y_i^{\lambda_{i+1}-1}~dy_i.
\end{equation}
We also introduce the notations
\begin{equation}
\tilde{\Lambda}=(\lambda_1+\lambda_2+2 k_1,\lambda_3,\cdots,\lambda_{n+1}),~~\tilde{\textbf{k}}=(k_2,\cdots,k_n).
\end{equation}
First, the following identity holds:
\begin{align}\label{eq:identityPolOrthoSimplex}
&^{(n)}R_{\textbf{k}}^\Lambda\left(\phi((y_1,\cdots,y_{n-1}),u)\right)\nonumber\\
&=P_{k_n}^{\lambda_{n+1}-1,|\Lambda|+2|k^{(n)}|-1}(2|y^{(n-1)}|-1)P_{k_1}^{\lambda_2-1,\lambda_1-1}(u)y_1^{k_1}\nonumber\\
&~~\times\prod_{i=2}^{n-1}{y^{(i)}}^{k_i}P_{k_i}^{\lambda_{i+1}-1,|\Lambda^{(i)}|+2|k^{(i-1)}|-1}\left(\frac{|y^{(i-1)}|-y_i}{|y^{(i-1)}|+y_i}\right)\nonumber\\
&={}^{(n-1)}R_{\tilde{\textbf{k}}}^{\tilde{\Lambda}}(y_1,\cdots,y_{n-1})P_{k_1}^{\lambda_2-1,\lambda_1-1}(u)y_1^{k_1}.
\end{align}
We now consider the inner product of two polynomials $^{(n)}R_{\textbf{k}}^\Lambda$ and $^{(n)}R_{\textbf{k'}}^{\Lambda'}$.
\begin{align*}
I&:=\int_{D_n} {}^{(n)}R_{\textbf{k}}^\Lambda(x){}^{(n)}R_{\textbf{k'}}^{\Lambda'}(x)(1-|x|)^{\lambda_{n+1}-1}\prod_{i=1}^n x_i^{\lambda_i-1}~dx_i\\
&=\frac{1}{2^{\lambda_1+\lambda_2-1}}\int_{D_{n-1}}\int_{-1}^1 {}^{(n)}R_{\textbf{k}}^\Lambda(\phi(y,u)){}^{(n)}R_{\textbf{k'}}^{\Lambda'}(\phi(y,u))\\
&~~~~\times(1-|y|)^{\lambda_{n+1}-1}y_1^{\lambda_1+\lambda_2+k_1+k_1'-1}(1-u)^{\lambda_2-1}(1+u)^{\lambda_1-1}\prod_{i=2}^{n-1}y_i^{\tilde{\lambda_i}-1}~dy_idu\\
&=\frac{1}{2^{\lambda_1+\lambda_2-1}}\int_{D_{n-1}}{}^{(n-1)}R_{\tilde{\textbf{k}}}^{\tilde{\Lambda}}(y){}^{(n-1)}R_{\tilde{\textbf{k'}}}^{\tilde{\Lambda'}}(y)(1-|y|)^{\tilde{\lambda}_n-1}y_1^{\lambda_1+\lambda_2+k_1+k_1'-1}\prod_{i=2}^{n-1}y_i^{\tilde{\lambda_i}-1}~dy_i\\
&~\times \int_{-1}^1P_{k_1}^{\lambda_2-1,\lambda_1-1}(u)P_{k'_1}^{\lambda_2-1,\lambda_1-1}(u)(1-u)^{\lambda_2-1}(1+u)^{\lambda_1-1}~du.
\end{align*}
The orthogonality of the Jacobi polynomials in $L^2((-1,1),(1-u)^{\lambda_2-1}(1+u)^{\lambda_1-1}~du)$ implies $I=0$ if $k_1\neq k'_1$. If $k_1= k'_1$ we have
\[
I=\frac{\|P_{k_1}^{\lambda_2-1,\lambda_1-1}\|^2}{2^{\lambda_1+\lambda_2-1}}\int_{D_{n-1}}{}^{(n-1)}R_{\tilde{\textbf{k}}}^{\tilde{\Lambda}}(y){}^{(n-1)}R_{\tilde{\textbf{k'}}}^{\tilde{\Lambda'}}(y)(1-|y|)^{\tilde{\lambda}_n-1}\prod_{i=1}^{n-1}y_i^{\tilde{\lambda_i}-1}~dy_i,
\]
which is zero if $\tilde{\textbf{k}}\neq \tilde{\textbf{k'}}$ by assumption.
\end{proof}

\textit{Remark:} For $n=3$, the diffeomorphism $\phi$ from $D_2$ to $D_3$ used in the proof of Theorem \ref{thm:polOrthoSimplex} corresponds to the picture \ref{fig:tetraedre1}. In \cite{DunklXu}, another basis of orthogonal polynomials is given, namely:

\begin{equation}
Q_{\textbf{k}}^\Lambda(x)=\prod_{i=1}^n\left( \frac{1-|x^{(i)}|}{1-|x^{(i-1)}|}\right)^{2|\textbf{k}^{(i+1)}|}P_{k_i}^{a_i,b_i}\left(\frac{2x_i}{1-|x^{(i-1)}|}-1 \right),
\end{equation}
with $a_i=2|k^{(i+1)}|+|\Lambda^{(i+1)}|+\frac{n}{2}-i-1$ and $b_i=\lambda_i-1$. These polynomials are obtained, for $n=3$, using the diffeomorphism corresponding to the embedding in figure \ref{fig:tetraedre2}.

\begin{figure}[!h]\label{fig:tetraedre}
\centering
\begin{subfigure}[b]{0.3\textwidth}
\includegraphics[scale=0.4]{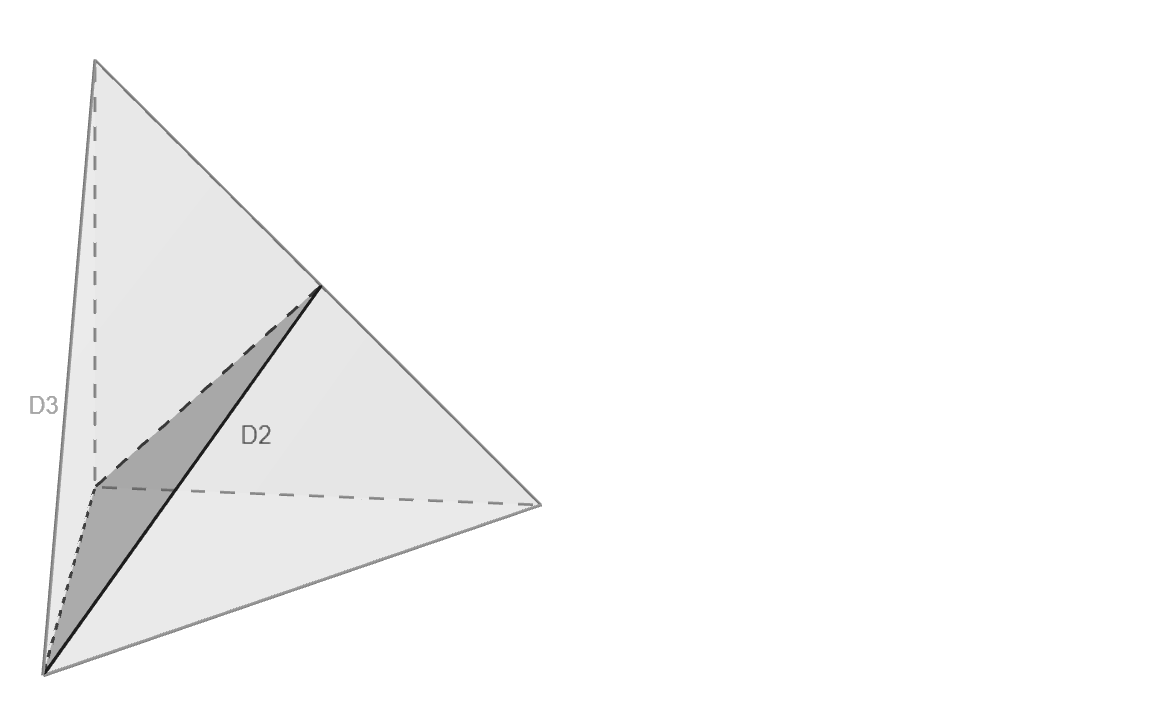}
\subcaption{Diagonal embedding}
\label{fig:tetraedre1}
\end{subfigure}
\begin{subfigure}[b]{0.3\textwidth}
\includegraphics[scale=0.4]{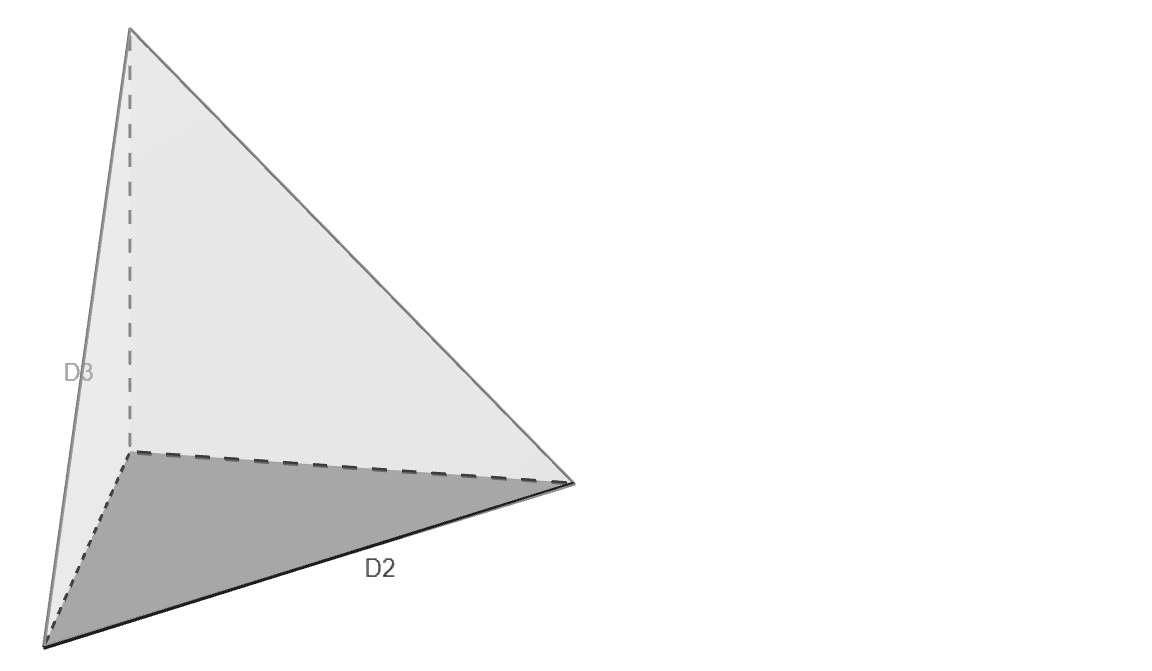}
\subcaption{"Side" embedding }
\label{fig:tetraedre2}
\end{subfigure}
\caption{Embeddings of $D_2$ into $D_3$}
\end{figure}

Now we are going to describe how we get the family of polynomials \eqref{def:polOrthoSimplex}. For this, we define the diffeomorphism $\phi_n$ from $(0;1)\times (-1;1)^{n-1}$ to $D_n$ by 
\begin{equation}
\phi_n(t,u_1,\cdots,u_{n-1})=(x_1,\cdots,x_n),
\end{equation}
where $x_1=\frac{1}{2^{n-1}}t\prod_{i=1}^{n-1}(1+u_i)$, $x_n=\frac{1}{2}(1-u_{n-1})t$ and for $2\leq k\leq n-1,~x_k=t\prod_{i=k}^{n-1}\left(\frac{1+u_i}{2}\right)\left(\frac{1-u_{k-1}}{2}\right)$.
This map is the iteration of the map $\phi$ from the proof of Theorem \ref{thm:polOrthoSimplex} where $(0;1)$ corresponds to $D_1$.
Its inverse is given by the formula
\begin{equation}
\phi_n^{-1}(x_1,\cdots,x_n)=\left( |x|,,\frac{x_1-x_2}{x_1+x_2},\cdots,\frac{|x^{(k)}|+x_{k+1}}{|x^{(k)}|+x_{k+1}},\cdots,\frac{|x^{(n-1)}|-x_n}{|x^{(n-1)}|+x_n}\right).
\end{equation}

\begin{prop}
The Jacobian of $\phi_n$ is given by
\begin{equation}
|\operatorname{Jac}\phi_n|=\frac{t^{n-1}}{2^{\frac{n(n-1)}{2}}}\prod_{i=1}^{n-1}(1+u_i)^{i-1}.
\end{equation}
\end{prop}

\begin{proof}
A direct computation shows that:
\[
|\operatorname{Jac}\phi_n|=\frac{t^{n-1}}{2^{n-1}\prod_{i=1}^{n-1}2^{n-i}}|\det(A)|,
\]
with \[A=\begin{pmatrix}
\prod_{i=1}^{n-1}(1+u_i)& \prod_{i=2}^{n-1}(1+u_i)&\cdots &\prod_{i\neq j}(1+u_i)&\cdots&\prod_{i=1}^{n-2}(1+u_i)\\
\prod_{i=2}^{n-1}(1+u_i)(1-u_1)& -\prod_{i=2}^{n-1}(1+u_i)&&\vdots&&\vdots\\
\vdots&0&&\vdots& &\vdots\\
\vdots&\vdots & &-\prod_{i=j}^{n-1}(1+u_i)& & \vdots\\
\vdots&\vdots & &0& &\vdots\\
(1+u_{n-1})(1-u_{n-2})&\vdots& &\vdots& &-(1+u_{n-1})\\
1-u_{n-1}&0& &0&0&-1
\end{pmatrix}.\]
Hence, the proposition is equivalent to 
\[|\det(A)|=2^{n-1}\prod_{i=1}^{n-1}(1+u_i)^{i-1}.\]
This is proved by introducing the matrices $A_k\in M_{n-k+1}$ for $1\leq k \leq n-2$
\[A_k=\begin{pmatrix}
\prod_{i=k}^{n-1}(1+u_i)& \prod_{i=k+1}^{n-1}(1+u_i)&\cdots &\prod_{i\neq j}(1+u_i)&\cdots&\prod_{i=k}^{n-2}(1+u_i)\\
\prod_{i=k+1}^{n-1}(1+u_i)(1-u_1)& -\prod_{i=k+1}^{n-1}(1+u_i)&&\vdots&&\vdots\\
\vdots&0&&\vdots& &\vdots\\
\vdots&\vdots & &-\prod_{i=j}^{n-1}(1+u_i)& & \vdots\\
\vdots&\vdots & &0& &\vdots\\
(1+u_{n-1})(1-u_{n-2})&\vdots& &\vdots& &-(1+u_{n-1})\\
1-u_{n-1}&0& &0&0&-1
\end{pmatrix}.\]
and developing this determinant with respect to the second column one gets
\[\det A_{k-1}=2\prod_{i=k}^{n-1}(1+u_i)\det A_k.\]
We finally get the result by recursion on $k$, and the fact that $\det A_{n-2}=4(1+u_{n-1})$.
\end{proof}

The pullback by the diffeomorphism $\phi_n$ leads to the following corollary which gives a hint on how the polynomials \eqref{def:polOrthoSimplex} are build.

\begin{coro}
\begin{multline}
L^2(D_n,(1-|v|)^{\lambda_{n+1}-1}\prod_{i=1}^n x_i^{\lambda_i-1}~dx_i)\simeq \\
L^2((0;1),t^{|\Lambda|-1}~dt)\hat{\otimes}\bigotimes_{i=1}^n L^2((-1;1),(1+u)^{|\lambda^{(i)}|-1}(1-u)^{\lambda_{i+1}-1}~du)
\end{multline}

\end{coro}
Finally, one recovers the family $\{R^\Lambda_\textbf{k}\}$ using the fact that Jacobi polynomials are orthogonal with respect to the measure $(1+u)^{|\lambda^{(i)}|-1}(1-u)^{\lambda_{i+1}-1}~du$ on $(-1;1)$.

\subsection{A basis for symmetry breaking operators}

We are now ready to prove Theorem \ref{thm:nfoldTensorProduct} with the help of polynomials \eqref{def:polOrthoSimplex}. For this, we use the associativity of the tensor product which gives the following equivalences of unitary representations
\begin{equation}\label{eq:Branching_Law_Tensor_Product}
\bigotimes_{i=1}^n \pi_{\lambda_i} \simeq \left.\sum_{l\in\N}\right.^\oplus\pi_{\lambda_1+\lambda_2+2l}\otimes\bigotimes_{i=3}^n \pi_{\lambda_i}\simeq {\sum_{l\in\N}}^\oplus \binom{n+k-2}{n-2}\pi_{|\Lambda|+2k}.
\end{equation}

We will prove Theorem \ref{thm:nfoldTensorProduct} by induction on $n$ from the $n=2$ case. To do so, we recall the following Proposition (see \cite{Labriet20}, Prop. 4.2) 

\begin{prop}\label{prop:n=2case}
The operator $T_k^{\lambda_1,\lambda_2}$ from $L^2(\R^+,t^{\lambda_1+\lambda_2-1}~dt)\hat{\otimes} L^2((-1,1), (1+v)^{\lambda_1-1}\linebreak(1-v)^{\lambda_2-1}~dv)$ to $L^2(\R^+,t^{\lambda_1+\lambda_2+2k-1}~dt)$ defined by
\begin{equation}
T_k^{\lambda_1,\lambda_2}f(t)=t^{-k}\int_{-1}^1 f(t,v)P_k^{\lambda_1-1,\lambda_2-1}(v)(1-v)^{\lambda_1-1}(1+v)^{\lambda_2-1}~dv,
\end{equation}
is a symmetry breaking operator in the stratified model between $\pi_{\lambda_1}\otimes \pi_{\lambda_2}$ and $\pi_{\lambda_1+\lambda_2+k}$.
\end{prop}

Notice that we have the following chain of subalgebras:
\begin{equation}
V_1\subset V_{n-1}\subset V_n,
\end{equation}
where $V_{n-1}=\R^{n-1}$ is viewed as a Jordan subalgebra of $V_n$ via the embedding $\linebreak\eta(x_1,\cdots,x_{n-1})=(x_1,x_1,x_2,\cdots,x_{n-1})$. Notice that this is a unital embedding of $V_{n-1}$ into $V_n$ but it does not satisfies condition \eqref{eq:conditionScalarProduct}. However, notice that $\eta$ is the composition of the map $x \mapsto (x,x)$ from $\R$ to $\R^2$ and of the map $((x_1,x_2),(x_3,\cdots,x_n))\mapsto (x_1,\cdots, x_n)$ from $\R^2\times\R^{n-2}$ to $\R^n$ which both satisfies condition \eqref{eq:conditionScalarProduct}.

Thus, we can use the stratification, associated to the pair $V_{n-1}\subset V_n$, described in section \ref{sec:defintionDiffeo}, however the diffeomorphism will not be given by formula \eqref{def:DiffeoGeneral}. The stratification space $X=\{v\in V_{n-1}^\bot~|~e+v\in \Omega_n\}$ defined in \eqref{def:SubsetXGeneral} is isomorphic to $(-1,1)$ and hence the stratification is given by the diffeomorphism $\iota_{n-1,n}$ from $\Omega_{n-1}\times (-1;1)$ to $\Omega_n$ defined by:
\begin{equation}
\iota_{n-1,n}(x_1,\cdots,x_{n-1},v)=\left(\frac{x_1(1+v)}{2},\frac{x_1(1-v)}{2},x_3,\cdots,x_n\right).
\end{equation}

In this setting, we can use induction over $n$ to prove Theorem \ref{thm:nfoldTensorProduct}.

\begin{proof}[Proof of Theorem \ref{thm:nfoldTensorProduct}]
Let $k\in \N$ and $\textbf{k}=(k_1,\cdots,k_{n-1})\in\N^{n-1}$ such that $|\textbf{k}|=k$. Notice that the dimension of $Pol_k(D_{n-1})$ is equal to $\binom{n+k-2}{n-2}$ which is the multiplicity of $\pi_{|\lambda|+2k}$ in the branching law of $\bigotimes_{i=1}^n \pi_{\lambda_i} $ (see \eqref{eq:branching_rule_n_SL2R}).\\
The map $P\mapsto\Psi_k^\Lambda(P)$ defined by \eqref{def:symBreakingSimplex} is linear and injective so we only need to prove the intertwining property for $\Psi_k^\Lambda(P)$ on the orthogonal basis $\left\{ {}^{(n-1)}R_\textbf{k}^{\Lambda}\right\}$ of $Pol_k(D_{n-1})$ defined by \eqref{def:polOrthoSimplex}. Set the notation ${}^{(n)}\Psi_{\textbf{k}}^\Lambda$ for the operator $\Psi_k^\Lambda({}^{(n-1)}R_\textbf{k}^{\Lambda})$.

The composition of the diffeomorphisms $\iota_{1n}$ and $\Phi_n$ from $\Omega_1\times D_{n-1}$ to $\Omega_n$ will be denoted $\theta_{n}$ and is given by 
\begin{equation}
\theta_n(t,v_1,\cdots,v_{n-1})=(tv_1,\cdots,tv_{n-1},t(1-|v|)).
\end{equation}

We are going to prove that the following diagram is commutative.
\[\xymatrix{
    L^2(\Omega_n,\prod_{i=1}^nx_i^{\lambda_i-1}~dx_i) \ar[rr]^{{\theta_n}^*}\ar[d]_{{\iota_{n-1,n}}^*}  && L^2(\Omega_1,t^{|\Lambda|-1}~dt)\hat{\otimes}L^2(D_{n-1},d\mu_\Lambda(v)) \ar[d]^{^{(n)}\Psi_{\textbf{k}}^\Lambda}\\
    L^2(\Omega_{n-1},\prod_{i=1}^n x_i^{\nu_i-1}~dx_i)\hat{\otimes}L^2((-1,1),d\eta_{\lambda_1,\lambda_2}(v))\ar[d]_{\frac{1}{2^{\lambda_1+\lambda_2-1}}T_{k_1}^{\lambda_2,\lambda_1}} && L^2(\Omega_1,t^{|\Lambda|+2k-1}~dt)\\
 L^2(\Omega_{n-1}, \prod_{i=1}^{n-1}t_i^{\tilde{\lambda_i}-1}~dt_i) \ar[rr]_{{\theta_{n-1}}^*}&& L^2(\Omega_1,t^{|\tilde{\Lambda}|-1}~dt)\hat{\otimes}L^2(D_{n-2},d\mu_{\tilde{\Lambda}}(v))\ar[u]_{ ^{(n-1)}\Psi_{\tilde{\textbf{k}}}^{\tilde{\Lambda}}}
  }\]
where $\tilde{\Lambda}=(\lambda_1+\lambda_2+2k_1,\lambda_3,\cdots,\lambda_n)$, $\nu=(\lambda_1+\lambda_2,\lambda_3,\cdots,\lambda_n)$, $\tilde{\textbf{k}}=(k_2,\cdots,k_{n-1})$, $\eta_{\lambda_1,\lambda_2}(v)=(1-v)^{\lambda_2-1}(1+v)^{\lambda_1-1}~dv$ and $d\mu_\Lambda(v)=(1-|v|)^{\lambda_{n+1}-1}\prod_{i=1}^n v_i^{\lambda_i-1}~dv_i$.

First notice, for $t\in\Omega_1,~y\in D_{n-2},~v\in(-1;1)$, the following equality:
\begin{align*}
\iota_{n-1,n}(\theta_{n-1}(t,y),v)&=\left(\frac{ty_1(1+v)}{2},\frac{ty_1(1+v)}{2},ty_3,\cdots,ty_n\right)\\
&=\theta_{n}\left(t,\left(\frac{y_1(1+v)}{2},\frac{y_1(1+v)}{2},y_3,\cdots,y_n\right)\right)\\
&=\theta_{n}(t,\phi(y,v)).
\end{align*}

This gives for $f\in L^2(\Omega_n,\prod_{i=1}^nt_i^{\lambda_i-1}~dt_i)$ and $t\in\Omega_1$
\begin{align*}
(I)&:={}^{(n)}\Psi_{\textbf{k}}^\Lambda\circ {\theta}^* f(t)\\
&=t^{-k}\int_{D_{n-1}}f(\theta_n(t,v)) {}^{(n-1)}R_\textbf{k}^{\Lambda}(v) d\mu_{\Lambda}(v)\\
&=\frac{t^{-k}}{2^{\lambda_1+\lambda_2-1}}\int_{D_{n-2}}\int_{-1}^1 f(\theta_n(t,\phi(y,u))){}^{(n-1)}R_\textbf{k}^{\Lambda}(\phi(y,u))d\mu_{\tilde{\Lambda}}(y)d\eta_{\lambda_1,\lambda_2}(u).
\end{align*}

Using the identity \eqref{eq:identityPolOrthoSimplex}, we get:
\begin{equation*}
(I)=\frac{t^{-k}}{2^{\lambda_1+\lambda_2-1}}\int_{D_{n-2}}\int_{-1}^1 f(\theta_n(t,\phi(y,u))){}^{(n-2)}R_{\tilde{\textbf{k}}}^{\tilde{\Lambda}}(y)y_1^{k_1}d\mu_{\nu}(y)d\eta_{\lambda_1,\lambda_2}(u).
\end{equation*}

On the other side, we have, for $t\in \Omega_1,~y\in D_{n-2}$:
\begin{multline*}
{\theta_{n-1}}^*\circ T_{k_1}^{\lambda_2,\lambda_1}\circ {\iota_{n-1,n}}^*f(t,y)=\\(ty_1)^{-k_1}\int_{-1}^1f(\iota_{n-1,n}(\theta_{n-1}(t,y),v))P_{k_1}^{\lambda_2-1,\lambda-1}(v)(1+v)^{\lambda_1-1}(1-v)^{\lambda_2-1}~dv.
\end{multline*}

Finally, for $t\in\Omega_1$:
\begin{align*}
(II)&:= ^{(n-1)}\Psi_{\tilde{\textbf{k}}}^{\tilde{\Lambda}}\circ {\theta_{n-1}}^*\circ T_{k_1}^{\lambda_2,\lambda_1}\circ {\iota_{n-1,n}}^*f(t)\\
&=t^{-(k_1+|\tilde{\textbf{k}}|)}\int_{D_{n-2}}\int_{-1}^1 f(\iota_{n-1,n}(\theta_{n-1}(t,y),v)){}^{(n-2)}R_{\tilde{\textbf{k}}}^{\tilde{\Lambda}}(y)y_1^{-k_1}P_{k_1}^{\lambda_2-1,\lambda-1}(u)~d\mu_{\tilde{\Lambda}}(y)d\eta_{\lambda_1,\lambda_2}(u)\\
&=t^{-k}\int_{D_{n-2}}\int_{-1}^1 f(\theta_n(t,\phi(y,u))){}^{(n-2)}R_{\tilde{\textbf{k}}}^{\tilde{\Lambda}}(y)y_1^{k_1}P_{k_1}^{\lambda_2-1,\lambda-1}(u)~d\mu_{\nu}(y)d\eta_{\lambda_1,\lambda_2}(u).
\end{align*}
Hence, $(II)=2^{\lambda_1+\lambda_2-1}(I)$.

Induction over $n$ proves that $^{(n)}\Psi_{\textbf{k}}^\Lambda$ is an intertwining operator. The $n=2$ case corresponds to Proposition \ref{prop:n=2case}.
\end{proof}

\paragraph{A proof of Theorem \ref{thm:nfoldTensorProduct} via Bessel operators}$\ $\\
In this paragraph, we give an alternative proof of Theorem \ref{thm:nfoldTensorProduct} based on the study of the Bessel operator $\mathcal{B_\lambda}$ defined in \eqref{def:BesselOperator}. This approach is independent from the results of \cite{KP}. 

First, let $N_1=\left\{ \begin{pmatrix}1&b\\0&1\end{pmatrix}|b\in \R\right\}$ be the subgroup of translations in $ SL_2(\R)$. The isotropy subgroup $K_1$ of $e=(1,1)$ under the action of $G_{\Omega}=\left\{\begin{pmatrix}
a&0\\0&a^{-1}
\end{pmatrix}| a\in\R^+\right\}$ is trivial, and so is the action of $K_1$ on $Pol_k(D_{n-1})$ defined in \eqref{eq:actionCompactSubgroupPolynomial}.
Hence, Proposition \ref{prop:restrictionParabolicVectorValued} proves that the space $L^2_{|\Lambda|+2k}(\R^+)\hat{\otimes}\C\cdot P$ is irreducible under the actions of $N_1$ and $G_{\Omega_{n-p}}$ for all $P\in Pol_k(D_{n-1})$.

We denote by $\n_1$ the Lie algebra of $N_1$ and $\mathfrak{l}_1$ the Lie algebra of $G_{\Omega}$. The Lie algebra $\mathfrak{sl}_2$ of $SL_2(\R)$ admits the Gelfand-Naimark decomposition \eqref{eq:decompositionLieAlgConformal}:
\begin{equation}
\mathfrak{sl}_2 =\n_1\oplus \mathfrak{l}_1 \oplus \bar{\n}_1.
\end{equation}
Hence, we are left to study the action of $\bar{\n}_1$. We have $\bar{\n}_1\subset \bar{\n}_2$, so that in the $L^2$-model the action of $X=(x,\cdots,x)\in\bar{\n}_1\simeq V_1$ is given by the operator (see Proposition \ref{prop:LieAlgebraAction}):
\begin{equation}
i(X|\mathcal{B}_\Lambda f(X))=ix\mathcal{B}_\Lambda f(X),
\end{equation}
where $\mathcal{B}_\Lambda$ is the second order differential operator defined by: 
\begin{equation}
\mathcal{B}_\Lambda f(x)=\sum_{i=1}^n \mathcal{B}_{\lambda_i} f(x)=\sum_{i=1}^n x_i \derptwo{f}{x_i}+\lambda_i\derpone{f}{x_i},
\end{equation}
for $f$ a smooth function in $L^2((\R^+)^n,\prod_{i=1}^n x_i^{\lambda_i-1})$ and  $\mathcal{B}_{\lambda}$ is the Bessel operator \eqref{def:BesselOperator} for the rank one Jordan algebra.

We use the diffeomorphism $\iota_{1n}:\R^+\times D_{n-1}\to (\R^+)^n$ defined in \eqref{eq:diffeo_iota_1n} to get the expression of the Bessel operator in the stratified model, and we get the following Proposition.
\begin{prop}\label{prop:EqualityBesselOperatorTensor}
In the coordinates $(t,v)$, the following equality holds:
\begin{multline}
\mathcal{B}_\Lambda f(t,v)=\\ 
\mathcal{B}_{|\Lambda|}^{(t)}f(t,v)+\frac{1}{t}\left(\sum_{i=1}^{n-1}v_i(1-v_i)\derptwo{f}{v_i}-2\sum_{1\leq i< j\leq n-1}v_iv_j\derptwobis{f}{v_i}{v_j}+\sum_{i=1}^{n-1}(\lambda_i-|\Lambda|v_j)\derpone{f}{v_i}\right).
\end{multline}
where $\mathcal{B}_{|\Lambda|}^{(t)}$ denotes the Bessel operator acting on the variable $t$.
\end{prop}

The proof is based on the following Lemma which can be proved by a direct computation using the chain rule.
\begin{lemme}
In the coordinates $(t,v)$, the first order derivatives are given by:
\begin{align*}
\derpone{f}{x_i}&=\derpone{f}{t}+\sum_{j=1}^{n-1}\left(\frac{\delta_{ij}}{t}-\frac{v_j}{t}\right)\derpone{f}{v_j} \text{ if } i\neq n,\\
\derpone{f}{x_n}&=\derpone{f}{t}-\sum_{j=1}^{n-1}\frac{v_j}{t}\derpone{f}{v_j}.
\end{align*}
The second order derivatives are given by:
\begin{align*}
\derptwo{f}{x_i}&= \derptwo{f}{t}+\frac{2}{t}\sum_{j=1}^{n-1}(\delta_{ij}-v_j)\derptwobis{f}{t}{v_j}+\frac{1}{t^2}\sum_{j,k}(\delta_{ij}-v_j)(\delta_{ik}-v_k) \derptwobis{f}{v_j}{v_k}-\frac{2}{t^2}\sum_{j=1}^{n-1}(\delta_{ij}-v_j)\derpone{f}{v_j} \text{ if } i\neq n,\\
\derptwo{f}{x_n}&= \derptwo{f}{t}-\frac{2}{t}\sum_{j=1}^{n-1} v_j\derptwobis{f}{v_i}{t}+\frac{2}{t^2}\sum_{j=1}^{n-1}v_j\derpone{f}{v_j} +\frac{1}{t^2}\sum_{j,k}^{n-1}v_jv_k\derptwobis{f}{v_j}{v_k}.
\end{align*}
\end{lemme}

\begin{proof}[Proof of Proposition \ref{prop:EqualityBesselOperatorTensor}]

In the coordinates $(t,v)$, the Bessel operator becomes:
\begin{equation}
\mathcal{B}_\Lambda f(x)=\sum_{i=1}^{n-1} tv_i\derptwo{f}{x_i} +(1-|v|)\derptwo{f}{x_n} +\sum_{i=1}^n\lambda_i \derpone{f}{x_i}.
\end{equation}
We first look at the second order derivatives in the variable $t$, and found:
\[
\sum_{i=1}^{n-1}tv_i \derptwo{f}{t}  +t(1-|v|)\derptwo{f}{t}=t\derptwo{f}{t}.
\]
Next, we look at the first order derivatives with respect to $t$ and we get:
\[ 
\sum_{i=1}^n \lambda_i\derpone{f}{t}=|\Lambda|\derpone{f}{t}.
\]
Then we look at the second order terms $\derptwobis{f}{v_j}{v_k}$:
\begin{align*}
&\sum_{i=1}^{n-1}\frac{v_i}{t}\sum_{j,k}(\delta_{ij}-v_j)(\delta_{ik}-v_k) \derptwobis{f}{v_j}{v_k}+\frac{1-|v|}{t}\sum_{j,k}v_jv_k \derptwobis{f}{v_j}{v_k}\\
&=\frac{1}{t}\left(\sum_{i=1}^{n-1}v_i(1-v_i)\derptwo{f}{v_i}-\sum_{k\neq j}v_jv_k\derptwobis{f}{v_j}{v_k} \right).
\end{align*}
Then we look at the second order terms $\derptwobis{f}{v_j}{t}$:
\[
\sum_{i=1}^{n-1}2v_i\sum_{j=1}^{n-1}(\delta_{ij}-v_j)\derptwobis{f}{v_j}{t}-2(1-|v|)\sum_{j=1}^{n-1}v_j\derptwobis{f}{v_j}{t}=0.
\]
Finally, we look at the first order terms with respect to the $v_i$'s, and we get:
\begin{align*}
&\sum_{i=1}^{n-1}\frac{2v_i}{t}\sum_{j=1}^{n-1}(v_j-\delta_{ij})\derpone{f}{v_j}+\frac{2(1-|v|)}{t}\sum_{j=1}^{n-1}v_j\derpone{f}{v_j}\\
&+\sum_{i=1}^{n-1}\frac{\lambda_i}{t}\sum_{j=1}^{n-1}(\delta_{ij}-v_j)\derpone{f}{v_j}-\frac{\lambda_n}\sum_{j=1}^{n-1}v_j\derpone{f}{v_j}\\
&=\frac{1}{t}\left(\sum_{j=1}^{n-1}(\lambda_j-|\Lambda|v_j)\derpone{f}{v_j}\right)
\end{align*}
\end{proof}

We are now able to prove Theorem \ref{thm:nfoldTensorProduct} using Proposition \ref{prop:EqualityBesselOperatorTensor}.
\begin{proof} [Proof of Theorem \ref{thm:nfoldTensorProduct}]
The fact that the operator $\Psi_k^\Lambda(P)$ intertwines the $N_1$ and the $G_{\Omega}$ actions of the $n$-fold tensor product representation in the stratified model with the representation $\pi_{|\lambda|+2k}$ can be check using \eqref{eq:FormulStratifEqualRankConforme}, \eqref{eq:FormulStratifEqualRankTranslation} and the fact that the $K_1$-action is trivial on the set $X\simeq D_{n-1}$.

For the $\bar{\n}_1$-action, the intertwining property is equivalent to the following identity for Bessel operators, on $f\in L^2_{|\Lambda|}(\R^+)\hat{\otimes} Pol_k(D_{n-1})$:
\begin{equation}\label{eq:egaliteBesselOperatorsTensor}
t^{-k}\mathcal{B}_\Lambda f=\mathcal{B}_{|\Lambda|+2k}\left( t^{-k} f\right).
\end{equation}

The following equality is true for the Bessel operator on a Jordan algebra $V$ and for a smooth function $f$ on $V$ (see \cite{Far_Kor}, proof of Prop.XV.2.4):
\begin{equation}
\mathcal{B}_\lambda (\Delta(x)^\mu f) =\Delta(x)^\mu(\mathcal{B}_{\lambda+2\mu}f+\mu(\mu+\lambda-\frac{n}{r})x^{-1}f).
\end{equation}

Using this formula, the equation \eqref{eq:egaliteBesselOperatorsTensor} becomes 
\begin{equation}
\mathcal{B}_\Lambda f(t,v)=\mathcal{B}_{|\Lambda|} f(t,v)-k(k+|\Lambda|-1)t^{-1}f(t,v).
\end{equation}
Now, using Proposition \ref{prop:EqualityBesselOperatorTensor}, equation \eqref{eq:egaliteBesselOperatorsTensor} is equivalent to 
\begin{equation}
\sum_{i=1}^{n-1}v_i(1-v_i)\derptwo{f}{v_i}-2\sum_{1\leq i< j\leq n-1}v_iv_j\derptwobis{f}{v_i}{v_j}+\sum_{i=1}^{n-1}(\lambda_i-|\Lambda|v_j)\derpone{f}{v_i}+k(k+|\Lambda|-1)f(t,v)=0
\end{equation}

Finally, it is known that the operator 
\[ 
\sum_{i=1}^{n-1}v_i(1-v_i)\derptwo{f}{v_i}-2\sum_{1\leq i< j\leq n-1}v_iv_j\derptwobis{f}{v_i}{v_j}+\sum_{i=1}^{n-1}(\lambda_i-|\Lambda|v_j)\derpone{f}{v_i}
\]
admits the space $Pol_k(D_{n-1})$ as an eigenspace (see \cite{DunklXu}, section 2.3.3). What ends the proof.

\end{proof}

\section{Application: Restriction of the holomorphic discrete series of $SO_0(2,n)$}\label{sec:SO(2,n)}

In this section, we first discuss the restriction of a member of the scalar valued holomorphic discrete series of the identity component of the indefinite orthogonal group $SO_0(2,n)$, with $n\geq 3$, to the subgroup $SO_0(2,n-p)$ with $1\leq p\leq n-1$. As in Theorem \ref{thm:nfoldTensorProduct}, we are going to prove a one-to-one correspondence between symmetry breaking operators for the restriction and some orthogonal polynomials on the $p$-dimensional unit ball. Finally, we describe the restrictionof a scalar valued holomorphic discrete series representation of $SO_0(2,n)$ to the subgroup $SO_0(2,n-p)\times SO(p)$.

\subsection{Setting}\label{sec:settingSO(2,n)}

\paragraph{Holomorphic model}$\ $

For $p,q\in\N$, let $Q_{p,q}$ be the standard quadratic form of signature $(p,q)$ on $\R^{p+q}$ defined by
\begin{equation}
Q_{p,q}(x)=\sum_{k=1}^px_k^2-\sum_{k=1}^q x_{p+k}^2.
\end{equation} 
Define the indefinite orthogonal group $O(p,q)$ as the group of invertible linear maps of $\R^{p+q}$ preserving the quadratic form $Q_{p,q}$, and denote by $SO_0(p,q)$ its identity component.

For $n\geq 3$, consider the tube domain $T_{\Omega_n}=\R^n + i\Omega_n$ where
\begin{equation}
\Omega_n=\{x\in \R^n~|~ Q_{1,n-1}(x)>0,\ x_1>0\},
\end{equation}
is the time-like cone. It corresponds to the tube domain associated to the symmetric cone described in example \ref{ex:SymmetricCone} \eqref{ex:SymmetricConeRank2}. It is knwon from the classification of irreducible symmetric cones (\cite{Far_Kor}, p.213) that $G(T_{\Omega_n})$ is isomorphic to the identity component $SO_0(2,n)$ of the indefinite orthogonal group $O(2,n)$. Hence, $T_{\Omega_n}$ is biholomorphic to $SO_0(2,n)/(SO(2)\times SO(n))$, where $SO(2)\times SO(n)$ is a maximal compact subgroup in $SO_0(2,n)$. 

The Euclidean Jordan algebra associated to the cone $\Omega_n$ is the vector space $V_n=\R\times\R^{n-1}$ with the product defined in example \ref{ex:JordanAlgebra}\eqref{ex:JordanAlgebraRank2} where $\R^{n-1}$ is equipped with the usual inner product denoted by $\langle\cdot, \cdot\rangle$. For $(x,v),(y,u)\in \R\times\R^{n-1}$, we recall that the product on $V_n$ is given by
\begin{equation}\label{eq:ProductJordanSO(2,n)}
(x,u)\cdot(y,v)=(xy+\langle u,v \rangle,xv+yu).
\end{equation} 
The identity element is $e=(1,0,\cdots,0)$, and for any $x=(x_1,\cdots,x_n)\in \R^n$, we have
\begin{equation}
x^2-2x_1 x+Q_{1,n-1}(x) e=0.
\end{equation}
Hence, the determinant and the trace of the Jordan algebra $V_n$ are 
\begin{equation}
\Delta(x)=Q_{1,n-1}(x), \text{ and } \tr(x)=2x_1.
\end{equation}
The Euclidean Jordan algebra is then of rank $2$ and of dimension $n$, hence $m=\frac{n}{2}$, and $V_n$ is simple for $n\geq 3$. Notice that $(x|x)^2=\tr(x^2)=2\|x\|^2$ where $\|x\|$ is the usual Euclidean norm on $\R^n$. Hence the Euclidean measure $dx= 2^{\frac{n}{2}}d\lambda(x)$ where $d\lambda$ denotes the Lebesgue measure.

Let $\lambda>n-1$, consider the weighted Bergman space
\begin{equation}
\Hi{n}{\lambda}=\mathcal{O}(T_{\Omega_n})\cap L^2(T_{\Omega_n}, Q_{1,n-1}(x)^{\lambda-n} dxdy).
\end{equation}
This is a Hilbert space of holomorphic functions which is non zero for $\lambda\geq n-1$, and whose reproducing kernel is given by equation \eqref{eq:ReproducingKernelGeneral} which becomes:
\begin{equation}
K_\lambda (\omega,\xi)=k_{\lambda,n}Q_{1,n-1}(\omega-\overline{\xi})^{-\lambda},\end{equation}
where $k_{\lambda,n}=\frac{(2i)^{2\lambda}(\lambda-n/2)\Gamma (\lambda)}{(4\pi )^n\Gamma (\lambda-n+1)}$.

The group $SO_0(2,n)$ acts on $\Hi{n}{\lambda}$ by a multiplier representation denoted $\pi^{(n)}_\lambda$ as in \eqref{def:HolomorphicDiscreteScalarValued}, and for $\lambda\in \N$ such that $\lambda>n-1$ this is an irreducible unitary representation of $SO_0(2,n)$. Notice that, for any parameter $\lambda>n-1$, the multiplier representation $\pi^{(n)}_\lambda$ can be extended to the universal covering of $SO_0(2,n)$ and the following results still holds in this context. However, for convenience, we restrict our presentation to integer parameters $\lambda$.

There is a natural embedding of $V_{n-1}$ in $V_n$ given by the map $(x_1,\cdots, x_{n-1})\mapsto (x_1,\cdots, x_{n-1},0)$ which extends naturally to an holomorphic embedding of $T_{\Omega_{n-1}}$ into $T_{\Omega_n}$ as described in section \ref{sec:defintionDiffeo}. In this setting, the group $SO_0(2,n-1)$ is realized as the subgroup of $SO_0(2,n)$ which fixes the last coordinate.

The branching law for the restriction of the representation $\pi^{(n)}_\lambda$ to the subgroup $SO_0(2,n-1)$ is multiplicity free and is given by (see \cite{Kobay_08}, Thm. 8.3):
\begin{equation}\label{eq:branchingSO(2,n)SO(2,n-1)}
\left.\pi^{(n)}_\lambda\right|_{SO_0(2,n-1)}\simeq {\sum_{k\in \N}}^\oplus \pi^{(n-1)}_{\lambda+k}.
\end{equation}

Let $\lambda,\ \nu\in \N$ such that $\lambda,\ \nu> n-1$ and $l=\nu-\lambda \in \N$. For convenience, set 
\begin{equation}\label{def:paramAlpha}
\alpha=\lambda-\frac{n-1}{2}.
\end{equation}

In this setting, there exists a symmetry breaking operator from $\Hi{n}{\lambda}$ to $\Hi{n-1}{\nu}$ with respect to the action of $SO(2,n-1)$, called the holomorphic Juhl operator, (see \cite{KP_SBO}), given by:
\begin{equation}\label{eq:Juhl_operator}
\Juhl =Rest_{\xi_n=0}\circ \left( \sum_{k=0}^{[\frac{l}{2}]}a_k(l,\alpha)\left(\frac{\partial}{\partial \xi_n}\right)^{l-2k}\Delta^k_{\C^{1,n-2}}\right)
\end{equation}
where $a_k(l,\alpha)=(-1)^k2^{l-2k}\frac{\Gamma(\alpha+l-k)}{\Gamma(\alpha)k!(l-2k)!}$ and $\Delta_{\mathbb{C}^{1,n-2}}=\derpar{2}{\xi_1}-\sum_{k=2}^{n-2}\derpar{2}{\xi_k}$.

Following \cite{KP}, we introduce the relative reproducing kernel, for $\xi \in T_{\Omega_n}$ and $z\in T_{\Omega_{n-1}}$: 
\begin{equation}
K_{\lambda,\nu}(\xi,z)=\xi_n^l\left((\xi_1-\overline{z_1})^2-(\xi_2-\overline{z_2})^2-\cdots-(\xi_{n-1}-\overline{z_{n-1}})^2-\xi_n^2\right)^{-\nu}.
\end{equation}

The corresponding holographic operator is then given by the following Theorem.

\begin{theorem}\label{thm:holographic_operator_conforme}
Let $\lambda,\ \nu\in \N$ such that $\lambda,\ \nu> n-1$ and $l=\nu-\lambda \in \N$. Let $f\in \Hi{n-1}{\nu}$, and $\xi \in T_{\Omega_n}$. Then:
\begin{align}
&\Juhladj f(\xi) =C'(\lambda;\nu)\int_{T_{\Omega_{n-1}}} f(z)K_{\lambda,\nu}(\xi,z)Q_{1,n-1}(x)^{\lambda-n}~dxdy \\
\end{align}
where $C'(\lambda;\nu)=\frac{2^{2\lambda-l-2n-3/2}i^{2\lambda+2l}}{\pi^nl!}(\lambda-n+1)_{n+l-1}(2\lambda-n)_{l+1}$.
\end{theorem}

This result is shown in \cite{KP} (Theorem 3.10) but we are going to give yet another proof in section \ref{sec:ProofThmHolographyConforme} which is based on the Laplace transform.

\paragraph{$L^2$-model and Fourier-Laplace transform}$\ $

In the setting of section \ref{sec:settingSO(2,n)}, the Laplace transform \eqref{def:LaplaceTransformGeneral} is given by the following formula:
\begin{equation} \label{eq:Fourier_transform_SO2n}
\mathcal{F}_n f(z)=\int_{\Omega_n} f(x) e^{i\langle x,z \rangle} Q_{1,n-1}(x)^{\lambda-\frac{n}{2}}~dx.
\end{equation}

This is, up to a scalar, an isometry (see \cite{Far_Kor}) from $\Hi{n}{\lambda}$ to $L^2(\Omega_n,\ Q_{1,n-1}(y)^{\lambda-\frac{n}{2}}~dy)$. More precisely:
\begin{equation}
\|\mathcal{F}_n f\|^2_{H_\lambda^2(T_{\Omega_n})}=b_n(\lambda)\|f\|^2_{L^2_\lambda(\Omega_n)},
\end{equation}
with $b_n(\lambda)=(2\pi)^{\frac{3}{2}n-1}2^{-2\lambda+n}\Gamma(\lambda-\frac{n}{2})\Gamma(\lambda-n+1)$.
 
Using this Laplace transform, one can realize the $L^2$-model for the representation $\pi_\lambda^{(n)}$ on the space $L^2(\Omega_n,\ Q_{1,n-1}(y)^{\lambda-\frac{n}{2}}~dy)$ as in \eqref{def:L2ModelGeneral}.
 
The symmetry breaking operator between $\pi_\lambda^{(n)}$ and $\pi_\nu^{(n-1)}$ in the $L^2$ model is given for all $\lambda,\nu\in\N$ with $\lambda>n-1$ and $l=\nu-\lambda\in \N$ in the following proposition.

\begin{theorem}[\cite{KP}, Prop.3.7]\label{thm:SBOperatorL2Conforme}
Let $\lambda,\nu\in\N$ such that $\lambda>n-1$ and $l=\nu-\lambda\in \N$. The operator $\widehat{\Juhl}$ from $L^2_\lambda(\Omega_n)$ to $L^2_\nu(\Omega_{n-1})$ defined by
\begin{equation}\label{eq:SBOperatorL2Conforme}
\widehat{\Juhl} f(x)=i^{-l}Q_{1,n-2}(x)^{\frac{-l}{2}}\int_{-1}^1 f(x,-Q_{1,n-2}(x)^{1/2}v)C_l^\alpha (v)(1-v^2)^{\lambda-\frac{n}{2}}~dv
\end{equation}
is a symmetry breaking operator between $\pi_\lambda^{(n)}$ and $\pi_\nu^{(n-1)}$.
\end{theorem}

In order to define the associated holographic transform, we follow \cite{KP} and define, for $x\in \R^+$ and $y\in\R$, a polynomial $I_lC_l^\alpha$ by
\begin{equation}\label{def:InflatedGegenbauerPolynomial}
I_lC_l^\alpha(x,y)=x^\frac{l}{2} C_l^\alpha\left(\frac{y}{x^\frac{1}{2}}\right),
\end{equation}
where $C_l^\alpha$ denotes the Gegenbauer polynomial \eqref{def:GegenbauerPolynomials} in one variable (see section \ref{sec:gegenbauer} for more details). 

\begin{theorem}[\cite{KP}, Prop.3.5]\label{thm:HolographicOperatorL2Conforme}
Let $\lambda,\nu\in\N$ such that $\lambda>n-1$ and $l=\nu-\lambda\in \N$. The holographic operator from $L^2_\nu(\Omega_{n-1})$ to $L^2_\lambda(\Omega_n)$ is given by the map $\phi_\lambda^\nu$ defined, for $f\in L^2_\nu(\Omega_{n-1}$ and $x\in \Omega_n$, by
\begin{equation}\label{eq:HolographicOperatorL2Conforme}
\phi_\lambda^\nu f(x)=I_lC_l^\alpha(Q_{1,n-2}(x'),-x_n)\cdot f(x'),
\end{equation}
where $x'=(x_1,\cdots,x_{n-1})$.
\end{theorem}

\textit{Remark:} As in the $SL_2(\R)$ case, notice that we choosed a different version of the Laplace transform that in \cite{KP} which leads to different formulas for the symmetry breaking operator \eqref{eq:SBOperatorL2Conforme} and the holographic operator \eqref{eq:HolographicOperatorL2Conforme}.

The proof of the following lemma, which gives the inverse Laplace transform of the reproducing kernel, is direct from Lemma IX.3.5 in \cite{Far_Kor}.
\begin{lemme}\label{lem:Inv_four_conforme}
We have, for $\xi\in T_{\Omega_n}$ and $x\in \Omega_n$:
\begin{equation}\label{eq:Inv_four_SO2n}
\mathcal{F}_n^{-1} K_\lambda(.,\xi) (x)=\frac{1}{b_n(\lambda)}e^{-i\langle x,\overline{\xi}\rangle}.
\end{equation}
\end{lemme}

\subsection{Proof of theorem \ref{thm:holographic_operator_conforme}}\label{sec:ProofThmHolographyConforme}

In this section, we  give a proof of Theorem \ref{thm:holographic_operator_conforme} based on the Laplace transform. We use Lemma \ref{lem:calc_adj} to reduce the computation of $\Juhladj$ to the computation of $\Juhl  K_\lambda (\cdot,\xi)$. In \cite{KP}, the authors proved this result by a direct computation of $\Juhl  K_\lambda (\cdot,\xi)$. 
 \begin{lemme}\label{lem:calc_adj}
 Let $D_j (j=1,2)$ be some complex manifolds, and $H_j$ some Hilbert spaces of holomorphic functions on $D_j$ with reproducing kernels $K^{(j)}(\cdot,\cdot)$. If $R\ :\ H_1 \to \ H_2$ is a continous linear map, then:
 \begin{enumerate}
 \item $\overline{RK^{(1)}(\cdot,\zeta)(\tau')}=(R^*K^{(2)}(\cdot,\tau'))(\zeta) $ for $\zeta \in D_1,\tau'\in D_2$.
 \item $(R^*g)(\zeta)=(g,RK^{(1)}(\cdot,\zeta))_{H_2}$ for $g\in H_2,\ \zeta \in D_1$.
 \end{enumerate}
\end{lemme}

The following Lemma is then the first step to compute $\Juhl  K_\lambda (\cdot,\xi)$.
\begin{lemme}\label{lem:Fourier_SBO_Noyau}
Let $\xi\in T_{\Omega_n}$ and $x\in \Omega_{n-1}$, then
\begin{multline}
\widehat{\Juhl}\mathcal{F}_n^{-1}K_\lambda(.,\xi)(x)=\\
C(\lambda,\nu)\overline{\xi_n}^lQ_{1,n-2}(x)^{\nu +\frac{n-1}{2}}e^{-i \langle x,\overline{\xi'}\rangle}~_0F_1\left(l+\alpha+1;-\frac{Q_{1,n-2}(x) \overline{\xi_n}^2}{4}\right),
\end{multline}
where $\xi'=(\xi_1,\cdots,\xi_{n-1})$ and $C(\lambda,\nu)=\frac{\sqrt{\pi}(2\alpha)_l\Gamma(\alpha+\frac{1}{2})}{2^l l!\Gamma(l+\alpha+1) b_n(\lambda)}$.
\end{lemme}

\begin{proof}
Using Lemma \ref{lem:Inv_four_conforme} and Lemma \ref{lem:Fourier_poids_conforme} we get.
\begin{align*}
&\widehat{\Juhl}\mathcal{F}_n^{-1}K_\lambda(.,\xi)(x)
=\frac{1}{b_n(\lambda)}\widehat{\Juhl}\left(y\mapsto Q_{1,n-1}(y)^{\lambda-\frac{n}{2}}e^{-i\langle y,\bar{\xi}\rangle}\right)\\
&=\frac{\sqrt{\pi}(2\alpha)_l\Gamma(\alpha+\frac{1}{2})}{2^l l!\Gamma(l+\alpha+1) b_n(\lambda)}
Q_{1,n-2}(x)^{-\frac{l}{2}}e^{-i\langle x,\bar{\xi'}\rangle}
\int_{-1}^1 e^{i v Q_{1,n-2}(x)^{\frac{1}{2}}\overline{\xi_n}}C_l^\alpha(v)(1-v^2)^{\lambda-\frac{n}{2}}~dv\\
&=C(\lambda,\nu)\overline{\xi_n}^l e^{-i\langle x,\bar{\xi'}\rangle}~_0F_1\left(l+\alpha+1;-\frac{Q_{1,n-2}(x)\overline{\xi_n}^2}{4}\right)
\end{align*}
\end{proof}

From this Lemma we get the following proposition.
\begin{prop}\label{prop:IntermediaireHolographicSO(2,n)}
Let $\lambda,\nu \in \N$ such that $\lambda\geq n-1$ and such that $l=\nu-\lambda\in \N$. Then, for $z\in T_{\Omega_{n-1}}$ and $\xi\in T_{\Omega_n}$, we have:
\begin{multline}
\Juhl K_\lambda(.,\xi)(z) = \\
C(\lambda;\nu)b_{n-1}(\nu)\overline{\xi_n}^l\int_{\Omega_{n-1}} \mathcal{F}_{n-1}^{-1}K_\nu(\cdot,\xi')(x)~_0F_1\left(-\frac{Q_{1,n-2}(x)\overline{\xi_n}^2}{4}\right)e^{i\langle x,z\rangle}Q_{1,n-2}(x)^{\nu-\frac{n-1}{2}}~dx.
\end{multline}
\end{prop}

\begin{proof}
This result is a direct computation using Lemma \ref{lem:Fourier_SBO_Noyau}.
\begin{align*}
&\Juhl K_\lambda(.,\xi)(z) \\
&=\int_{\Omega_{n-1}} \widehat{\Juhl}\mathcal{F}_n^{-1}K_\lambda(.,\xi)(x) e^{i\langle x,z \rangle}Q_{1,n-2}(x)^{\nu-\frac{n-1}{2}}~dx\\
&=C(\lambda;\nu)\overline{\xi_n}^l \int_{\Omega_{n-1}} e^{-i\langle x,\overline{\xi'} \rangle }~_0F_1\left(l+\alpha+1;-\frac{Q_{1,n-2}(x)\overline{\xi_n}^2}{4}\right) e^{i \langle x,z \rangle }Q_{1,n-2}(x)^{\nu-\frac{n-1}{2}}~dx\\
&=C(\lambda;\nu){b_{n-1}(\nu)}\overline{\xi_n}^l\int_{\Omega_{n-1}} \mathcal{F}_{n-1}^{-1}K_\nu(\cdot,\xi')(x)~_0F_1\left(l+\alpha+1;-\frac{Q_{1,n-2}(x)\overline{\xi_n}^2}{4}\right) e^{i \langle x,z \rangle}Q_{1,n-2}(x)^{\nu-\frac{n-1}{2}}~dx
\end{align*}
Lemma \ref{lem:Inv_four_conforme} shows the last step.
\end{proof}

We are finally able to prove Theorem \ref{thm:holographic_operator_conforme}.

\begin{proof}[Proof of Theorem \ref{thm:holographic_operator_conforme}]
From Proposition \ref{prop:IntermediaireHolographicSO(2,n)}:
\begin{align*}
&\Juhl K_\lambda(.,\xi)(z)  \\
&=C(\lambda;\nu)b_{n-1}(\nu)\overline{\xi_n}^l\int_{\Omega_{n-1}} \mathcal{F}_{n-1}^{-1}K_\nu(\cdot,\xi')(x)\\
&~~~~~~~~~~~~~~~\times{}_0F_1\left(l+\alpha+1;-\frac{Q_{1,n-2}(x)\overline{\xi_n}^2}{4}\right) e^{i \langle x,z \rangle}Q_{1,n-2}(x)^{\nu-\frac{n-1}{2}}~dx\\
&=C(\lambda;\nu)b_{n-1}(\nu)\overline{\xi_n}^l\int_{\Omega_{n-1}} \mathcal{F}_{n-1}^{-1}K_\nu(.,\xi')(x)e^{i\langle x,z\rangle} \\
&~~~~~~~~~~~~~~~\times \sum_{k\geq 0} \frac{(-1)^k}{(l+\alpha+1)_k k!4^k}Q_{1,n-2}(x)^k\overline{\xi_n}^{2k}Q_{1,n-2}(x)^{\nu-\frac{n-1}{2}}~dx\\
&=C(\lambda;\nu)b_{n-1}(\nu)\overline{\xi_n}^l\sum_{k\geq 0}\frac{\overline{\xi_n}^{2k}}{(l+\alpha+1)_kk!4^k}\left(\Delta_{\C^{1,n-2}}\right)^k K_\nu(\cdot,\xi')(z)\\
&=C(\lambda;\nu)b_{n-1}(\nu)k_{\nu,n-1}\overline{\xi_n}^l \sum_{k\geq 0} \frac{ s_k(n-1;\nu)}{k! (l+\alpha+1)_k 4^k}\overline{\xi_n^{2k}}Q_{1,n-2}(z-\overline{\xi'})^{-\nu-k}\\
&=C(\lambda;\nu)b_{n-1}(\nu)k_{\nu,n-1}\overline{\xi_n}^l\sum_{k\geq 0}b_k\overline{\xi_n^{2k}}Q_{1,n-2}(z-\overline{\xi'})^{-\nu-k},
\end{align*}
where $b_k=\frac{s_k(n-1;\nu)}{k! (l+\alpha+1)_k 4^k}$ and $s_k(n-1;\nu)$ is the polynomial obtained by iteration of the following relation (see Lemma 3.14 in \cite{KP}):
\[\Delta_{\mathbb{C}^{1,n-2}}Q_{1,n-2}(\xi)^{-\nu}=2\nu(2\nu-n+3)Q_{1,n-2}(\xi)^{-\nu-1}.\]

More precisely, it satisfies the following recursion relation
\[s_{k+1}(n-1;\nu)=2(\nu+k)(2(\nu+k)-n+3)s_k(n-1;\nu),\] 
which gives us the following induction relation for $b_k$:
\[b_{k+1}=\frac{2(\nu+k)(2(\nu+k)-n+3)}{4(k+1)(k+\nu-\frac{n-1}{2}+1)}b_k=\frac{(\nu+k)}{k+1}b_k,\]
with $b_0=1$.

On the other hand, using power series expansion, we have:
\begin{align*}
\overline{K_{\lambda ,\nu}(\xi;z)}&=\overline{\xi_n}^l(Q_{1,n-2}(z-\overline{\xi '})-\overline{\xi_n}^2)^{-\nu}\\
&=\overline{\xi_n}^lQ_{1,n-2}(z-\overline{\xi'})^{-\nu}\left(1-\frac{\overline{\xi_n}^2}{Q_{1,n-2}(z-\overline{\xi'})}\right)^{-\nu}\\
&=\overline{\xi_n}^l\sum_{k\geq 0}b_k'Q_{1,n-2}(z-\overline{\xi'})^{-\nu-k}\overline{\xi_n}^{2k}
\end{align*}
with $b_0'=1$ and $b_k'=\frac{1}{k!}\nu(\nu+1)\cdots(\nu+k-1)$.
One can check that $(b'_k)$ satisfies the same recursion relation as $(b_k)$, and finally $b_k=b'_k$ for all $k$.
This gives:
\[D_{\lambda \to \nu}K_\lambda(.,\xi)(z)=C(\lambda;\nu)b_{n-1}(\nu)k_{\nu,n-1}\overline{K_{\lambda ,\nu}(\xi;z)}.\]
Notice that the statement is true for $\left|\frac{\overline{\xi_n}^2}{Q_{1,n-2}(z-\overline{\xi'})}\right|<1$ but one uses analytic continuation on $T_{\Omega(n-1)}$ to get the result.

Finally, we compute the scalar $C'(\lambda;\nu)$.
\[C'(\lambda;\nu)=\frac{\sqrt{\pi}(2\alpha)_l\Gamma(\alpha+\frac{1}{2})}{2^l l!\Gamma(l+\alpha+1)} k_{\nu,n-1}\frac{b_{n-1}(\nu)}{b_n(\lambda)}.\]
First, we have :
\begin{align*}
\frac{b_{n-1}(\nu)}{b_n(\lambda)}=\frac{2^{2\lambda-2\nu-\frac{5}{2}}}{\pi^{\frac{3}{2}}}\frac{\Gamma(\nu-\frac{n-1}{2})\Gamma(\nu-n+2)}{\Gamma(\lambda-\frac{n}{2})\Gamma(\lambda-n+2)}
\end{align*}
This gives us the following :
\begin{align*}
&C'(\lambda;\nu)=\frac{2^{2\lambda-l+1-2n-\frac{1}{2}}i^{2\lambda+2l}}{\pi^n l!}\\
&\times \frac{(\nu-\frac{n-1}{2})\Gamma(\nu-\frac{n-1}{2})\Gamma(\nu-n+2)\Gamma(\nu)\Gamma(\alpha+\frac{1}{2})(2\alpha)_l}{\Gamma(\lambda-\frac{n}{2})\Gamma(\lambda-n+1)\Gamma(\nu-n+2)\Gamma(\nu-\frac{n-1}{2}+1)}\\
&=\frac{2^{2\lambda-l+1-2n-\frac{1}{2}}i^{2\lambda+2l}}{\pi^n l!}(\lambda-n+1)_{n+l-1}(2\lambda-n)_{l+1}\\
&\times \frac{(\lambda+l-\frac{n-1}{2})\Gamma(\lambda-\frac{n}{2}+1)\Gamma(\lambda+l-\frac{n-1}{2})\Gamma(\lambda+l-\frac{n}{2}+1)}{(2\lambda-n)\Gamma(\lambda+l-\frac{n}{2}+1)\Gamma(\lambda-\frac{n}{2})\Gamma(\lambda+l-\frac{n-1}{2}+1)}\\
&=\frac{2^{2\lambda-l-2n-\frac{3}{2}}i^{2\lambda+2l}}{\pi^n l!}(\lambda-n+1)_{n+l-1}(2\lambda-n)_{l+1}
\end{align*}

\end{proof}

\subsection{The stratified model}

In this section, we use the formalism of section \ref{sec:defintionDiffeo} in order to introduce the stratified model for the holomorphic discrete series representations $\pi_\lambda^{(n)}$ of $SO_0(2,n)$.

We set $d\mu_\alpha(v)=2^{-\frac{1}{2}}(1-v^2)^{\alpha-\frac{1}{2}}~dv$ as a measure on $(-1,1)$ where $dv$ is the Lebesgue measure on $(-1,1)$. Orthogonal polynomials with respect to the measure $d\mu_\alpha$ on $(-1,1)$ are the Gegenbauer polynomials (see Appendice \ref{sec:gegenbauer}).

The Euclidean Jordan algebra $V_{n-1}$ is naturally embedded in $V_n$, and we have 
\begin{equation}
V_{n-1}^\bot=\left\{(0,\cdots ,0,v)~|~v\in\R\right\},
\end{equation}
thus the space $X$ defined in \eqref{def:SubsetXGeneral} is given by:
\begin{equation}
X=\left\{(0,\cdots ,0,v)~|~Q_{1,n-1}((1,0,\cdots ,0,v))>0\right\}\simeq (-1,1).
\end{equation}
In order to find the explicit formula of the diffeomorphism $\iota$, defined by \eqref{def:DiffeoGeneral}, we first need to compute $x^\frac{1}{2}$ for an element $x\in \Omega_n$. For $x=(x_1,\cdots x_n)\in \Omega_n$ we have $x^2=(\sum_k x_k^2,2x_1x_2,\cdots,2x_1x_n)$. Now, solving the equation $y=(y_1,\cdots, y_n)=(\sum_k x_k^2,2x_1x_2,\cdots,2x_1x_n)$ leads to the formula
\begin{equation}\label{eq:SquareRootConforme}
y^\frac{1}{2}=\left(x_1,\frac{y_2}{2x_1},\cdots,\frac{y_n}{2x_1}\right),
\end{equation}
where $x_1=\left(\frac{y_1+Q_{1,n-1}(y)^\frac{1}{2}}{2}\right)^\frac{1}{2}$. Notice that 
\[\frac{1}{2x_1}=\left(\frac{y_1-Q_{1,n-1}(y)^\frac{1}{2}}{2\sum_{k\geq 2} y_k^2}\right)^\frac{1}{2}.\]

Finally, for $y=(y_1,\cdots,y_{n-1},0)\in \Omega_{n-1}$ and $v\in (-1;1)$, we have:
\begin{equation}
P(y^\frac{1}{2})(0,\cdots,0,v)=(0,\cdots,0,Q_{1,n-1}(y)^\frac{1}{2}v).
\end{equation}

Using the symmetry $v\mapsto -v$ on $X$, the diffeomorphism $\iota$ becomes:
\begin{equation}\label{eq:diffeo_conforme}
\begin{aligned}[c]
\iota : \Omega_{n-1}\times(-1,1) &\to  \Omega_n\\
(y',v)&\mapsto \left(y',-vQ_{1,n-2}(y')^{\frac{1}{2}}\right),
\end{aligned}
\end{equation}
and then
\begin{equation}\label{eq:changement_mesure_conforme}
Q_{1,n-2}(y')^{\frac{1}{2}}~dy'dv=dy.
\end{equation}
Notice that 
\begin{equation}
Q_{1,n-1}(\iota(y',v))=Q_{1,n-2}(y')(1-v^2).
\end{equation}
\textit{Remark: }We use the symmetry $v\mapsto -v$ of $(-1,1)$, in order to recover the diffeomorphism used in \cite{KP} (equation (3.15)) for this example.

The Hilbert space isomorphism \eqref{eq:Hilbert_space_isom_EqualRank} given by the pullback $\iota^*$ becomes:
\begin{align*}
&L^2\left(\Omega_{n-1}\times (-1;1),Q_{1,n-2}(y')^{\lambda-\frac{n-1}{2}}~dy'd\mu_\alpha(v)\right)\\
\simeq~~ & \confL{\lambda}{n-1}\hat{\otimes}L^2\left((-1;1),d\mu_\alpha(v)\right)\\
\simeq~~&{\sum_{l\geq 0}}^\oplus  \confL{\lambda}{n-1}\hat{\otimes}~ \C C_l^\alpha
\end{align*}
where the last sum is an orthogonal Hilbert sum, and $C_l^\alpha$ are the Gegenbauer polynomials. The last isomorphism is equivalent to the fact that the family of Gegenbauer polynomials is an Hilbert basis for $L^2((-1;1),d\mu_\alpha(v))$ for $\alpha>1$. This allows to consider the orthogonal projection for $h \in L^2(\Omega_{n-1}\times (-1;1),Q_{1,n-2}(y')^{\lambda-\frac{n-1}{2}}~dy'd\mu_\alpha(v))$:
\begin{equation}
P_lh(y',v)= \frac{C_l^\alpha(v)}{\|C_l^\alpha\|^2}\int_{-1}^1 h(y',u)C_l^\alpha(u)~d\mu_\alpha(u).
\end{equation}

Finally, we introduce one last operator $\Theta$ going from $\confL{\nu}{n-1}$ to $L^2_\lambda\left(\Omega_{n-1}\right) \hat{\otimes}~ \C C^\alpha_l$ defined, for $f\in \confL{\nu}{n-1}$ and $(x',v)\in \Omega_{n-1}\times (-1,1)$ by
\begin{equation}
\Theta f(x',v)= Q_{1,n-2}(x')^{\frac{l}{2}}\frac{C_l^\alpha(v)}{\|C_l^\alpha\|}f(x').
\end{equation}
This is a one-to-one isometry, whose inverse is given by
\begin{equation}
\Theta^{-1}(f\times C_l^\alpha)(x')= Q_{1,n-2}(x')^{-\frac{l}{2}}\|C_l^\alpha\| f(x').
\end{equation}

We sum up the situation with the diagram in the following Theorem.

\begin{theorem}\label{thm:diagramm_conform}
Let $\lambda,~\nu \in \N$ such that $\lambda>n-1$ and $l=\nu-\lambda\in\N$. The following diagram is commutative:\\
 \xymatrix{
    \confL{\lambda}{n} \ar[rr]^{c_1\widehat{\Juhl}} \ar[d]_{\iota^*}  & &\confL{\nu}{n-1} \ar[d]^{\Theta} \ar[r]^{c_2\phi_\lambda^\nu} & \confL{\lambda}{n}\\
    L^2_\lambda(\Omega_{n-1})\otimes L^2\left((-1;1),d\mu_\alpha(v)\right) \ar[rr]_{~~~~~~P_l} && L^2_\lambda\left(\Omega_{n-1}\right)\otimes\C C_l^{\alpha}\ar@{^{(}->} [ur]_{{\iota^*}^{-1}}
  }
  
  where $c_1=\frac{2^{-\frac{1}{2}}i^l}{\|C_l^\alpha\|}$ and $c_2=\frac{1}{\|C_l^\alpha\|}$.
\end{theorem}

\begin{proof}
Let $f\in\confL{\lambda}{n}$ and $(x',v)\in \Omega_{n-1}\times(-1,1)$, we then have:
\[
P_l\circ \iota^* f(y',v)= \frac{C_l^\alpha(v)}{\|C_l^\alpha\|^2}\int_{-1}^1f(y',-Q_{1,n-2}(x)u) C_l^\alpha(u)~d\mu_\alpha(u).
\]
Comparing this formula with the definition \eqref{eq:SBOperatorL2Conforme} of $\widehat{\Juhl}$ this proves
\[c_1\widehat{\Juhl}=\Theta^{-1}\circ P_l\circ \iota^*.\]
We recall that $(I_lC_l^\alpha)(u^2,v)=u^lC_l^\alpha\left(\frac{v}{u}\right)$, so that for $f\in\confL{\nu}{n-1}$ and $x\in \Omega_n$:
\begin{align*}
{\iota^*}^{-1}\circ \Theta f(x)&=C_l^\alpha\left(\frac{-x_n}{Q_{1,n-1}(x)^\frac{1}{2}}\right)\frac{Q_{1,n-2}(x)^\frac{l}{2}}{\|C_l^\alpha\|^2} f(x')\\
&=c_2 I_lC_l^\alpha(Q_{1,n-2}(x),-x_n) f(x')=c_2\phi_\lambda^\nu f(x).
\end{align*}
\end{proof}

This Theorem implies the following Corollary.
 \begin{coro}
The orthogonal projection $P_l$ is a symmetry breaking operator for the $SO_0(2,n-1)$ action in the stratified model, and the Hilbert space decomposition 
\begin{equation}
L^2\left(\Omega_{n-1}\times (-1;1),Q_{1,n-2}(y')^{\lambda-\frac{n-1}{2}}~dy'd\mu_\alpha(v)\right)
\simeq {\sum_{l\geq 0}}^\oplus \confL{\lambda}{n-1}\hat{\otimes} \C C_l^\alpha,
\end{equation}
implements the branching rule for the restriction of the holomorphic discrete series representation $\pi_\lambda$ of $SO_0(2,n)$ to $SO_0(2,n-1)$.
\end{coro}

\subsection{Restriction to the subgroup $SO_0(2,n-p)$}\label{sec:SO(2,n)SO(2,n-p)}

Using induction over $p$, we are going to investigate the restriction of scalar valued holomorphic discrete series of $SO_0(2,n)$ to $SO_0(2,n-p)$. Let $n\geq 3$ and $p\geq 1$ such that $n-p\geq 2$. Recall that $V_n=\R^n$ is equipped with the structure of a Euclidean Jordan algebra as in Example \ref{ex:JordanAlgebra} \eqref{ex:JordanAlgebraRank2} where the inner product \eqref{eq:InnerProductRank2} corresponds to the usual inner product $\langle\cdot,\cdot\rangle$ on $\R^n$. Using the natural embedding of $V_{n-p}$ in $V_n$, we consider $V_{n-p}$ as a unital Jordan subalgebra of $V_n$. $V_n$ is simple and has rank $2$, and $V_{n-p}$ has the same rank but is not simple if $n-p=2$. The irreducible symmetric cone associated to $V_n$ is (see Example \ref{ex:SymmetricCone} \eqref{ex:SymmetricConeRank2})
\begin{equation}
\Omega_n =\left\{(x_1,\cdots,x_n)~|~x_1>0,~Q_{1,n-1}(x)>0\right\}.
\end{equation}
The orthogonal complement of $V_{n-p}$ in $V_n$ is isomorphic to $\R^p$ and the subset 
\[X=\{u\in V_{n-p}^{\bot}~|~e+u\in \Omega_2\},\]
defined by equation \eqref{def:SubsetXGeneral} becomes 
\begin{equation}\label{def:setXSO(2,n-p)}
X=\left\{(0,\cdots,0,v_1,\cdots,v_p)\in V_2~|~1-\sum_{k=1}^p v_k^2>0\right\}\simeq\unitball{p},
\end{equation}
where $\unitball{p}$ denotes the unit ball in $\R^p$.
For $x\in V_{n-p}$ and $v\in X$, using equation \eqref{eq:SquareRootConforme} for $x^\frac{1}{2}$, we have:
\begin{equation}
P(x^\frac{1}{2})v=Q_{1,n-p-1}(x)^\frac{1}{2}v.
\end{equation}

Hence, using the symmetry $v\mapsto -v$ of $X$, the diffeomorphism $\iota_{n-p}^n:\Omega_{n-p}\times X \to \Omega_n$ defined by equation \eqref{def:DiffeoGeneral} is given by:
\begin{equation}\label{def:diffeoSO(2,n)SO(2,n-p)}
\iota_{n-p}^n(x,v)=x-Q_{1,n-p-1}(x)^\frac{1}{2}v=\left(x_1,\cdots,x_{n-p},-Q_{1,n-p-1}(x)^\frac{1}{2}v_1,\cdots,-Q_{1,n-p-1}(x)^\frac{1}{2}v_p\right).
\end{equation}

As a consequence of Proposition \ref{prop:Hilbert_space_isom_EqualRank}, the pullback of the map $\iota_{n-p}^n$ yields, for $\lambda>n-1$, the following isomorphism of Hilbert spaces:
\begin{equation}
L^2_\lambda(\Omega_n)\simeq L^2_\lambda(\Omega_{n-p})\hat{\otimes}L^2(\unitball{p},d\mu_\alpha (v)),
\end{equation}
where $\alpha=\lambda-\frac{n-1}{2}$ as in Definition \eqref{def:paramAlpha}, and where we set $d\mu_\alpha(v)=2^{-\frac{p}{2}}(1-\|v\|^2)^{\alpha-\frac{1}{2}}~dv$ with the Lebesgue measure $dv$.

Let us introduce the family $\{^{(p)}P_\textbf{k}^\alpha\}$ of polynomials in $p$ variables defined, for $\textbf{k}\in \N^p$ and $\lambda>n-1$ by
\begin{equation}\label{def:PolOrthoBouleUnite}
^{(p)}P_\textbf{k}^\alpha(v)=\prod_{j=1}^p(1-\|v^{(j-1)}\|^2)^\frac{k_j}{2}\times C_{k_j}^{\alpha+|k_{(j+1)}|+\frac{p-j}{2}}\left(\frac{v_j}{(1-\|v^{(j-1)}\|^2)^\frac{1}{2}}\right),
\end{equation}
where $C_k^\alpha$ denote the Gegenbauer polynomials (see \eqref{def:GegenbauerPolynomials}). The polynomial $^{(p)}P_\textbf{k}^\alpha$ is of degree $|\textbf{k}|$ and the family $\{^{(p)}P_\textbf{k}^\alpha\}$ is an orthogonal Hilbert basis of $L^2(\unitball{p},d\mu_\alpha (v))$ (\cite{DunklXu}, Prop.2.3.2). Finally, if the context is clear one can drop the superscript $(p)$. Using the inflated polynomials $I_lC_l^\alpha$ defined in \eqref{def:InflatedGegenbauerPolynomial}, we can rewrite these polynomials:
\begin{equation}
P_\textbf{k}^\alpha(v)=\prod_{j=1}^p I_{k_j}C_{k_j}^{\alpha+|k_{(j+1)}|+\frac{p-j}{2}}\left((1-\|v^{(j-1)}\|^2)^\frac{1}{2},v_j\right).
\end{equation}

The group $SO_0(2,n-p)$ is viewed as the subgroup of $SO_0(2,n)$ which stabilizes the  last $p$ coordinates when acting on $\R^{n+2}$. We now consider, for $\lambda\in \N$ such that $\lambda>n-1$, a scalar valued holomorphic discrete series representation $\pi^{(n)}_\lambda$ of $SO_0(2,n)$ as defined in \eqref{def:HolomorphicDiscreteScalarValued}. Its restriction to the subgroup 
$SO_0(2,n-p)$ admits the following branching rule, which is deduced from \eqref{eq:branchingSO(2,n)SO(2,n-1)} by induction on $p$:
\begin{equation}
\left.\pi^{(n)}_\lambda\right|_{SO_0(2,n-p)}\simeq {\sum_{k\geq 0}}^\oplus \binom{p+k-1}{p-1}\pi^{(n-p)}_{\lambda+k}.
\end{equation}

Similarly to the $n$-fold tensor product case, we prove a one-to-one correspondence between symmetry breaking operators and orthogonal polynomials on $L^2(\unitball{p},d\mu_\alpha(v))$. More precisely, we have the following Theorem.
\begin{theorem}\label{thm:SBO_SO(2,n)SO(2n-p)}
Let $\lambda\in\N$ such that $\lambda>n-1$, and $Pol_k(\unitball{p})$ be the space of polynomials in $p$ variables of degree $k$ which are orthogonal to all polynomials of smaller degree with respect to the inner product on $L^2(\unitball{p},d\mu_\alpha(v))$.\\
Define the operator $\Phi_k^\alpha(P)$ for a polynomial $P\in Pol(\unitball{p})$, $ f\in L^2_\lambda(\Omega_{n-p})\hat{\otimes} L^2(\unitball{p}, d\mu_\lambda(v))$ and $x\in \Omega_{n-p}$ by 
\begin{equation}
\Phi_k^\alpha(P)f(x)=Q_{1,n-p-1}(x)^{-\frac{k}{2}}\int_{\unitball{p}}f(x,v)P(v)~d\mu_\alpha(v).
\end{equation}
Then $\Phi_k^\alpha \in \Hom_{SO_0(2,n-p)}(\pi^{(n)}_\lambda,\pi^{(n-p)}_{\lambda+k})$ if and only if $P\in Pol_k(\unitball{p})$.
\end{theorem}

\begin{proof}
Let us introduce the diffeomorphism $\theta:\unitball{p-1}\times (-1,1)\to \unitball{p}$ defined by:
\[\theta(x,v)=(x,(1-\|x\|^2)^\frac{1}{2}v).\]
Its Jacobian is given by $|Jac(\theta)|=(1-\|x\|^2)^\frac{1}{2}$ and it satisfies the following identity:
\[1-\|\theta(x,u)\|^2=(1-\|x\|^2)(1-u^2).\]

Let $\textbf{k}=(k_1,\cdots,k_p)\in \N^p$ such that $|\textbf{k}|=k$. The following identity for the family of orthogonal polynomials $\{^{(p)}P_\textbf{k}^\alpha\}$ defined in \eqref{def:PolOrthoBouleUnite} and $(x,v)\in \unitball{p-1}\times (-1,1)$ holds:
\begin{align*}
&^{(p)}P_\textbf{k}^\alpha(\theta(x,v))\\
&=(1-\|v\|^2)^\frac{k_p}{2}C_{k_p}^\alpha(u)\prod_{j=1}^{p-1}(1-\|v^{(j-1)}\|^2)^\frac{k_j}{2}\times C_{k_j}^{\alpha+|k_{(j+1)}|+\frac{p-j}{2}}\left(\frac{v_j}{(1-\|v^{(j-1)}\|^2)^\frac{1}{2}}\right)\\
&={}^{(p-1)}P_{\textbf{k}'}^{\alpha'}(x)(1-\|v\|^2)^\frac{k_p}{2}C_{k_p}^\alpha(u),
\end{align*}
where $\textbf{k}'=(k_1,\cdots,k_{p-1})$ and $\alpha'=\alpha+k_p+\frac{1}{2}$.

We use induction on $p$ to prove that $\Phi_k^\alpha({}^{(p)}P_\textbf{k}^{\alpha})$, denoted $^{(p)}\Phi_\textbf{k}^\alpha$, is a symmetry breaking operator. The case $p=1$ corresponds to Theorem \ref{thm:diagramm_conform}. More precisely, we are going to prove that the following diagram is commutative.
\[\xymatrix{
    L^2_\lambda(\Omega_n) \ar[rr]^{{\iota_{n-p}^n}^*}\ar[d]_{{\iota_{n-1}^n}^*}  && L^2_\lambda(\Omega_{n-p})\hat{\otimes}L^2(\unitball{p},d\mu_\alpha(v)) \ar[d]^{^{(p)}\Phi_{\textbf{k}}^\alpha}\\
    L^2_\lambda(\Omega_{n-1})\hat{\otimes}L^2((-1,1),(1-v^2)^{\alpha-\frac{1}{2}}dv)\ar[d]_{T_{k_p}^{\alpha}} && L^2_{\lambda+k}(\Omega_{n-p})\\
 L^2_{\lambda'}(\Omega_{n-1}) \ar[rr]_{{\iota_{n-p}^{n-1}}^*}&& L^2_{\lambda'}(\Omega_{n-p})\hat{\otimes}L^2(\unitball{p-1},d\mu_{\alpha'}(v))\ar[u]_{ ^{(p-1)}\Phi_{\textbf{k}'}^{\alpha'}}
  }\]
where the operator $T_{k_p}^\alpha$ is defined, for $f\in L^2_\lambda(\Omega_{n-1})\hat{\otimes}L^2((-1,1),(1-v^2)^{\alpha-\frac{1}{2}}dv)$ and $y\in \Omega_{n-1}$, by
\[
T_{k_p}^\alpha f(y)=Q_{1,n-2}(y)^{-\frac{k_p}{2}} \int_{-1}^1 f(y,v)C_{k_p}^\alpha(v)~(1-v^2)^{\alpha-\frac{1}{2}}dv.
\]
Notice that $ T_{k_p}^\alpha= \Theta\circ P_{k_p}$ where $\Theta$ and $P_{k_p}$ are defined as in Theorem \ref{thm:diagramm_conform}.

Let $f\in L^2_{\lambda}(\Omega_n)$ and $x\in \Omega_{n-p}$. On one hand, we have:
\begin{align*}
&\Phi_k^\alpha({}^{(p)}P_\textbf{k}^{\alpha})\circ {\iota_{n-p}^n}^*f(x)\\
&=Q_{1,n-p-1}(x)^{-\frac{k}{2}}\int_{\unitball{p}}f(\iota_{n-p}^n(x,v)){}^{(p)}P_\textbf{k}^{\alpha}(v)d\mu_\alpha(v)\\
&=Q_{1,n-p-1}(x)^{-\frac{k}{2}}\int_{\unitball{p-1}\times (-1,1)}f(\iota_{n-p}^n(x,\theta(w,u))){}^{(p-1)}P_{\textbf{k}'}^{\alpha'}(w)C_{k_p}^\alpha(u)~(1-u^2)^{\alpha-\frac{1}{2}}dud\mu_{\alpha+\frac{k_p+1}{2}}(w).
\end{align*}
On the other hand, we have:
\begin{align*}
&\Phi_k^\alpha({}^{(p)}P_\textbf{k}^{\alpha})\circ {\iota_{n-p}^{n-1}}^*\circ T_{k_p}^\alpha\circ {\iota_{n-1}^n}^*f(x)\\
&=Q_{1,n-p-1}(x)^{-\frac{k'}{2}}\int_{\unitball{p-1}\times (-1,1)}Q_{1,n-2}(\iota_{n-p}^{n-1}(x,w))^{-\frac{k_p}{2}}f(\iota_{n-1}^n(\iota_{n-p}^{n-1}(x,w),u))\\
&~~~~~~\times C_{k_p}^\alpha(u){}^{(p-1)}P_{\textbf{k}'}^{\alpha'}(w)(1-u^2)^{\alpha-\frac{1}{2}}dud\mu_{\alpha'}(w).
\end{align*}
For $x\in \Omega_{n-p}$, $w\in \unitball{p-1}$ and $u\in (-1,1)$, notice that
\[\iota_{n-1}^n(\iota_{n-p}^{n-1}(x,w),u)=\iota_{n-p}^n(x,\theta(w,u)),\]
and 
\[Q_{1,n-2}(\iota_{n-p}^{n-1}(x,w))=Q_{1,n-p-1}(x)(1-\|w\|^2).\]
Hence, this proves that
\[\Phi_{k'}^{\alpha'}({}^{(p-1)}P_{\textbf{k}'}^{\alpha'})\circ {\iota_{n-p}^{n-1}}^*\circ T_{k_p}^\alpha\circ {\iota_{n-1}^n}^*=\Phi_k^\alpha({}^{(p)}P_\textbf{k}^{\alpha})\circ {\iota_{n-p}^n}^*.\] 
Finally, Theorem \ref{thm:diagramm_conform} implies that $T_{k_p}^\alpha$ is a symmetry breaking operator and by assumption $\Phi_{k'}^{\alpha'}({}^{(p-1)}P_{\textbf{k}'}^{\alpha'})$ is also a symmetry breaking operator. The operator $\Phi_k^\alpha(P)$ depends linearly on $P$, and the map $P\mapsto \Phi_k^\alpha(P)$ is injective. Moreover, the spaces $Pol_k(\unitball{p})$ and $\Hom_{SO_0(2,n-p)}(\pi^{(n)}_\lambda,\pi^{(n-p)}_{\lambda+k})$ have the same dimension equal to $\binom{p+k-1}{p-1}$, hence the map $P\mapsto \Phi_k^\alpha(P)$ is actually surjective onto $\Hom_{SO_0(2,n-p)}(\pi^{(n)}_\lambda,\pi^{(n-p)}_{\lambda+k})$.
\end{proof}
The following Corollary describes the holographic operators for the restriction of $\pi_\lambda^{(n)}$ to $SO_0(2,n-p)$.
\begin{coro}
Let $\lambda \in \N$ such that $\lambda>n-1$ and $k\in \N$. Define, for $P\in Pol(\unitball{p})$ and $f\in L^2_{\lambda+k}(\Omega_{n-p})$, the operator $\psi_k^\alpha(P)$ by:
\begin{equation}
\psi_k^\alpha(P)f(t,v)=P(v)Q_{1,n-p-1}(t)^{\frac{k}{2}}f(t).
\end{equation}
Then $\psi_k^\alpha(P)\in \Hom_{SO_0(2,n-p)}(\pi^{(n-p)}_{\lambda+k},\pi^{(n)}_\lambda)$ if and only if $P\in Pol_k(\unitball{p})$.
\end{coro}
Looking at the intertwining relation for the inversion $j$:
\begin{equation}
\pi^{(n-1)}_{\lambda+k}(j)\circ \psi_k^\alpha(P)=\psi_k^\alpha(P)\circ \pi^{(n)}_\lambda(j),
\end{equation}
one can relate Bessel functions $J_\lambda^{(n)}$, defined in \eqref{def:BesselFunctionJordan}, over the cone $\Omega_n$ with Bessel functions $J_\lambda^{(n-p)}$ over the cone $\Omega_{n-p}$. The following Corollary gives the analogue of formula \eqref{eq:objectif} in this setting.
\begin{coro}\label{coro:FormulaBesselFunctionsSO(2,n-p)}
Let $\lambda\in \N$ such that $\lambda>n-1$, $k\in \N$ and for $P\in Pol_k(\unitball{p})$. Then:
\begin{multline}
\frac{P(v)Q_{1,n-p-1}(t)^\frac{1}{2}Q_{1,n-p-1}(x)^\frac{1}{2}}{\Gamma_{\Omega_{n-p}}(\lambda+k)i^{(\lambda+k)(n-1)}}J_{\lambda+k}^{(n-p)}\left(P(x^\frac{1}{2})t \right)\\
=\frac{1}{\Gamma_{\Omega_n}(\lambda)i^{n\lambda}}\int_{\unitball{p}}P(u)J_\lambda^{(n)}\left(P(\iota_{n-p}^n(x,u)^\frac{1}{2})\iota_{n-p}^n(t,v) \right)~d\mu_\alpha(u).
\end{multline}
\end{coro}

\paragraph{A proof of Theorem \ref{thm:SBO_SO(2,n)SO(2n-p)} using Bessel operators}$\ $

In this paragraph, we give an alternative proof of Theorem \ref{thm:SBO_SO(2,n)SO(2n-p)} based on the study of the Bessel operator $\mathcal{B_\lambda}$ defined in \eqref{def:BesselOperator}. More precisely, considering the derived representation of the corresponding Lie algebra, we are going to prove that spaces of the form $L^2_\lambda(\Omega_{n-p})\hat{\otimes}\C\cdot P_l^\alpha$, with $P_l^\alpha$ being the polynomials defined in \eqref{def:PolOrthoBouleUnite}, are irreducible summands in the stratified model of the representation $\pi^{(n)}|_{SO_0(2,n-p)}$. This approach is independent from the results of \cite{KP}.

First, we denote by $N_1$ the subgroup of translations in $G_{T_{\Omega_{n-p}}}\simeq SO_0(2,n-p)$ and by $K_1\simeq SO(n-p)$ the isotropy subgroup of $e=(1,0,\cdots,0)$ under the action of $SO_0(1,n-p)\simeq G_{\Omega_{n-p}}$. Notice that the action of $K_1$ on $Pol_k(\unitball{p})$ defined in \eqref{eq:actionCompactSubgroupPolynomial} by 
\begin{equation}
k\cdot P(v)=P(k^{-1}\cdot v),
\end{equation}
is trivial in this situation because the action of the whole group $G_{\Omega_{n-p}}$ on the set $X\simeq \unitball{p}$ defined in \eqref{def:setXSO(2,n-p)} is already trivial. Hence, Proposition \ref{prop:restrictionParabolicVectorValued} proves that the space $L^2_\lambda(\Omega_{n-p})\hat{\otimes}\C\cdot P_l^\alpha$ is irreducible under the actions of $N_1$ and $G_{\Omega_{n-p}}$.

We denote by $\n_1$ the Lie algebra of $N_1$ and $\mathfrak{l}_1$ the Lie algebra of $G_{\Omega_1}$. The Lie algebra $\g(T_{\Omega_1})$ of $SO_0(2,n-p)$ admits the Gelfand-Naimark decomposition \eqref{eq:decompositionLieAlgConformal}:
\begin{equation}
\g(T_{\Omega_1})=\n_1\oplus \mathfrak{l}_1 \oplus \bar{\n}_1.
\end{equation}
Hence, we have to prove that $L^2_\lambda(\Omega_{n-p})\hat{\otimes}\C\cdot P_l^\alpha$ is stable under the action of $\bar{\n}_1$. We have $\bar{\n}_1\subset \bar{\n}_2$ so that in the $L^2$-model the action of $X\in\bar{\n}_1$ is given by the operator:
\begin{equation}
d\rho_\lambda(X)f(x)=i(X|\mathcal{B}^{(n)}_\lambda f(x)),
\end{equation}
where $\mathcal{B}^{(n)}_\lambda$ denotes the Bessel operator for the Jordan algebra $V_n=\R\times \R^{n-1}$.

Using the pullback by the diffeomorphism $\iota_{n-p}^n~:~\Omega_{n-p}\times \unitball{p}\to \Omega_n$ defined in \eqref{def:diffeoSO(2,n)SO(2,n-p)}, we transfer this operator in the stratified model. To understand the $\bar{\n}_1$-action in the stratified model, we only need to understand the $V_{n-p}$-component of the Bessel operator, denoted $\left.\mathcal{B}^{(n)}_\lambda\right|_{n-p}$, in the coordinates $(x,v)\in \Omega_{n-p}\times \unitball{p}$.
This is given in the following proposition.
\begin{prop}\label{prop:RestrictionBesselOperator}
In the coordinates $(x,v)\in \Omega_{n-p}\times \unitball{p}$, the $V_{n-p}$-component of the Bessel operator $\left.\mathcal{B}^{(n)}_\lambda\right|_{n-p}$ becomes:
\begin{equation}
\left.\mathcal{B}^{(n)}_\lambda\right|_{n-p} =\mathcal{B}^{(n-1)}_\lambda +\frac{x^{-1}}{4}\left(\sum_{i}\frac{\partial^2}{\partial v_i^2} -\sum_{i,j}v_iv_j\frac{\partial^2}{\partial v_i\partial v_j}-(2\lambda-n+p+1)\sum_{i}v_i\frac{\partial}{\partial v_i}\right).
\end{equation}
\end{prop}

Before proving Proposition \ref{prop:RestrictionBesselOperator}, assume this is true and let us prove Theorem \ref{thm:SBO_SO(2,n)SO(2n-p)}.
\begin{proof}[Proof of Theorem \ref{thm:SBO_SO(2,n)SO(2n-p)}]
The fact that the operator $\Phi_k^\alpha(P)$ intertwines the $N_1$ and the $G_{\Omega_{n-p}}$ actions in the stratified model with the representation $\pi^{(n-1)}_{\lambda+k}$ is a consequence of \eqref{eq:actionRestrictionTranslation}, \eqref{eq:actionRestrictionStructure} and the fact that the $K_1$-action is trivial on the set $X\simeq \unitball{p}$.

For the $\bar{\n}_1$-action, the intertwining property is equivalent to the following identity for Bessel operators, on $f\in L^2_\lambda(\Omega_{n-p})\hat{\otimes} Pol_k(\unitball{p})$.
\begin{equation}\label{eq:egaliteBesselOperators}
\Delta(x)^{-\frac{k}{2}}\left.\mathcal{B}^{(n)}_\lambda\right|_{n-p}f=\mathcal{B}_\nu^{(n-p)} \Delta(x)^{-\frac{k}{2}}f,
\end{equation}
where $\Delta(x)=Q_{1,n-p-1}(x)$ is the Jordan determinant of $x$ in the Jordan algebra $\R\times\R^{n-p-1}$.

The following equality is true for a smooth function $f$ on $V$ (see \cite{Far_Kor}, proof of Prop.XV.2.4):
\begin{equation}
\mathcal{B}_\lambda (\Delta(x)^\mu f) =\Delta(x)^\mu(\mathcal{B}_{\lambda+2\mu}f+\mu(\mu+\lambda-\frac{n}{r})x^{-1}f).
\end{equation}

Using this formula and $\nu=\lambda+k$, the equation \eqref{eq:egaliteBesselOperators} becomes 
\begin{equation}
\left.\mathcal{B}^{(n)}_\lambda\right|_{n-p}f(x,v)=\mathcal{B}^{(n-p)}_\lambda f(x,v)-\frac{k}{2}(\frac{k}{2}+\lambda-\frac{n-p}{2})x^{-1}f(x,v).
\end{equation}
Now, using Proposition \ref{prop:RestrictionBesselOperator}, equation \eqref{eq:egaliteBesselOperators} is equivalent to 
\begin{multline}
\sum_{i}\frac{\partial^2 f}{\partial v_i^2}(x,v) -\sum_{i,j}v_iv_j\frac{\partial^2 f}{\partial v_i\partial v_j}(x,v)-(2\lambda-n+p+1)\sum_{i}v_i\frac{\partial f}{\partial v_i}(x,v)\\
+k(k+2\lambda-n+p)f(x,v)=0.
\end{multline}
Finally, a direct computation shows that this is equivalent to:
\begin{equation}
\sum_{i=1}^p\frac{\partial^2 f}{\partial v_i^2} f-\sum_{i=1}^p\frac{\partial}{\partial v_i} v_i\left[(2\alpha-1)f+\sum_{j=1}^p v_j\frac{\partial f}{\partial v_j}\right]+(k+p)(k+2\alpha-1)f=0,
\end{equation}
where $\alpha=\lambda-\frac{n-1}{2}$.

It is known that polynomials in $Pol_k(\unitball{p})$ are solutions to this partial differential equation (\cite{DunklXu}, section 2.3.2), so that \eqref{eq:egaliteBesselOperators} is always true on $L^2_\lambda(\Omega_{n-p})\hat{\otimes} Pol_k(\unitball{p})$. What concludes the proof.
\end{proof}

Now let us prove Proposition \ref{prop:RestrictionBesselOperator}. First, we need two technical lemmas.
\begin{lemme}\label{lem:QudraticPolarizationBasis}
Let $\{e_i\}_{1\leq i\leq n}$ be the canonical basis of $\R^n$. Then we have:
\begin{enumerate}
\item $P(e_i,e_j)e_k=0$ if $i\neq j\neq k$.
\item $P(e_i,e_k)e_k=e_i$.
\item $P(e_i,e_i)e_k=\epsilon(i,k)e_k$ where $\epsilon(i,k)=1$ if $i=1$, $k=1$ or $i=k$ and $\epsilon(i,k)=-1$ otherwise.
\end{enumerate}
\end{lemme}

\begin{proof}
Using the definition \eqref{eq:QuadraticRepresentationPolarization} of $P$, we have:
\[P(e_i,e_j)e_k=e_i(e_je_k)+e_j(e_ie_k)-(e_ie_j)e_k.\]
Using the definition of the product \eqref{eq:ProductJordanSO(2,n)} in the Jordan algebra $V_n=\R\times \R^{n-1}$, we have:
\begin{enumerate}
\item $e_i^2=e_1$.
\item $e_1\cdot e_j=e_j$.
\item $e_i\cdot e_j=0$ otherwise.
\end{enumerate}
A case by case computation ends the proof.
\end{proof}

The following Lemma describes the partial derivatives of a smooth function $g$ which appears in the Bessel operator \eqref{def:BesselOperator} in terms of the new coordinates $(x,v)\in \Omega_{n-p}\times \unitball{p}$.
\begin{lemme}\label{lem:PartialDerivativeSO(2,n)}
Let $g$ be a smooth function on the cone $\Omega_n$. Set $z=\iota_{n-p}^n(x,v)$ with $(x,v)\in \Omega_{n-p}\times \unitball{p}$, so that
\[z= (x,-\Delta(x)^{\frac{1}{2}}v).\]
Set
\begin{equation}
\epsilon_i=1 \text{ if } i=1 \text{ and } -1 \text{ otherwise}.
\end{equation}
The first order derivatives in these coordinates are given by:
\begin{align}
&\frac{\partial g}{\partial z_j}=\frac{\partial g}{\partial x_j}-\epsilon_jx_j\Delta(x)^{-1}\sum_{k=1}^p v_k \frac{\partial g}{\partial v_k} &\text{ if } 1\leq j \leq n-p,\\
&\frac{\partial g}{\partial z_j}=-\Delta(x)^{-\frac{1}{2}}\frac{\partial g}{\partial v_j} &\text{ if } n-p+1\leq j \leq n.
\end{align}
The second order derivatives are given by:

\begin{align}
&\frac{\partial^2 g}{\partial z_i\partial z_j} =\frac{\partial^2 g}{\partial x_i\partial x_j}-\epsilon_i x_i\Delta(x)^{-1}\sum_{k=1}^p v_k\frac{\partial^2 g}{\partial x_i\partial v_k}-\epsilon_j x_j\Delta(x)^{-1}\sum_{k=1}^p v_k\frac{\partial^2 g}{\partial x_j\partial v_k}\\
&~~~~~~~~~~+ (3\epsilon_ix_jx_i\Delta(x)^{-1}- \delta_{ij})\epsilon_j\Delta(x)^{-1}\sum_{k=1}^p v_k \frac{\partial g}{\partial v_k} +\epsilon_i \epsilon_jx_ix_j\Delta(x)^{-2}\sum_{k,l=1}^p v_kv_l \frac{\partial^2 g}{\partial v_k\partial v_l} \nonumber\\
&~~~~~~~~~~\text{ if } 1\leq i,j \leq n-p,\nonumber\\
&\frac{\partial^2 g}{\partial z_i\partial z_j}= \epsilon_ix_i\Delta(x)^{-\frac{3}{2}}\frac{\partial g}{\partial v_j}-\Delta(x)^{-\frac{1}{2}}\frac{\partial^2 g}{\partial x_i\partial v_j}+\epsilon_ix_i\Delta(x)^{-\frac{3}{2}}\sum_{k=1}^p v_k \frac{\partial^2 g}{\partial v_j\partial v_k} \\
&~~~~~~~~~~\text{ if } 1\leq i\leq n-p,\ n-p+1\leq j \leq n,\nonumber\\
&\frac{\partial^2 g}{\partial z_i\partial z_j}=\Delta(x)^{-1}\frac{\partial^2 g}{\partial v_i\partial v_j} \text{ if } n-p+1\leq i,j \leq n.
\end{align}
\end{lemme}

\begin{proof}
The proof is a direct computation based on the chain rule and the following formulas
\begin{align*}
&\frac{\partial x_i}{\partial z_j}=\delta_{ij},\\
&\frac{\partial v_i}{\partial z_j}=- \epsilon_jx_j\Delta(x)^{-1}v_i \text{ for } 1\leq j\leq n-p,\\
&\frac{\partial v_i}{\partial z_j}=-\delta_{ij}\Delta(x)^{-1}\text{ for } n-p+1\leq j\leq n.
\end{align*}
\end{proof}

We are now ready to prove Proposition \ref{prop:RestrictionBesselOperator}.
\begin{proof}[Proof of Proposition \ref{prop:RestrictionBesselOperator}]
The canonical basis $\{e_i\}$ of $\R^n$ is orthogonal but not orthonormal with respect to the trace form $\tr(xy)$, so we choose the family $\{e'i=\frac{e_i}{\sqrt{2}}\}$ as an orthonormal basis. Recall that in the orthonormal basis $\{e_i'\}$ the Bessel operator \eqref{eq:BesselOperatorBasis}, for a smooth function $g$ on $\Omega_n$ is given by: 
\begin{align*}
\mathcal{B}^{(n)}_\lambda g(z)&=\sum_{k,l}\frac{\partial^2 g}{\partial z'_k\partial z'_l}P(e'_k,e'_l)z+\lambda\sum_k \frac{\partial g}{\partial z_k} e'_k\\
&=\frac{1}{4}\sum_{k,l}\frac{\partial^2 g}{\partial z_k\partial z_l}P(e_k,e_l)z+\frac{\lambda}{2}\sum_k \frac{\partial g}{\partial z_k} e_k.
\end{align*}
Writing $z=\sum_k z_ke_k$, using Lemma \ref{lem:QudraticPolarizationBasis} the $V_{n-p}$ component of the Bessel operator is:
\[
\left.\mathcal{B}^{(n)}_\lambda\right|_{n-p} g(z)=\frac{1}{4}\sum_{k=1}^{n-p}z_k\sum_{i=1}^n \frac{\partial^2 g}{\partial z_i^2}\epsilon(i,k)e_k+\frac{1}{2}\sum_{k=1}^{n-p}\sum_{\underset{i\neq k}{i=1}}^nz_i\frac{\partial^2 g}{\partial z_i\partial z_k}  e_k+ \frac{\lambda}{2}\sum_{k=1}^{n-p} \frac{\partial g}{\partial z_k} e_k.
\]
Thus, the $l$-th component of $\mathcal{B}^{(n)}_\lambda g(x)$, for $1\leq l\leq n-p$, is given by:
\[
\frac{1}{4}z_l\sum_{i=1}^n\epsilon(i,l)\frac{\partial^2 g}{\partial z_i^2}+\frac{1}{2}\sum_{\underset{i\neq l}{i=1}}^n z_i\frac{\partial^2 g}{\partial z_i\partial z_l} +.
\frac{\lambda}{2}\frac{\partial g}{\partial z_l}
\]
In the coordinates $(x,v)\in \Omega_{n-p}\times \unitball{p}$, using Lemma \ref{lem:PartialDerivativeSO(2,n)}, we get
\begin{align*}
&z_l\sum_{i=1}^n\epsilon(i,l)\frac{\partial^2 g}{\partial z_i^2}=z_l\sum_{i=1}^{n-p}\epsilon(i,l)\frac{\partial^2 g}{\partial z_i^2}+z_l\sum_{i=n-p+1}^n\epsilon(i,l)\frac{\partial^2 g}{\partial z_i^2} \\
&=\sum_{i=1}^{n-p}\epsilon(i,l)\big[\frac{\partial^2 g}{\partial x_i^2}-2\epsilon_ix_i\Delta(x)^{-1}\sum_{k=1}^p v_k\frac{\partial^2 g}{\partial x_i\partial v_k}+(3x_i^2\Delta(x)^{-1}-\epsilon_i)\Delta(x)^{-1}\sum_{k=1}^p v_k\frac{\partial g}{\partial v_k}\\
&~~~~~~~~~~~~~~~~+x_i^2\Delta(x)^{-2}\sum_{k,k'=1}^p v_kv_{k'}\frac{\partial^2 g}{\partial v_kv_{k'}}\big]+\epsilon_lx_l\Delta(x)^{-1}\sum_{i=1}^p\frac{\partial^2 g}{\partial v_i^2}.
\end{align*}
And, similarly we have:
\begin{align*}
&\sum_{\underset{i\neq l}{i=1}}^n z_i\frac{\partial^2 g}{\partial z_i\partial z_l}=\sum_{\underset{i\neq l}{i=1}}^{n-p} z_i\frac{\partial^2 g}{\partial z_i\partial z_l}+\sum_{i=n-p+1}^n z_i\frac{\partial^2 g}{\partial z_i\partial z_l}\\
&=\sum_{\underset{i\neq l}{i=1}}^{n-p}\big[x_ii\frac{\partial^2 g}{\partial x_i\partial x_l}-\epsilon_lx_lx_i\Delta(x)^{-1}\sum_{k=1}^{n-p}v_ki\frac{\partial^2 g}{\partial x_i\partial v_k}-\epsilon_ix_i^2\Delta(x)^{-1}\sum_{k=1}^{n-p}v_k\frac{\partial^2 g}{\partial x_l\partial v_k}\\
&~~~~~~+3\epsilon_l\epsilon_ix_i^2x_l\Delta(x)^{-2}\sum_{k=1}^p v_k\frac{\partial g}{\partial v_k}+\epsilon_l\epsilon_ix_i^2x_l\Delta(x)^{-2}\sum_{k,k'=1}^p v_kv_{k'}\frac{\partial^2 g}{\partial v_kv_{k'}}\big]\\
&~~~ -\epsilon_lx_l\Delta(x)^{-1}\sum_{k=1}^p v_k\frac{\partial g}{\partial v_k}+\sum_{i=1}^pv_k\frac{\partial^2 g}{\partial x_l\partial v_k}-\epsilon_lx_l\Delta(x)^{-1}\sum_{k,k'=1}^p v_kv_{k'}\frac{\partial^2 g}{\partial v_kv_{k'}}.
\end{align*}
Finally, the gradient contribution is given by:
\begin{equation}\label{eq:GradientContributionSO(2,n)}
\frac{\lambda}{2}\frac{\partial g}{\partial z_l}=\frac{\lambda}{2}\left(\frac{\partial g}{\partial x_l}-\epsilon_l x_l\Delta(x)^{-1}\sum_{k=1}^p\frac{\partial g}{\partial v_k}\right).
\end{equation}
Now, we put together the second order derivatives with respect to the $v_i$'s, and a direct computation gives the expression:
\begin{equation}\label{eq:SecondOrderv}
\frac{\epsilon_l x_l \Delta(x)^{-1} }{4}\left[\sum_{k=1}^{n-p}\frac{\partial^2 g}{\partial v_k^2}+\left(\Delta(x)^{-1}\left(\sum_{i=1}^{n-p} x_i^2 \epsilon(i,l) \epsilon_l+2\sum_{\underset{i\neq l}{i=1}}^{n-p}\epsilon_i x_i^2\right)-2\right)\sum_{k,k'=1}^p v_kv_{k'}\frac{\partial^2 g}{\partial v_k v_{k'}}\right].
\end{equation}
Moreover, we have:
\begin{equation}
\sum_{i=1}^{n-p} x_i^2 \epsilon(i,l) \epsilon_l+2\sum_{\underset{i\neq l}{i=1}}^{n-p}\epsilon_i x_i^2= \sum_{i=1}^{n-p} x_i^2(\epsilon(i,l) \epsilon_l+2\epsilon_i)-2\epsilon_l x_l^2=\Delta(x).
\end{equation}
Indeed, for $l=1$, we have:
\[\sum_{i=1}^{n-p} x_i^2(\epsilon(i,l) \epsilon_l+2\epsilon_i)-2\epsilon_l x_l^2=\sum_{i=1}^{n-p} x_i^2(1+2\epsilon_i)-2 x_1^2=\sum_{i=1}^{n-p} \epsilon_i x_i^2=\Delta(x).\]
For $l\neq 1$:
\[\sum_{i=1}^{n-p} x_i^2(\epsilon(i,l) \epsilon_l+2\epsilon_i)-2\epsilon_l x_l^2=
\sum_{i=1}^{n-p} x_i^2(2\epsilon_i-\epsilon(i,l))+2 x_l^2=\Delta(x).\]
Finally, the second order term with respect to $v_i$'s \eqref{eq:SecondOrderv} is equal to 
\[\frac{\epsilon_l x_l \Delta(x)^{-1} }{4}\left(\sum_{k=1}^{n-p}\frac{\partial^2 g}{\partial v_k^2}-\sum_{k,k'=1}^p v_kv_{k'}\frac{\partial^2 g}{\partial v_k v_{k'}}\right).\]
Putting together the partial derivatives $\frac{\partial^2 }{\partial x_l v_k}$, one gets:
\[\left(\frac{1}{2}-\frac{\Delta(x)^{-1}}{2}(\epsilon_l x_l^2+\sum_{k\neq l}\epsilon_kx_k^2)\right)\sum_{k=1}^{n-p}v_k\frac{\partial^2 g}{\partial x_l v_k}=0.\]
Similarly, putting together the partial derivatives $\frac{\partial^2 }{\partial x_i v_k}$ with $i\neq l$, one gets:
\[-\frac{\epsilon_lx_l\Delta(x)^{-1}}{2}\sum_{\underset{i\neq l}{i=1}}^{n-p}\left((\epsilon(i,l)\epsilon_i\epsilon_l +1)x_i\frac{\partial^2 g}{\partial x_i v_k}\right)=0,\]
because $\epsilon(i,l)\epsilon_i\epsilon_l=-1$ for all $i,l$ such that $i\neq l$.

If we look at the second order derivatives with respect to the variables $x_i$ we get:
\[\frac{x_l}{4}\sum_{i=1}^{n-p}\epsilon(i,l) \frac{\partial^2 g}{\partial x_i^2}+\frac{1}{2}\sum_{\underset{i\neq l}{i=1}}^{n-p}x_i\frac{\partial^2 g}{\partial x_i\partial x_l},\]
what corresponds to the $l$-th component of $\mathcal{B}^{(n-p)}_0 (g)$.
Moreover, adding the first order derivatives coming from the contribution of the gradient \eqref{eq:GradientContributionSO(2,n)} one gets the $l$-th component of $\mathcal{B}^{(n-p)}_\lambda (g)$.

To conclude the proof, we compute the coefficient of the term $\sum_{k=1}^p\frac{\partial g}{\partial v_k}$ and find:
\begin{align*}
&\frac{x_l\Delta(x)^{-1}}{4}\left(3\left(\sum_{i=1}^{n-p}(\epsilon(i,l)+2\epsilon_l\epsilon_i)x_i^2-2x_l^2\right)\Delta(x)^{-1}-2\epsilon_l-\sum_{i=1}^{n-p}\epsilon_i\epsilon(i,l)\right)-\frac{\lambda\epsilon_l x_l\Delta(x)^{-1}}{2}\\
&=\frac{x_l\Delta(x)^{-1}}{4}\left(\epsilon_l-\sum_{i=1}^{n-p}\epsilon_i\epsilon(i,l)\right)-\frac{\lambda\epsilon_l x_l\Delta(x)^{-1}}{2}\\
&=-\frac{\epsilon_l x_l\Delta(x)^{-1}}{4}(2\lambda-(n-p-1)),
\end{align*}
where we used the identities:
\[\sum_{i=1}^{n-p}(\epsilon(i,l)+2\epsilon_l\epsilon_i)x_i^2-2x_l^2=\epsilon_l\Delta(x),\]
and
\[\sum_{i=1}^{n-p}\epsilon_i\epsilon(i,l)=-(n-p-2)\epsilon_l.\]

The result is then true if one put together all the derivatives, and notice that $\epsilon_l x_l \Delta(x)^{-1}$ is the $l$-th component of $x^{-1}$.
\end{proof}

\subsection{Restriction to the subgroup $SO_0(2,n-p)\times SO(p)$}\label{sec:SO(2,n)SO(2,n-p)SO(p)}

In this section, we study the restriction of a scalar valued holomorphic discrete series representation $\pi_\lambda^{(n)}$ of $SO_0(2,n)$ restricted to the group $G_1=SO_0(2,n-p)\times SO(p)$ seen as a subgroup of $SO_0(2,n)$ via the map
\begin{equation}\label{eq:embeddingSO(2,n-p)SO(p)}
(g,k)\mapsto
\begin{pmatrix}
g&0\\0& k
\end{pmatrix}.
\end{equation}
Equivalently, $G_1$ can be viewed as the set of fixed points under the involution $\tau$ defined, on $SO_0(2,n)$, by conjugation by the matrix
\[\begin{pmatrix}
I_2 &0&0\\0&-I_{n-p}&0\\0&0&I_p
\end{pmatrix}.\] 

Using the embedding \eqref{eq:embeddingSO(2,n-p)SO(p)}, we have $SO_0(1,n-p)\times SO_0(p)\subset SO_0(1,n)$. Hence, the group $SO_0(2,n-p)\times SO(p)$ is generated by $N_1$ the subgroup of translations in the conformal group $SO_0(2,n-p)$, $G_{\Omega_{n-p}}\times SO(p)$ and the inversion $j$. Notice that this setting is different from the one considered in section \ref{sec:defintionDiffeo}, because the group $G_{\Omega_{n-p}}\times SO(p)$ actually corresponds to the subgroup of $G_{\Omega_n}$ which stabilizes $V_{n-p}$:
\begin{equation}
\left\{ g\in GL(V_n)~|~g\cdot V_{n-p}\subset V_{n-p}\right\}.
\end{equation}

The Lie algebra $\g$ of $SO_0(2,n-p)\times SO(p)$ admits the following decomposition:
\begin{equation}
\g=\n_1\oplus (\mathfrak{l_1}+\mathfrak{so}(p))\oplus \bar{\n}_1,
\end{equation}
where $\n_1\oplus \mathfrak{l_1}\oplus \bar{\n}_1$ is the Gelfand-Naimark decomposition \eqref{eq:decompositionLieAlgConformal} of the Lie algebra of $SO_0(2,n-p)$.

Let $\mathcal{H}^p_n$ be the space of harmonic polynomials of degree $n$ in $p$ variables, and recall that the representation $\eta_n$ of $SO(p)$ defined by:
\begin{equation}
\eta_n(k)P(v)=P(k^{-1}\cdot v),
\end{equation}
for $P\in \mathcal{H}^p_n$ is irreducible. The space $P\in \mathcal{H}^p_n$ can be equipped with the norm defined by
\begin{equation}\label{eq:InnerProductSpherical}
\|P\|_\mathcal{H}^2=\int_{S^{p-1}}|P(x)|^2 ~d\omega(x), 
\end{equation}
where $d\omega$ denotes the surface measure on the unit sphere $S^{p-1}$, so that representation $\eta_n$ is unitary.

The following Theorem gives the branching law for the restriction of $\pi^{(n)}_\lambda$ to $G_1$. 
\begin{theorem}\label{thm:BranchinSO(2,n)SO(2,n-p)SO(p)}
Let $\lambda\in \N$ such that $\lambda>n-1$, then the restriction to $G_1$ of the scalar valued holomorphic discrete series representation $\pi_\lambda^{(n)}$ of $SO_0(2,n)$ is given by
\begin{equation}
\left.\pi_\lambda^{(n)}\right|_{G_1}\simeq {\sum_{k\in \N}}^\oplus\bigoplus_{j=0}^{\lfloor \frac{k}{2}\rfloor}\pi^{(n-1)}_{\lambda+k}\otimes \eta_{k-2j}.
\end{equation}
\end{theorem}

\begin{proof}
Consider the representation $\pi_\lambda^{(n)}$ acting in the $L^2$-model. Then $(g,k)\in SO_0(1,n-p)\times SO(p)\subset SO_0(1,n)$ acts on $L^2_\lambda(\Omega_n)$ via the operators
\[\rho_\lambda(g,k)f(t)=\det\begin{pmatrix}g&0\\0&k\end{pmatrix}^\frac{\lambda}{2m_2}f(g'.x+k'.y)\]
where $t=x+y\in V_1\oplus V_1^\bot$.

In the stratified model, for $f\in L^2_\lambda(\Omega_{n-p})\hat{\otimes}L^2(\unitball{p},(1-\|v\|v^2)^{\lambda-m_2}dv)$, a similar computation as in the proof of \eqref{eq:actionRestrictionStructure} gives:
\[S_\lambda(g,k)f(x,v)=\det(g)^\frac{\lambda}{2m_1}f(g'\cdot x,k^{-1}\cdot v).\]
As the measure $(1-\|v\|^2)^{\lambda-m_2}dv$ is $SO(p)$ invariant, the space $Pol_l(\unitball{p})$ of polynomials of degree $l$ which are orthogonal to all polynomials of smaller degree is stable under the action of $SO(p)$, thus $Pol_l(\unitball{p})$ is a representation of $SO(p)$. Moreover, the family of polynomials defined by:
\begin{equation}\label{eq:OrthogonalBasisWj}
P_{j,k}(v)=P_j^{\lambda-\frac{n}{2},l-2j+\frac{p-2}{2}}(2\|v\|^2-1)Y_k^{l-2j}(v)
\end{equation}
forms an orthogonal basis of $Pol_l(\unitball{p})$ (see \cite{DunklXu}, Prop.2.3.1), where $\{Y_k^{l-2j}\}$ denotes an orthonormal basis of the space $\mathcal{H}^p_{l-2j}$ of spherical harmonic polynomials and recall that $P_j^{a,b}$ denotes the Jacobi polynomials. Hence, the representation $Pol_l(\unitball{p})$ of $SO(p)$ admits the following irreducible decomposition:
\[Pol_l(\unitball{p})=\bigoplus_{j=1}^{\lfloor \frac{l}{2}\rfloor} W_{l,j}\simeq \bigoplus_{j=1}^{\lfloor \frac{l}{2}\rfloor} \mathcal{H}^p_{l-2j},\]
where $W_{l,j}=P_j^{\lambda-\frac{1}{2},l-2j+\frac{p-2}{2}}(2\|v\|^2-1)\mathcal{H}^p_{l-2j}$ is the vector space generated by the family $\{P_{j,k}\}_k$. 
Finally, the Hilbert space $L^2_\lambda(\Omega_{n-p})\hat{\otimes}W_{l,j}$ is stable and irreducible under the action of $N_1$ and $SO_0(1,n-1)\times SO(p)$. Theorem \ref{thm:SBO_SO(2,n)SO(2n-p)} proves that it is also stable under the action of the inversion $j$.
\end{proof}

Let us introduce the reproducing kernel $K_n$ of $\mathcal{H}^p_{n}$ defined (see \cite{DunklXu}, Section 2.3) by
\begin{equation}
K_n(x,y)=\frac{2n+p-2}{p-2}\|x\|^n\|y\|^nC_n^{\frac{p-2}{2}}\left(\frac{\langle x,y\rangle}{\|x\|\|y\|}
 \right),
\end{equation}
where $C_n^\alpha$ is a Gegenbauer polynomial (see \eqref{def:GegenbauerPolynomials}).

Then, we get the following Corollary:
\begin{coro}
Let $\lambda\in \N$ such that $\lambda>n-1$. The operator $\Psi_{l,j}$ defined by:
\begin{equation}
\Psi_{l,j}f(x,u)=\gamma Q_{1,n-p-1}(x)^{-\frac{l}{2}}\int_{\unitball{p}} f(x,v)P_j^{\lambda-\frac{n}{2},l-2j+\frac{p-2}{2}}(2\|v\|^2-1)K_{l-2j}(u,v)~d\mu_\alpha(v),
\end{equation}
where $\gamma=\frac{2^{\frac{1}{2}(l-2j+\lambda+\frac{p-n}{2}+1)}}{\|P_j^{\lambda-\frac{n}{2},l-2j+\frac{p-2}{2}}\|}$  intertwines in the stratified model the restriction of $\pi_\lambda^{(n)}$ to $SO_0(2,n-p)\times SO(p)$, and the representation $\pi^{(n-1)}_{\lambda+l}\otimes \eta_{l-2j}$ acting on $L^2_{\lambda+l}\hat{\otimes}\mathcal{H}^p_{l-2j}$.
\end{coro}

\begin{proof}
Recall from the definition of the space $W_{l,j}=P_j^{\lambda-\frac{1}{2},l-2j+\frac{p-2}{2}}(2\|v\|^2-1)\mathcal{H}^p_{l-2j}$ introduced in the  proof of Theorem \ref{thm:BranchinSO(2,n)SO(2,n-p)SO(p)} that it admits the orthogonal basis $\{P_{j,k}\}$ described in \eqref{eq:OrthogonalBasisWj}, thus the reproducing kernel $K'_j$ of $W_{l,j}$ is given by
\[K'_j(x,y)=\sum_k \frac{P_{j,k}(x)P_{j,k}(y)}{\|P_{j,k}\|_\lambda^2},\]
where $\|\cdot\|_\lambda$ denotes the $L^2$ norm on $L^2(\unitball{p},(1-\|v\|^2)^{\lambda-\frac{n}{2}}dv)$.
If $\{Y_k\}$ denotes an orthogonal basis of $\mathcal{H}^p_{l-2j}$, we have 
\begin{equation}\label{eq:normEquality}
 \|P_{j,k}\|_\lambda^2=\frac{1}{2^{l-2j+\lambda+\frac{p-n}{2}+1}}\|P_j^{\lambda-\frac{n}{2},l-2j+\frac{p-2}{2}}\|^2\|Y_k\|_\mathcal{H}^2.
 \end{equation}
Indeed, using the change of variables $v=rv'$ with $r=\|v\|$ and $v'=\frac{v}{\|v\|}$ we get
\begin{align*}
\|P_{j,k}\|_\lambda^2&=\int_{\unitball{p}}|P_{j,k}(v)|^2(1-\|v\|^2)^{\lambda-\frac{n}{2}}~dv\\
&=\int_0^1|P_j^{\lambda-\frac{n}{2},l-2j+\frac{p-2}{2}}(2r^2-1)|^2r^{p-1}(1-r^2)^{\lambda-\frac{n}{2}}\int_{S^{p-1}} |Y_k(rv')|^2~d\omega(v')~dr\\
&=\int_0^1|P_j^{\lambda-\frac{n}{2},l-2j+\frac{p-2}{2}}(2r^2-1)|^2r^{2(l-2j)+p-1}(1-r^2)^{\lambda-\frac{n}{2}}~dr \|Y_k\|_\mathcal{H}^2\\
&=\frac{1}{2^{l-2j+\lambda+\frac{p-n}{2}+1}}\int_{-1}^1|P_j^{\lambda-\frac{n}{2},l-2j+\frac{p-2}{2}}(v)|^2(1+v)^{2(l-2j)+p-1}(1-v)^{\lambda-\frac{n}{2}}~dr \|Y_k\|_\mathcal{H}^2.
\end{align*}
Thus, we have:
\begin{align*}
K'_j(x,y)&= 2^{l-2j+\lambda+\frac{d-n}{2}+1}\frac{P_j^{\lambda-\frac{n}{2},l-2j+\frac{p-2}{2}}(2|x|^2-1)P_j^{\lambda-\frac{n}{2},l-2j+\frac{p-2}{2}}(2|y|^2-1)}{\|P_j^{\lambda-\frac{n}{2},l-2j+\frac{p-2}{2}}\|^2}\sum_k \frac{Y_k(x)Y_k(y)}{\|Y_k\|_\mathcal{H}^2}.\\
&=2^{l-2j+\lambda+\frac{p-n}{2}+1}\frac{P_j^{\lambda-\frac{n}{2},l-2j+\frac{p-2}{2}}(2|x|^2-1)P_j^{\lambda-\frac{n}{2},l-2j+\frac{p-2}{2}}(2|y|^2-1)}{\|P_j^{\lambda-\frac{n}{2},l-2j+\frac{p-2}{2}}\|^2}K_{l-2j}(x,y).
\end{align*}
Moreover, equation \eqref{eq:normEquality} proves that the map defined for $Y\in \mathcal{H}^p_{l-2j}$ by
 \[Y\mapsto \frac{2^{\frac{1}{2}(l-2j+\lambda+\frac{p-n}{2}+1)}}{\|P_j^{\lambda-\frac{n}{2},l-2j+\frac{p-2}{2}}\|}P_j^{\lambda-\frac{n}{2},l-2j+\frac{p-2}{2}}(2|v|^2-1)Y(v)\]
is an isometry onto $W_{l,j}$. Furthermore, it is an intertwining operator for the $SO(p)$ action. Thus, in view of Theorem \ref{thm:SBO_SO(2,n)SO(2n-p)}, the operator 
\[\Psi_{l,j}f(x,u)=\gamma Q_{1,n-p-1}(x)^{-\frac{l}{2}}\int_{\unitball{p}} f(x,v)P_j^{\lambda-\frac{n}{2},l-2j+\frac{p-2}{2}}(2|v|^2-1)K_{l-2j}(u,v)~d\mu_\alpha(v)\]
with $\gamma=\frac{2^{\frac{1}{2}(l-2j+\lambda+\frac{p-n}{2}+1)}}{\|P_j^{\lambda-\frac{n}{2},l-2j+\frac{p-2}{2}}\|}$ intertwines the representation $\pi_\lambda^{(n)}$ acting on $L^2_\lambda(\Omega_{n-p})\hat{\otimes}L^2(\unitball{p},\linebreak(1-\|v\|^2)^{\lambda-\frac{n}{2}}dv)$ restricted to $SO_0(2,n-p)\times SO(p)$ and the representation $\pi_{\lambda+l}^{(n-1)}\otimes \eta_{l-2j}$ acting on $L^2_{\lambda+l}(\Omega_{n-p})\otimes \mathcal{H}^p_{l-2j}$.
\end{proof}

A direct consequence of this proof is an explicit formula for the associated holographic operator.

\begin{coro}\label{coro:HolographicSO(2,n)SO(2,n-p)SO(p)}
Let $\lambda\in \N$ such that $\lambda>n-1$. The operator $\Phi_{l,j}$ defined on $L^2_{\lambda+l}(\Omega_{n-p})\hat{\otimes}\mathcal{H}^p_{l-2j}$ by 
\begin{equation}
\Phi_{l,j}f(t,u)=Q_{1,n-p-1}(t)^{\frac{l}{2}}P_j^{\lambda-\frac{n}{2},l-2j+\frac{p-2}{2}}(2\|u\|^2-1)f(t,u)
\end{equation}
is a holographic operator between the representation $\pi^{(n-1)}_{\lambda+l}\otimes \eta_{l-2j}$ acting on $L^2_{\lambda+l}(\Omega_{n-p})\hat{\otimes}\mathcal{H}^p_{l-2j}$ and the representation $\pi_\lambda^{(n)}$ in the stratified model.
\end{coro}

\section{Appendix: Orthogonal polynomials}

The goal of this appendix is to recall some basic results about generalized hypergeometric functions and orthogonal polynomials which are used in this work. Moreover, we prove an integral representation for the Kummer function needed in in the proof of Theorem \ref{thm:Fomule_adjoint_RC_relative_reproducing}. Most of the result described in this section can be found in textbooks about special functions, for example \cite{Beals_Wong} and \cite{AndAsk} are good references on this topic.

Recall the definition for the Pochammer symbol, for $\lambda \in \C$ and $n\in\N$:
\[(\lambda)_n=\lambda(\lambda+1)\cdots(\lambda+n-1)=\frac{\Gamma(\lambda+n)}{\Gamma(\lambda)}.\]

\subsection{Generalized hypergeometric functions}\label{sec:hypergeo}

The generalized hypergeometric functions are defined for $p,\ q\in\N$, $\alpha_1,\cdots,\alpha_p\in \C$ and $\beta_1,\cdots,\beta_q\in \C$ such that non of the $\beta_i$ is a negative integer.
\begin{equation}\label{def:HypergeometricFunction}
_pF_q(\alpha_1,\cdots,\alpha_p;\beta_1,\cdots,\beta_q;z)=\sum_{n\geq 0} \frac{(\alpha_1)_n\cdots(\alpha_p)_n}{(\beta_1)_n\cdots(\beta_q)_n}\frac{z^n}{n!}.
\end{equation}
If one of the $\alpha_i$ is a negative integer then \eqref{def:HypergeometricFunction}
is a polynomial.

Generalized hypergeometric functions are solutions of the generalized hypergeometric equation:
\begin{equation}
x\prod_{n=1}^p(x\frac{d}{dx}+\alpha_n)F(x)-x\frac{d}{dx}\prod_{n=1}^p(x\frac{d}{dx}+\beta_n-1)F(x),
\end{equation}
with initial condition $F(0)=1$.

The case $_1F_1$ is also called the Kummer function and is defined, for complex parameter $a,b$ where $b$ is not a negative integer, by 
\begin{equation}\label{def:KummerFunction}
_1F_1(a,b;x)=\sum_{n=0}^\infty\frac{(a)_n}{(b)_n}\frac{x^n}{n!}.
\end{equation}

\subsection{Jacobi polynomials}\label{sec:JacobiPolynomials}

Let $n\in \N$ and $\alpha,~\beta\in \C$. The family of Jacobi polynomials  $P_n^{\alpha,\beta}$ is defined as a polynomial solution of the following second order differential equation
\begin{equation}\label{eq:equaDiffJacobi}
(1-t^2)\frac{d^2f}{dt^2}+(\beta-\alpha -(\alpha+\beta +2)t) \frac{df}{dt}+n(n+\alpha+\beta+1)f=0.
\end{equation} 

They are normalized so that they verify the so-called Rodrigue's formula:
\begin{equation}\label{eq:RodrigueFormulaJacobi}
(1-t)^\alpha(1+t)^\beta P_n^{\alpha,\beta}(t)=\frac{(-1)^n}{2^n n!}\left( \frac{d}{dt}\right)^n\left((1-t)^{\alpha+n}(1+t)^{\beta+n}\right).	
\end{equation}

Jacobi polynomials can be expressed using the hypergeometric function $_2F_1$:
\begin{equation}\label{eq:JacobiHypergeometric}
P_n^{\alpha,\beta}(t)=\frac{(\alpha+1)_n}{n!}\left(\frac{1+x}{2}\right)^n~_2F_1\left(-n,-n-\beta;\alpha+1;\frac{x-1}{x+1}\right).
\end{equation}

A quick computation shows that this gives the following expression for the Jacobi polynomials.
\begin{equation}\label{eq:Jacobi_def}
P_n^{\alpha,\beta}(t)=\frac{1}{2^n}\sum_{s=0}^n\binom{n+\alpha}{n-s}\binom{n+\beta}{s}(t-1)^s(t+1)^{n-s}.
\end{equation}

For real parameters $\alpha>-1$ and $\beta>-1$, they form an orthogonal basis of the Hilbert space $L^2((-1;1),(1-v)^\alpha(1+v)^\beta~dv)$, and their norm is given by:
\begin{equation}
\int_{-1}^1\left|P_n^{\alpha,\beta}(v)\right|^2(1-v)^\alpha(1+v)^\beta~dv=\frac{2^{\alpha+\beta+1}\Gamma(n+\alpha+1)\Gamma(n+\beta+1)}{(2n+\alpha+\beta+1)\Gamma(n+\alpha+\beta+1)n!}.
\end{equation}

The orthogonal projection $\mathcal{J}_n^{\alpha,\beta}$ on the polynomial $P_n^{\alpha,\beta}$ of a function $f\in L^2((-1;1),(1-v)^\alpha(1+v)^\beta~dv)$ is called the Jacobi transform:
\begin{equation}
\mathcal{J}_n^{\alpha,\beta}f=\int_{-1}^1f(v)P_n^{\alpha,\beta}(v)(1-v)^\alpha (1+v)^\beta~dv.
\end{equation}
\subsection{Gegenbauer polynomials}\label{sec:gegenbauer}

Gegenbauer polynomials are special cases of the Jacobi polynomials. They are defined, for $\alpha\in \C$ and $n\in \N$, by 
\begin{equation}\label{def:GegenbauerPolynomials}
C_n^\alpha(x)=\frac{(2\alpha)_n}{(\alpha+\frac{1}{2})_n}P_n^{\alpha-\frac{1}{2},\alpha-\frac{1}{2}}(x).
\end{equation}

We have $C_n^\alpha(1)=\frac{(2\alpha)_n}{n!}$, and the Rodrigue's formula is given by
\begin{equation}\label{eq:RodriguesGegenbauer}
C_n^\alpha (x) (1-x^2)^{\alpha-\frac{1}{2}}=\frac{(-1)^n(2\alpha)_n}{2^n n! (\alpha+\frac{1}{2})_n} \left( \frac{d}{dx}\right)^n (1-x^2)^{n+\alpha-\frac{1}{2}}.
\end{equation}

Finally, for $\alpha>-\frac{1}{2}$, they form an orthogonal basis of the Hilbert space $L^2((-1;1),(1-v^2)^{\alpha-\frac{1}{2}}~dv)$, with norm 
\begin{equation}
\int_{-1}^1\left|C_n^\alpha\right|^2(1-v^2)^{\alpha-\frac{1}{2}}~dv=\frac{\pi2^{1-2\alpha}\Gamma(2\alpha+n)}{n!(n+\alpha)\Gamma(\alpha)^2}.
\end{equation}

\subsection{An integral representation for the Kummer function}

We recall that the Kummer function admits the following integral representation in terms of Jacobi polynomials (see \cite{Labriet20}, Lemma 5.1):
\begin{lemme}\label{lem:Fourier_poids_Kummer}
Let $\alpha,\beta>1$, $l\in\N$ and $x\in \R^+$. Then:
\begin{equation}\label{eq:integral_representation_Kummer}
C_{\alpha,\beta}x^{l}e^{ix}~_1F_1(\alpha+l,\alpha+\beta+2l;-2ix)=\int_{-1}^1P_l^{\alpha-1,\beta-1}(v)e^{ivx}(1-v)^{\alpha-1}(1+v)^{\beta-1}~dv,
\end{equation}
where $C_{\alpha,\beta}=\frac{2^{\alpha+\beta+l-1}i^l}{l!}B(\alpha+l,\beta+l)$, and $B(x,y)$ is the usual Euler beta function.
\end{lemme}

In the case $\alpha=\beta=\nu-\frac{1}{2}$, we get the following integral representation:

\begin{coro}\label{lem:Fourier_poids_conforme}
For $\nu>1$, $l \in \N$ and $z\in \C$:
\begin{equation}
\int_{-1}^1 C_l^\nu (v) (1-v^2)^{\nu-\frac{1}{2}}e^{ixv}~dv=c(l;\nu)z^l~_0F_1(l+\nu+1;-\frac{x^2}{4}),
\end{equation}
where $c(l;\nu)=\frac{i^l\sqrt{\pi}(2\nu)_l\Gamma(\nu+\frac{1}{2})}{2^l l! \Gamma(l+\nu+1)}$.
\end{coro}
\begin{proof}
This result is a direct consequence of Lemma \ref{lem:Fourier_poids_Kummer}, and the following identity about Kummer functions (\cite{Beals_Wong}, eq.(6.1.11)):
\begin{equation}
_1F_1(a,2a;4x)=e^{2x}{}_0F_1\left(a+\frac{1}{2};x^2\right).
\end{equation}
\end{proof}

\bibliographystyle{alpha}
\bibliography{./Biblio}
\end{document}